\documentclass[10pt,twoside]{article}
\usepackage{amsfonts}
\usepackage{fancyhdr}
\usepackage{titlesec}
\usepackage{cite}
\usepackage{ifthen}
\usepackage{pifont}
\usepackage{stmaryrd}
\usepackage{setspace}
\usepackage{indentfirst}
\usepackage{amsmath,amssymb,amscd,bbm,amsthm,mathrsfs,dsfont}

\usepackage{float}
\usepackage{subfig}
\usepackage{overpic}
\usepackage{graphicx}
\usepackage{amsthm}
\newtheorem{theorem}{Theorem}[section]
\newtheorem{remark}{Remark}[section]
\newtheorem{corollary}{Corollary}[section]

\usepackage{booktabs}
\usepackage{multirow}
\usepackage{url}
\usepackage{threeparttable}

\input amssym.def

\titleformat{\section}{\centering\large\bfseries}{\S\arabic{section}}{1em}{}
\newboolean{first}
\setboolean{first}{true}

\begin{document}

\setlength\abovedisplayskip{2pt}
\setlength\abovedisplayshortskip{0pt}
\setlength\belowdisplayskip{2pt}
\setlength\belowdisplayshortskip{0pt}

\title{\bf \Large Revised BDS Test\author{Wenya Luo, Zhidong Bai, Jiang Hu, Chen Wang }\date{}} \maketitle%作者姓名,标题
 \footnote{Received: .}%%%%%%%%% 请输入投稿日期
 \footnote{MR Subject Classification: Primary 62G10; Secondary 62E20.}%%%%%%%%% 请输入相应的MR主题分类号
%%%%%%%%%%%%%%%%%关键词至少三个及以上
 \footnote{Keywords: independence, correlation integral, over-rejection, BDS test.}
 \footnote{Digital Object Identifier(DOI): 10.1007/s11766-016-****-*.}
 \footnote{Supported by the National Natural Science Foundation of China\,(********).}
\begin{center}
\begin{minipage}{135mm}

{\bf \small Abstract}.\hskip 2mm {\small
In this paper, we focus on the BDS test, which is  a nonparametric test of independence. Specifically, the null hypothesis $H_{0}$ of it is that $\{u_{t}\}$ is i.i.d. (independent and identically distributed), where $\{u_{t}\}$ is a random sequence. The BDS test is widely used in economics and finance, but it has a weakness that cannot be ignored: over-rejecting $H_{0}$ even if the length $T$ of $\{u_{t}\}$ is as large as $(100,2000)$. To improve the over-rejection problem of BDS test, considering that the correlation integral is the foundation of BDS test, we not only accurately describe the expectation of the correlation integral under $H_{0}$, but also calculate all terms of the asymptotic variance of the correlation integral whose order is $O(T^{-1})$ and $O(T^{-2})$, which is essential to improve the finite sample performance of BDS test. Based on this, we propose a revised BDS (RBDS) test and prove its asymptotic normality under $H_{0}$. The RBDS test not only inherits all the advantages of the BDS test, but also effectively corrects the over-rejection problem of the BDS test, which can be fully confirmed by the simulation results we presented. Moreover, based on the simulation results, we find that similar to BDS test, RBDS test would also be affected by the parameter estimations of the ARCH-type model, resulting in size distortion, but this phenomenon can be alleviated by the logarithmic transformation preprocessing of the estimate residuals of the model. Besides, through some actual datasets that have been demonstrated to fit well with ARCH-type models, we also compared the performance of BDS test and RBDS test in evaluating the goodness-of-fit of the model in empirical problem, and the results reflect that, under the same condition, the performance of the RBDS test is more encouraging.}

\end{minipage}
\end{center}

\thispagestyle{fancyplain} \fancyhead{}
\fancyhead[L]{\textit{Appl. Math. J. Chinese Univ.}\\
2017, 32(*): ***-***} \fancyfoot{} \vskip 10mm%\setcounter{page}{112}

\section{Introduction}
Independence has always been highly concerned in econometrics, finance, time series analysis and statistics, due to the fact that many problems boil down to testing independence hypothesis. Thus, many independence tests are constructed, such as Skaug and Tj{\o}stheim (1993), Delgado (1996), Hong (1998), Matilla-Garc{\' \i} and Mar{\' \i}n (2008) and so on. Among them, there is a widespread nonparametric independent test --- the BDS test, which was first proposed by Brock et al. (1987) and its theory was elaborated in Broock et al. (1996). The BDS test has some desirable properties. First, same as other nonparametric tests, the BDS test can be applied to the random sequence $\{u_{t}\}$ without knowing more information about it. Second, the BDS test still performs well in the absence of higher moments of $\{u_{t}\}$. De Lima (1997) investigated the robustness of some tests, including the BDS test, when they suffer from moment condition failure and found that all the tests considered in his study require $\{u_{t}\}$ to have at least finite fourth moment except the BDS test. This property makes the BDS test popular in economics and finance because of the fact that the fourth moment of the time series in economics and finance is generally not finite (Jansen and De Vires (1991) and Loretan and Phillips (1994)). Third, there's a fast algorithm of the BDS test given by LeBaron (1997), which make it easy to implement by simply invoking the function in R software. The popularity of the BDS test is also inseparable from this. Moreover, Belaire-Franch and Contreras (2002) discussed and compared the algorithm of the BDS test in available softwares.
All the above advantages make the BDS test widely used in economics and finance. Specifically, its applications can be divided into two categories.
One is to detect whether there is non-linear structure in the data. Generally speaking, the BDS test is often used to conduct preliminary research on the data, which is helpful for model identification, for example, Hsieh (1991) used the BDS test to capture the possible nonlinear structure of the stock market; Madhavan (2013) used the BDS test to detect the nonlinearity in US and European Investment Grade Credit Default Swap Indices; Akintunde et al. (2015) detected the nonlinearity of the commercial bank savings in Nigeria using the BDS test. 
Another important application is that the BDS test can be used as a model diagnostic tool. Specifically, considering a model below:
\begin{equation}\label{model1.1}
	y_{t}=f(x_{t},b,\delta_{t}),\ \delta_{t}\stackrel{\mathrm{i.i.d.}}{\sim}N(0,1),
\end{equation}
where $\{y_{t}\}$ and $\{x_{t}\}$ are two observable time series, $b$ is the unknown parameter to be estimated consistently, and $\{\delta_{t}\}$ is a sequence of i.i.d. variables and is independent of $\{x_{t}\}$. If the fitted model \eqref{model1.1} is correctly specified, the residuals $\hat{\delta}_{t}=g(y_{t},x_{t},\hat{b})$ should pass the BDS test, otherwise it indicates that the fitted model is misspecified. There are an  impressive body of literature using the BDS test as a tool of model selection, such as, Chen and Kuan (2002), Brock and Durlauf (2007), Racine and Maasoumi (2007). Besides, some researchers studied the effects of the residuals from different fitted models on the performance of the BDS test, for example, Broocks and Heravi (1999), Lai (2000), Caporale et al. (2005), Fernandes and Preumont (2012). 

Considering that the BDS test has so many attractive advantages, many researchers compared independence test they proposed with the BDS test, for example,  Lee et al. (1993) compared the neural network methods with some alternative tests containing the BDS test; Pinkse (1998)  compared their nonparametric test for serial independence with the BDS test; Diks and Panchenko (2007) proposed a new test using kernel-based quadratic forms and compared it with the BDS test; C{\'{a}}novas et al. (2013) obtained an independence test based on permutation and compared it with the BDS test. Similar researches includes: Hui et al. (2017), Hjellvik and Tj{\o}stheim (1996), Granger et al. (2004) etc.
In addition, inspired by the BDS test, some researchers developed new test, for example: Beak and Brock (1992) put forward the vector version of the BDS test to meet the need of multivariate time series model; Genest et al. (2007) propose a ranked-based extension of the BDS test.

Along with its strengths, the BDS test was found to have some drawbacks. First, the BDS test involves several parameters that need to be set manually, specifically, embedding dimension $m$, dimensional distance $\epsilon$, and delay times $\varsigma$. Among them, $\varsigma$ is always set as 1 and more literatures on $\varsigma$ can refer to Matilla-Garc{\'\i}a et al. (2004a), Matilla-Garc{\'\i}a et al. (2004b), Matilla-Garc{\'\i}a et al. (2005) and Matilla-Garc{\'\i}a and Mar{\'\i}n (2010). Besides, Kanzler (1999) discovere that the behavior of the BDS test is sensitive to the choice of the embedding dimension $m$, dimension distance $\epsilon$, and the length $T$ of the random sequence $\{u_{t}\}$. Broock et al. (1996) suggests that $m$ should be chosen as any integer in $[2,5]$ and $\epsilon$ should be set as $0.5\sqrt{\mathbf{Var}(u_{t})}$ when the length $T$ of $\{u_{t}\}$ is 200 or larger. To weaken the influence of $\epsilon$, Ko{\v{c}}enda (2001)  study the ratio of $\log{C_{m,T}}$ to $\log{\epsilon}$ and Ko{\v{c}}enda (2005)  calculate the results of this ratio when $\epsilon$ takes different values and give the optimal choice. 

Another drawback of the BDS test is over-rejection problem, which is the concern of our research. Studies show that although $\epsilon$ and $m$ are well set following the advice given by Broock et al. (1996), the BDS test has a shortcoming of  over-rejecting $H_{0}$ even when the length $T$ of $\{u_{t}\}$ is as large as several hundred or even two thousand. This phenomenon is gradually weakened when $T>2000$, until $T\geq{3000}$, become negligible. Therefore, it's necessary to improve this drawback of the BDS test, which is the target of our research. Since the BDS test originates in a chaos theory and is based on the correlation integral $C_{m,T}(\epsilon)$ in Procacia et al. (1983), to explore the reason for this problem, we studied the theory given by Broock et al. (1996) and found that the mistake of treating $C_{m,T}(\epsilon)$ as the U-statistic is the root reason of this problem, since $C_{m,T}(\epsilon)$ doesn't satisfy the definition of U-statistic, which will be explained in detail in section 2.
%Based on this, we gave the revised BDS , which effectively alleviated the over-rejection problem.

In this paper, without basing on the theory of U-statistic, we precisely depict the expectation of the correlation integral $C_{m,T}(\epsilon)$ under  $H_{0}$, and calculate all terms with order $O(T^{-1})$ and $O(T^{-2})$ in the asymptotic variance of $C_{m,T}(\epsilon)$. So the asymptotic variance given here is more accurate than that given by Broock et al. (1996), and our study shows that terms with order $O(T^{-2})$ in the asymptotic variance can not be ignored to improve the finite sample performance of the statistic. Based on the new asymptotic theory of $C_{m,T}(\epsilon)$, we present a revised BDS (RBDS) test and proof its asymptotic normality under $H_{0}$. The RBDS test is still a nonparametric test for independence based on the correlation integral, which inherits all the advantages of the BDS test, more importantly, it effectively improve the over-rejection problem of the BDS test. In addition, we design some simulation experiments to compare the finite sample behavior of the RBDS test and the BDS test, and the results confirm that the RBDS test gets rid of the over-rejection problem even when $T$ is small. In addition, several experiments are designed to compare the different performances of BDS test and RBDS test when they are subjected to parameter estimation of ARCH-type model. The results show that, similar to BDS test, RBDS test also has size distortion phenomenon due to the influence of model parameter estimations, but the logarithmic transformation pre-processing of the estimate residual sequence can effectively improve the distortion problem. Besides, we apply RBDS test and BDS test to some real datasets that have been demonstrated to fit well with ARCH-type model, to compare their performance as model diagnostic tools in practice. And the results reflect, under the same condition, the performance of RBDS test is more encouraging.

The rest of this paper are organized as follows: In section 2, we introduce the notations that will be encountered in this paper, review the theory of the BDS test briefly and discuss why the correlation integral $C_{m,T}(\epsilon)$ is not a U-statistic; In section 3, we introduce the exact expectation of the correlation integral $C_{m,T}(\epsilon)$ and the modified asymptotic limit theory of $C_{m,T}(\epsilon)$. The RBDS test and the CLT (central limit theorem) of it are presented in section 4. The results of comparing the empirical sizes and empirical powers of the two tests and analyzing the performances of the two tests affected by the model estimation parameters are presented in section 5. The differences between BDS test and RBDS test in real datasets are displayed in section 6. Section 7 gives some conclusions. Some proofs are relegated to the appendixes.

\section{The BDS test}
\subsection{Notation}
Throughout the paper, we use $\{u_{t}\}$ to denote a random sequence of length $T$. $m$ stands for an integer belonging to $\mathbb{Z}^+$ and $\epsilon$ denotes a constant satisfying $\epsilon>0$. For fixed $m$, we use $\{Y_{l}^m\}$ to denote an m-dimensional random vector sequence of length $T_{m}=T-m+1$, where $
Y_{l}^m=\left(u_{l},u_{l+1},\cdots,u_{l+m-1}\right)^{'}$. Besides, notation
$\|\cdot\|$ denotes the maximum norm for a vector and $\lfloor{\cdot}\rfloor$ denotes the round down function. In addition, we use $I_{\epsilon}(\cdot)$ to represent an indicator function, defined as follows:
\begin{equation}\label{indicator}
	I_{\epsilon}(x)=\left\{
	\begin{array}{rcl}
		1 & & \ x<{\epsilon},\\
		0 & & \   \text{else.}
	\end{array} \right.\\
\end{equation}
Weak convergence is denoted by $\stackrel{D}{\longrightarrow}$. And the expectation and variance are denoted by $\mathbf{E}(\cdot)$ and  $\mathbf{Var}(\cdot)$ respectively. Moreover, we introduce the following two symbols to represent the two quantities for a given $x$:
\begin{equation}\label{C_M}
C_{T,x}=\frac{(T-1-x)!}{T!}, M_x=\frac{(T_m-m+1)!}{(T_m-m+1-x)!}.
\end{equation}
In addition, we use
$\omega_l^r$, $\eta_l^{l-1}$ and $\xi_l^{\kappa}$ to represent the probability defined in \eqref{omegalr}, \eqref{eta} and \eqref{xi} respectively and use $\widehat{\omega}_l^r$, $\widehat{\eta}_l^{l-1}$ and $\widehat{\xi}_l^{\kappa}$ to represent their consistent estimations based on U-statistic, respectively. For instance, when $r=0$, $\omega_{l}^{0}=P(|u_{t}-u_{t+1}|<\epsilon,|u_{t+1}-u_{t+2}|<\epsilon,..., |u_{t+l-1}-u_{t+l}|<\epsilon)$, its consistent estimation is as follows:
\begin{equation}\label{estomegal0}
	\widehat{\omega}^0_{l}=C_{T,l}\sum \limits_{t_{0},t_{1},\cdots,t_{l} \atop \mbox{\tiny{distrinct}}} {\prod_{\rho=0}^{l-1}I_{\epsilon}(|u_{t_{\rho}}-u_{t_{\rho+1}}|)},\quad C_{T,l}=\frac{(T-1-l)!}{T!}.
\end{equation}
To better understand, we associate $\omega_l^r$, $\eta_l^{l-1}$ and $\xi_l^{\kappa}$ with graphics shown in Figure \ref{figureomegalr}, Figure \ref{figureeta} and Figure \ref{figurexi} respectively.
\begin{equation}\label{omegalr}
	\begin{aligned}
		\omega_{l}^{r}&=
		P\left(\left\{\mbox{$\begin{array}{ll}
				|u_{t}-u_{t+1}|<{\epsilon},&\\
				|u_{t+1}-u_{t+2}|<{\epsilon}, &\\
				\qquad \cdots & \\
				|u_{t+l-r-1}-u_{t+l-r}|<{\epsilon}, & \\
				|u_{t+l-r}-u_{t+l-r+1}|<{\epsilon}, & |u_{t+l-r+1}-u_{\alpha}|<{\epsilon}, \\
				|u_{t+l-r+1}-u_{t+l-r+2}|<{\epsilon}, & |u_{t+l-r+2}-u_{\alpha+1}|<{\epsilon},\\
				\qquad \cdots & \\
				|u_{t+l-1}-u_{t+l}|<{\epsilon}, & |u_{t+l}-u_{\alpha+r-1}|<{\epsilon},  \\
			\end{array}$}\right\}\right)\\
		&=P\left(\left\{\mbox{$\begin{array}{ll}
				|u_{\alpha}-u_{t}|<{\epsilon}, &|u_{t}-u_{t+1}|<{\epsilon}, \\
				|u_{\alpha+1}-u_{t+1}|<{\epsilon}, &|u_{t+1}-u_{t+2}|<{\epsilon}, \\
				\qquad \cdots & \\
				|u_{\alpha+r-2}-u_{t+r-2}|<{\epsilon}, & |u_{t+r-2}-u_{t+r-1}|<{\epsilon},  \\
				|u_{\alpha+r-1}-u_{t+r-1}|<{\epsilon}, & |u_{t+r-1}-u_{t+r}|<{\epsilon},  \\
				|u_{t+r}-u_{t+r+1}|<{\epsilon}, & \\
				\qquad \cdots & \\
				|u_{t+l-1}-u_{t+l}|<{\epsilon}, & \\
			\end{array}$}\right\}\right),
	\end{aligned}
\end{equation}

where $\left\{(\alpha,\alpha_{1},\cdots,\alpha+r-1) \in{{\mathbb{D}}_{1}^{c}}\right\},\ {\mathbb{D}_{1}}=\{t,t+1,\cdots,t+l\}.$ 

\begin{equation}\label{eta}
	\begin{aligned}
		\eta_{l}^{l-1}=P\left(\left\{\mbox{$\begin{array}{ll}
				|u_{t}-u_{t+1}|<{\epsilon}, & \\
				|u_{t+1}-u_{t+2}|<{\epsilon}, &|u_{t+1}-u_{\beta}|<{\epsilon}, \\
				|u_{t+2}-u_{t+3}|<{\epsilon}, &|u_{t+2}-u_{\beta+1}|<{\epsilon}, \\
				\qquad \cdots & \\
				|u_{t+l-1}-u_{t+l}|<{\epsilon}, & |u_{t+l-1}-u_{\beta+l-2}|<{\epsilon}, \\
				|u_{t+l}-u_{t+l+1}|<{\epsilon} &
			\end{array}$}\right\}\right),
	\end{aligned}
\end{equation}

where $\left\{(\beta,\beta+1,\cdots,\beta+l-2)\in{\mathbb{D}_{2}^c}\right\}$, $\mathbb{D}_{2}=\{t,t+1,\cdots,t+l+1\}$. 

\begin{equation}\label{xi}
	\begin{aligned}
		\xi_{l}^{\kappa}&=
		P\left(\left\{\mbox{$\begin{array}{ll}
				|u_{t}-u_{t+1}|<{\epsilon},& \\
				|u_{t+1}-u_{t+2}|<{\epsilon}, & \\
				\qquad \cdots & \\
				|u_{t+\kappa-1}-u_{t+\kappa}|<{\epsilon}, & \\
				|u_{t+\kappa}-u_{t+\kappa+1}|<{\epsilon}, & |u_{t+\kappa}-u_{\gamma}|<{\epsilon}, \\
				|u_{t+\kappa+1}-u_{t+\kappa+2}|<{\epsilon}, & \\
				\qquad  \cdots & \\
				|u_{t+l-1}-u_{t+l}|<{\epsilon},  &  \\
			\end{array}$}\right\}\right),
	\end{aligned}
\end{equation}

where $\gamma\in{\mathbb{D}_{1}^c}$, $\mathbb{D}_{1}=\{t,t+1,\cdots,t+l\}$. 

\begin{figure}[H]
	\centering
	\includegraphics[width=0.7\textwidth]{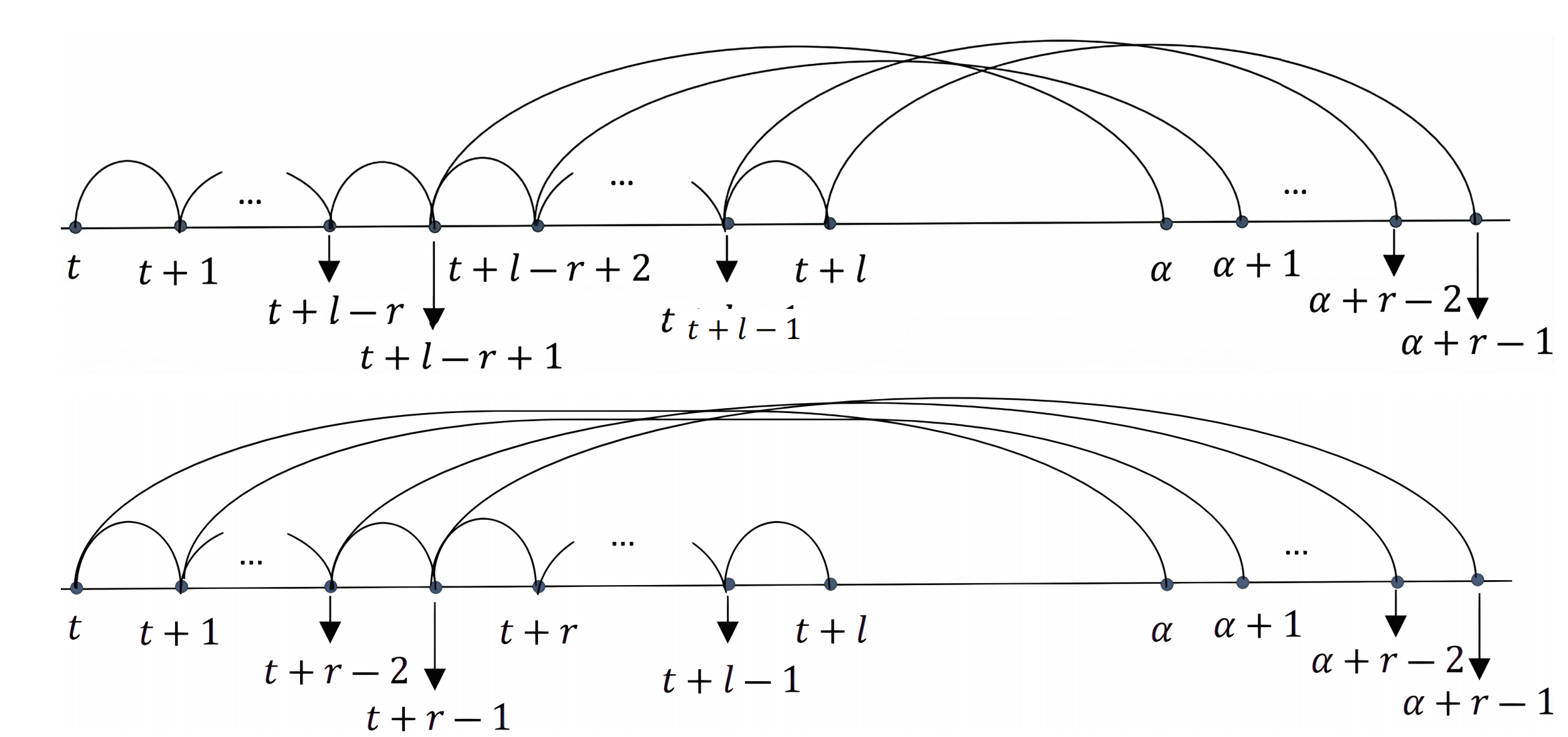}
	\caption{Graphics for $\omega_{l}^{r}$.}
	\label{figureomegalr}\vspace{-0.75cm}
\end{figure}
\begin{figure}[H]
	\centering
	\includegraphics[width=0.7\textwidth]{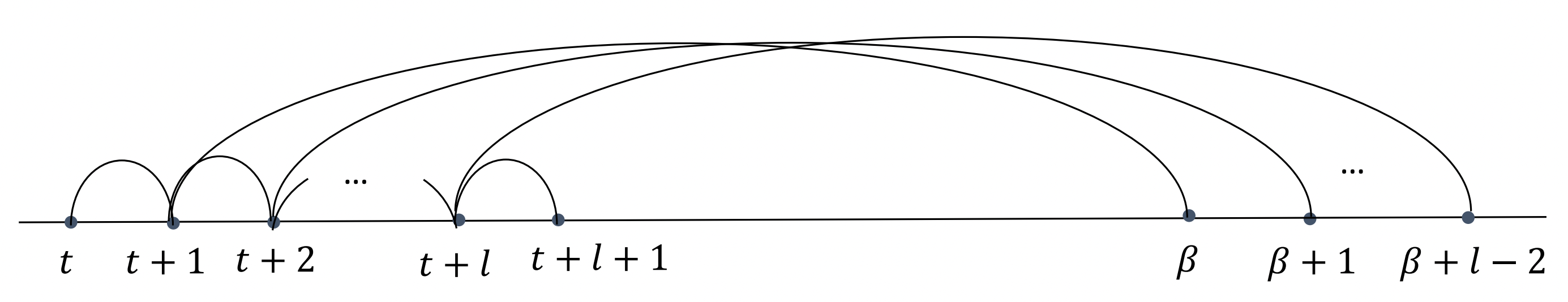}
	\caption{Graph for $\eta_{l}^{l-1}$.}
	\label{figureeta}\vspace{-0.75cm}
\end{figure}
\begin{figure}[H]
	\centering
	\includegraphics[width=0.7\textwidth]{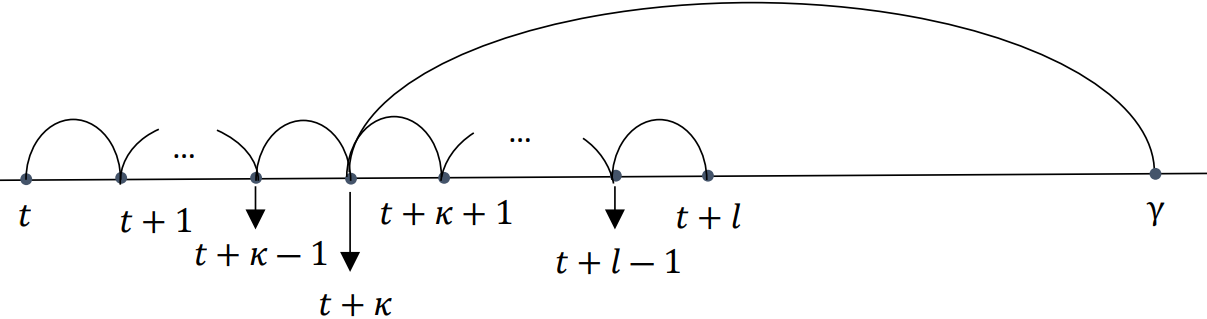}
	\caption{Graph for $\xi_{l}^{\kappa}$.}
	\label{figurexi}
\end{figure}

\subsection{The BDS test}

The BDS test proposed by Brock et al. (1987) is a nonparametric test used to detect $H_{0}:\{u_{t}\}$ is i.i.d.. Its theory is given in detail by Broock et al. (1996), specifically,
\begin{equation}\label{BDStest}
	W_{m,T}(\epsilon) = \sqrt{T}\frac{\big(C_{m,T}(\epsilon)-\left(\widehat{\omega}_{1}^0\right)^m\big)} {V_{m,T}(\epsilon)}\stackrel{D}{\underset{under H_{0}}{\longrightarrow}}{N(0,1)},
\end{equation}
where $\epsilon$ and $m$ are two preset parameters under $\epsilon>0$ and $m\in{\mathbb{Z}^+}$, $m$ is called the embedding dimension,
\begin{equation}\label{CIdef}
	C_{m,T}(\epsilon)=\frac{1}{N} \sum\limits_{1\leq{t}<{s}\leq{T_m}}{I_{\epsilon}(\|{Y_{t}^{m}-Y_{s}^{m}}\|)},\ N=\binom{T_m}{2},T_{m}=T-m+1,
\end{equation}
is the correlation integral of $\{u_{t}\}$ introduced by Procacia et al. (1983), 
\begin{equation}\label{hatomega10}
	\widehat \omega_1^0=\frac{1}{T(T-1)}\sum_{t\ne s}^TI_{\epsilon}(|u_{t}-u_{s}|),
\end{equation}
\begin{equation}\label{hatomega20}
	\widehat\omega_2^0=\frac{1}{T(T-1)(T-2)}\sum\limits_{t,s,r \mbox{\tiny{distinct}}}^T{I_{\epsilon}(|u_{t}-u_{s}|) I_{\epsilon}(|u_{s}-u_{r}|)},
\end{equation}
\begin{equation}\label{BDSVar}
	\begin{aligned}
		&V_{m,T}^2(\epsilon)=4m(m-2)\left(\widehat{\omega}_{1}^0\right)^{2m-2} \left[\left(\widehat{\omega}_{2}^0\right) -\left(\widehat{\omega}_{1}^0\right)^2\right] +\left(\widehat{\omega}_{2}^0\right)^m-\left(\widehat{\omega}_{1}^0\right)^{2m}\\
		&\quad+8\sum\limits_{k=1}^{m-1}{\left\{\left(\widehat{\omega}_{1}^0\right)^{2k} \left[\left(\widehat{\omega}_{2}^0\right)^{m-k}-\left(\widehat{\omega}_{1}^0\right)^{2m-2k}\right] -m\left(\widehat{\omega}_{1}^0\right)^{2m-2}\left[\left(\widehat{\omega}_{2}^0\right) -\left(\widehat{\omega}_{1}^0\right)^2\right]\right\}}.
	\end{aligned}
\end{equation}
$H_{0}$ will be rejected with level $\alpha$ if $|W_{m,T}(\epsilon)|>{z_{\alpha{/2}}},$ where $z_{\alpha{/2}}$ is the upper $\alpha{/2}$ quantile of the standard normal distribution $N(0,1).$

The BDS test has some attractive properties, which makes it broadly used in economics and finance. However, it also has a shortcoming that can not be ignored: over-rejecting the null hypothesis, which will weaken with the increase of sample size and become no longer obvious when $T\geqslant{3000}$. In this paper, our target is to improve this defect of the BDS test. To explore the reason for over-rejection, we studied the asymptotic theory of the BDS test given by Broock et al. (1996) and find that to obtain the asymptotic normality of the BDS test under $H_{0}$, they directly applied the CLT of U-statistic to $C_{m,T}(\epsilon)$ and got the CLT (central limit theorem) of $C_{m,T}(\epsilon)$ as below:
\begin{equation}\label{CLTCI}
	\sqrt{T}\frac{\left(C_{m,T}(\epsilon)-\left(\omega_{1}^0\right)^m\right)} {v_{m}(\epsilon)} \stackrel{D}{\underset{under H_{0}}{\longrightarrow}}{N(0,1)},
\end{equation}
where
$\frac{1}{4}v_{m}^{2}(\epsilon)=\left(\omega_{2}^0\right)^{m} -\left(\omega_{1}^0\right)^{2m}+2\sum\limits_{k=1}^{m-1} \left[\left(\omega_{2}^0\right)^{m-k}\left(\omega_{1}^0\right)^{2k}-\left(\omega_{1}^0\right)^{2m}\right],$ $\omega_{1}^0=P(|u_{1}-u_{2}|<{\epsilon}),$
$\omega_{2}^0=P(|u_{1}-u_{2}|<{\epsilon},|u_{2}-u_{3}|<{\epsilon}).$
Then, they replaced $\left(\omega_{1}^0\right)^{m}$ with $\left(\widehat{\omega}_{1}^0\right)^m$ and got the BDS test in \eqref{BDStest} based on the delta method and Slutsky's theorem. Therefore, easy to understand that the asymptotic theory of the BDS test holds on the premise that $C_{m,T}(\epsilon)$ is a U-statistic. However, unfortunately, $C_{m,T}(\epsilon)$ isn't an U-statistic, because $Y_{t}^m$ and $Y_{s}^m$ involved in $C_{m,T}(\epsilon)$ are not independent when $m\geq{2}$ and $0<s-t<m$.
Therefore, treating $C_{m,T}(\epsilon)$ as an U-statistic to get \eqref{CLTCI} will definitely cause bias, and the over-rejection problem of the BDS test is a manifestation of this bias.

\section{The Asymptotic Theory of the Correlation Integral}
To improve the over-rejection problem of the BDS test, we divide the correlation integral $C_{m,T}(\epsilon)$ into the following two parts according to $0<s-t<m$ and $s-t\geq{m}$:
\begin{equation}\label{DivitionCI}
	\begin{aligned}
		C_{m,T}(\epsilon)&=\frac{1}{N}\sum_{1\leq{t}<s\leq{T_{m}}}I_{\epsilon}(\|Y_{t}^m-Y_{s}^m\|)=\frac{1}{N}\sum_{t=1}^{T_m-1}\sum_{s=t+1}^{T_m}I_{\epsilon}(\|Y_{t}^m-Y_{s}^m\|)\\ &=\breve{C}_{m,T}(\epsilon)+\frac{N_{0}}{N}\widetilde{C}_{m,T}(\epsilon),
	\end{aligned}
\end{equation}
where $N=\binom{T_{m}}{2}$, $N_{0}=\binom{T_{m}-m+1}{2}$,
\begin{equation}\label{breveCI}
	\breve{C}_{m,T}(\epsilon)=\frac{1}{N}\sum\limits_{1\leq{s-t}\leq{m-1}} I_{\epsilon}(\|Y_{t}^m-Y_{s}^m\|)
	=\frac{1}{N}\sum_{k=1}^{m-1}\sum_{t=1}^{T_{m}-k} I_{\epsilon}(\|Y_{t}^m-Y_{t+k}^m\|),
\end{equation}
\begin{equation}\label{tildeCI}
	\widetilde{C}_{m,T}(\epsilon)=\frac{1}{N_{0}} \sum\limits_{s-t\geq{m}}I_{\epsilon}(\|Y_{t}^m-Y_{s}^m\|)
	=\frac{1}{N_0}\sum_{t=1}^{T_m-m}\sum_{s=t+m}^{T_m}I_{\epsilon}(\|Y_{t}^m-Y_{s}^m\|).
\end{equation}
Easy to find that under $H_0$ the independence betweenn $Y_t^m$ and $Y_s^m$ in $\widetilde{C}_{m,T}(\epsilon)$ is always satisfied, while the relationship between $Y_t^m$ and $Y_{t+k}^m$ in $\breve{C}_{m,T}(\epsilon)$ is a little more complicated.

Before stating the main results of this section, for brievity, we define the following two functions with $m$, $k$ and $d$ as parameters in combination with the preceding notations:
\begin{equation}\label{Wmk}
\begin{aligned}
	W_m(k,d)&=\left(\omega_{h+1}^{r+1}\right)^{\tau}\left(\omega_{h+1}^{r}\right)^{i-\tau}\left(\omega_h^{r}\right)^{k-i}\left(1-I_0(i-\tau)\right)\\
	&+\left(\omega_{h+1}^{r+1}\right)^i\left(\omega_h^{r+1}\right)^{d-rk-i}\left(\omega_h^{r}\right)^{k-d-rk}I_0(i-\tau),
\end{aligned}
\end{equation}
and 
\begin{equation}\label{Umk}
U_m(k,d)=\left(\omega_{h+1}^{h+1}\right)^{i-d}\left(\eta_{h+1}^h\right)^{d}\left(\omega_h^{h+1}\right)^{d}\left(\omega_h^h\right)^{k-i-d}-W_m(k,0)W_m(m,0),
\end{equation}
where $h=\lfloor{\frac{m}{k}}\rfloor$, $i=m-hk$, $r=\lfloor{\frac{d}{k}}\rfloor$, $\tau=d-rk$, $m, k, d\in{\mathbb{Z}^{+}}$.
Easy to verify 
\begin{equation}\label{EWmk}
W_m(k,0)=\left(\omega_{h}^0\right)^{k-i}\left(\omega_{h+1}^0\right)^i \text{\ and\ } W_m(m,0)=\left(\omega_1^0\right)^m.
\end{equation}  
Let $W_m^{(h)}(k,0)$ and $W_m^{(h+1)}(k,0)$ denote the partial derivatives of $W_m(k,0)$ with respect to $\omega_h^0$ and $\omega_{h+1}^0$ respectively, specifically, 
\begin{equation}\label{Wmkhh}
\begin{aligned}
	&W_m^{(h)}(k,0)=(k-i)\left(\omega_h^0\right)^{k-i-1}\left(\omega_{h+1}^0\right)^i,\quad  W_m^{(1)}(m,0)=m\left(\omega_1^0\right)^{m-1}
	\\
	&W_m^{(h+1)}(k,0)=i\left(\omega_h^0\right)^{k-i}\left(\omega_{h+1}^0\right)^{i-1},
\end{aligned}
\end{equation}
As before, we use 
$\widehat{W}_m(k,d)$, $\widehat{U}_m(k,d)$, $\widehat{W}_m^{(h)}(k,0)$ and $\widehat{W}_m^{(h+1)}(k,0)$ to denote the consistent estimation of $W_m(k,d)$, $U_m(k,d)$, $W_m^{(h)}(k,0)$ and $W_m^{(h+1)}(k,0)$ respectively.
Similar notations are used throughout the paper without further explanation.

With the above notations, we introduce the main result of this section, that is, the CLT of the correlation integral $C_{m,T}(\epsilon)$ under $H_0$.

\begin{theorem}\label{ThCLTCI}
	If $\{u_{t}\}$ is i.i.d., for fixed $m$,
	\begin{equation}\label{RCLTCI}
	\frac{C_{m,T}(\epsilon)-\mu_{m}}{\sigma_{m}} \stackrel{D}{\longrightarrow}{N(0,1)},
	\end{equation}
	where 
	\begin{equation}\label{ECI}
		\begin{aligned}
			\mu_{m}=\mathbf{E}C_{m,T}(\epsilon) &=\frac{1}{N}\sum_{k=1}^{m-1}(T_{m}-k)W_m(k,0)+\frac{N_{0}}{N}W_m(m,0),
		\end{aligned}
	\end{equation}
	\begin{equation}\label{CIsigma}
		\begin{aligned}
			\sigma_{m}^2&=\breve{\sigma}_m^2+\left(\frac{N_{0}}{N}\right)^2\widetilde{\sigma}_m^2,
		\end{aligned}
	\end{equation}
	where $\widetilde{\sigma}_m^2$ is equal to the variance in Theorem 2.1 of Luo et al.(2020),
	\begin{equation}\label{CIsigma1}
		\begin{aligned}
			\breve{\sigma}_m^2&=\frac{2}{N^2}
			\sum_{k=1}^{m-1}\mathcal{M}_{T,m}(k) \left\{2\sum_{d_{1}=1}^{m}\left[W_m(k,d_1)-W_m(k,0)W_m(m,0)\right] \right. \\
			&+\left[1-I_0(k-2i)\right]\left[\sum_{d_2=1}^iU_m(k,d_2)+ (k-2i)U_m(k,i)+\sum_{d_2=k-i+1}^{k-1}U_m(k,k-d_2)\right]\\
			&\left. +I_0(k-2i)\left[\sum_{d_2=1}^{k-i}U_m(k,d_2)+(2i-k)U_m(k,k-i)+\sum_{d_2=i+1}^{k-1}U_m(k,k-d_2)\right]\right\},
		\end{aligned}
	\end{equation}
	where $\mathcal{M}_{T,m}(k)=(T-4m-k+3)(T-4m-k+4)$, $i=m-hk$, $h=\lfloor{\frac{m}{k}}\rfloor$.
\end{theorem}
\begin{proof}[Proof of Theorem \ref{ThCLTCI}]
	is shown in Appendix A.
\end{proof}
\begin{remark}
	Unlike the CLT of $C_{m,T}(\epsilon)$ proposed by Broock et al. (1996), shown in \eqref{CLTCI}, here $\mu_m$ is  the exact expectation of $C_{m,T}(\epsilon)$, so there's no location bias for the CLT here.
\end{remark}
\begin{remark}
	Another difference from the CLT given by Broock et al. (1996) is that the asymptotic variance here includes all terms of order $O(T^{-1})$ and $O(T^{-2})$, which is necessary to construct the RBDS test to achieve good finite sample performance.
\end{remark}

\section{The Revised BDS (RBDS) Test}

In this section, we state the RBDS test and its asymptotic theory under $H_0$. First, we obtain the consistent estimation of $\mu_m$ by applying $\widehat{\omega}_{1}^0$, $\widehat{\omega}_{h}^0$ and $\widehat{\omega}_{h+1}^0$ to replace all unknown parameters in $\mu_m$, which is shown below:
\begin{equation}\label{estmu}
	\mu_{m,T}=\frac{1}{N}\sum\limits_{k=1}^{m-1}(T_{m}-k) \widehat{W}_m(k,0) +\frac{N_{0}}{N}\widehat{W}_m(m,0).
\end{equation}
Let $
\mathcal{K}_{m,T}(\epsilon)=C_{m,T}(\epsilon)-\mu_{m,T},
$
according to the delta method, the asymptotic variance of $\mathcal{K}_{m,T}(\epsilon)$ is approximately equal to the asymptotic variance of the following linear combination:
\begin{equation}\label{tildeKmT}
	\begin{aligned}
		\widetilde{\mathcal{K}}_{m,T}(\epsilon)=&\left[C_{m,T}(\epsilon)-\mu_{m}\right]\\ &-\frac{1}{N}\sum_{k=1}^{m-1}(T_{m}-k)W_m^{(h)}(k,0) \left[\left(\widehat{\omega}_{h}^0\right)-\left(\omega_{h}^0\right)\right]\\ &-\frac{1}{N}\sum_{k=1}^{m-1}(T_{m}-k)W_m^{(h+1)}(k,0) \left[\left(\widehat{\omega}_{h+1}^0\right)-\left(\omega_{h+1}^0\right)\right]\\ &-\frac{N_{0}}{N}W_m^{(1)}(m,0) \left[\left(\widehat{\omega}_{1}^0\right)-\left(\omega_{1}^0\right)\right].
	\end{aligned}
\end{equation}
Then, after some routine calculations, the CLT of $\mathcal{K}_{m,T}(\epsilon)$ can be obtained, which is given in the following Theorem \ref{ThCLTKCI}.
\begin{theorem}\label{ThCLTKCI}
	If $\{u_{t}\}$ is i.i.d., for $m\geq{2}$,
	\begin{equation}\label{}
		\frac{C_{m,T}(\epsilon)-\mu_{m,T}}{\nu_{m}} \stackrel{D}{\longrightarrow}{N(0,1)},
	\end{equation}
	where $\mu_{m,T}=\eqref{estmu}$,
	\begin{equation}\label{KCLsigma}
		\nu_{m}^2=\sigma_{m,m}^2 +\sigma_{1,1}^2-\sigma_{m,h}^2
		-\sigma_{m,h+1}^2 -\sigma_{m,1}^2+\sigma_{1,h}^2 +\sigma_{1,h+1}^2,
	\end{equation}
	where
	$\sigma_{m,m}^2=\eqref{CIsigma}$, 
	\begin{equation*}
		\begin{aligned}
			&\sigma_{1,1}^2=2\left(\frac{N_{0}}{N}W_m^{(1)}(m,0)\right)^2C_{T,1} \left[\left(\omega_{1}^0\right)+2(T-2)\left(\omega_{2}^0\right) -(2T-3)\left(\omega_{1}^0\right)^2\right],\\
					\end{aligned}
			\end{equation*} 
			\begin{equation*}
			\begin{aligned}
			&\sigma_{m,h}^2=\frac{4}{N^2}\sum_{k=1}^{m-1}(T_m-k)W_m^{(h)}(k,0)C_{T,h}M_{h+2}\\
			&\qquad \qquad \times{}  \left[W_m^{(1)}(m,0)\left(2\left(\omega_{h}^1\right)+\sum_{\kappa=1}^{h-1}\left(\xi_{h}^{\kappa}\right)\right)-m(h+1)W_m(m,0) \left(\omega_{h}^0\right) \right],\\
					\end{aligned}
			\end{equation*} 
			\begin{equation*}
			\begin{aligned}
			&\sigma_{m,h+1}^2=\frac{4}{N^2}\sum_{k=1}^{m-1}(T_m-k)W_m^{(h+1)}(k,0)C_{T,h+1}M_{h+3}\\
			&\qquad \qquad \times{}  \left[W_m^{(1)}(m,0)\left(2\left(\omega_{h+1}^1\right)+  \sum_{\kappa =1}^{h}\left(\xi_{h+1}^{\kappa}\right)\right)  -m(h+2)W_m(m,0) \left(\omega_{h+1}^0\right)\right],\\
			&\sigma_{m,1}^2=\frac{2N_{0}C_{T,1}}{N^2}W_m^{(1)}(m,0) 
			\left\{\sum\limits_{k=1}^{m-1}2\mathcal{N}_{T_m,k}\left[ W_m^{(h)}(k,0)\left(2\left(\omega_h^1\right)+\sum\limits_{\kappa=1}^{h_1}\left(\xi_h^{\kappa}\right)+i_1\left(\xi_h^{h-1}\right) \right) \right.\right.\\
			&\left. \qquad \qquad \qquad \qquad +W_m^{(h+1)}(k,0) \left(2\left(\omega_{h+1}^1\right)+\sum\limits_{\kappa=1}^{h_1}\left(\xi_{h+1}^{\kappa}\right)\right)-(m-k)W_m(k,0)\left(\omega_1^0\right)\right] \\
			& \qquad \qquad +\left[2W_m^{(1)}(m,0)\left(\left(M_3+(m-1)M_2\right)\left(\omega_1^1\right) +M_2\left(\left(\omega_1^0\right)+(m-1)\left(\eta_1^2\right)\left(\omega_1^0\right)^{-1}\right) \right) \right. \\
			&\left. \left.  \qquad\qquad\qquad\qquad -W_m(m,0)\left(\omega_1^0\right)\left(mM_3+(4m-2)M_2\right)\right]\right\},\\
			&\sigma_{1,h}^2=\frac{4N_{0}C_{T,1}}{N^2}W_m^{(1)}(m,0) \sum\limits_{k=1}^{m-1}(T_{m}-k)W_m^{(h)}(k,0)C_{T,h}/C_{T,h+1}\\ 
			& \qquad \qquad \times{} \left[2\left(\omega_h^1\right)+\sum\limits_{\kappa=1}^{h-1}\left(\xi_h^{\kappa}\right) -(h+1)\left(\omega_{1}^0\right)\left(\omega_{h}^0\right)\right],\\
			&\sigma_{1,h+1}^2=\frac{4N_{0}C_{T,1}}{N^2}W_m^{(1)}(m,0) \sum\limits_{k=1}^{m-1}(T_{m}-k)W_m^{(h+1)}(k,0)C_{T,h+1}/C_{T,h+2}\\ 
			& \qquad \qquad \times{} \left[2\left(\omega_{h+1}^1\right)+\sum\limits_{\kappa=1}^h\left(\xi_{h+1}^{\kappa}\right) -(h+2)\left(\omega_{1}^0\right)\left(\omega_{h+1}^0\right)\right],
		\end{aligned}
	\end{equation*}

	$\mathcal{N}_{T_m,k}=\binom{T_m-k}{2}$, $h=\lfloor{\frac{m}{k}}\rfloor$, $i=m-hk$,  $h_1=\lfloor{\frac{m-k}{k}}\rfloor$, $i_1=m-k-h_1k$. 
\end{theorem}

\begin{proof}[Proof of Theorem \ref{ThCLTKCI}]
	is shown in Appendix B.
\end{proof}

\begin{remark}
	Each term constituting the asymptotic variance in Theorem \ref{ThCLTKCI} has an explicit expression, and if $m$ is fixed, the $\nu_m$ in \ref{ThCLTKCI} can be simplified to a neat result. 
\end{remark}

\begin{corollary}\label{CLTRBDS}
	If $\{u_{t}\}$ is i.i.d., for $m\geq{2}$,
	\begin{equation}\label{RBDS}
		M_{m,T}(\epsilon):=\frac{C_{m,T}(\epsilon)-\mu_{m,T}} {\nu_{m,T}} \stackrel{D}{\longrightarrow}{N(0,1)},
	\end{equation}
	where $\mu_{m,T}=\eqref{estmu}$, $\nu^2_{m,T}$ is the consistent estimation of $\nu^2_{m}$, that is, $\nu^2_{m,T}$ is derived by replacing all unknown paremeters contained in the asymptotic variance $\nu^2_{m}$ of Theorem \ref{ThCLTKCI} with their consistent estimators.
\end{corollary}
\begin{proof}[Proof of Corollary \ref{CLTRBDS}]
	Since $\nu_{m,T}^2$ is the consistent estimation of $\nu_{m}^2$, based on the Slutsky's theorem, easy to get \eqref{RBDS}.
\end{proof}

$M_{m,T}(\epsilon)$ in \eqref{RBDS} is the RBDS test we propose. The null hypothesis will be rejected, when $\left|M_{m,T}(\epsilon)\right|>z_{\alpha/2}$, where $z_{\alpha/2}$ is the upper $\alpha/2$ quantile of the standard normal distribution.

\section{Finite Sample Behavior of the RBDS Test}

In this section, we design four experiments to compare the different performances of  BDS test and RBDS test. Among them, the main purpose of the first two experiments is to compare the empirical sizes and powers of these two tests under different sample sizes. And the purpose of the latter two experiments is to analyze the different performances of BDS test and RBDS test affected by parameter estimation when they are used to evaluate the goodness-of-fit for a model. In all experiments, the significant level is denoted by $\alpha$ and $\alpha=0.05, 0.1$, and the setting of parameter $m$ is the same as that set in Brook et al. (1996): $m=2, 3$. Besides, the experiment was repeated 1000 times under each parameter setting.

The four experiments are designed as follows:
\begin{itemize}
\item Experiment 1: $\{u_{t}\}\stackrel{\mathrm{i.i.d.}}{\sim}N(0,1)$.
\item Experiment 2: $\{u_{t}\}$ is generated from the following GARCH(1,1) model:
\begin{equation}\label{simu:garch1}
	X_{t}=h_{t}e_{t}, \qquad  h_{t}=\sqrt{1+0.1X_{t-1}^2+0.1h_{t-1}^2}, \quad e_{t}\overset{\mathrm{i.i.d.}}{\thicksim} \mathrm{N}(0,1).
\end{equation}
\end{itemize}
Throughout Experiment 1 and 2, the length $T$ of $\{u_t\}$ is chosen to be 200, 400, 600, 800, 1000, 3000, the parameter $\epsilon$ involved in these two tests is set: $\epsilon=0.5\sqrt{\mathbf{Var}(u_{t})}$, which is recommended by Broock et al. (1996). These results are shown in Table \ref{size} and \ref{power} respectively.

\begin{itemize}
	\item Experiment 3: $\{\hat{u}_{t}\}$ is the estimation residuals of the following GARCH(1,1) model:
	\begin{equation}\label{simu:garch2}
		X_{t}=h_{t}e_{t}, \qquad  h_{t}=\sqrt{1+0.1X_{t-1}^2+0.85h_{t-1}^2}, \quad e_{t}\overset{\mathrm{i.i.d.}}{\thicksim} \mathrm{N}(0,1).
	\end{equation}
	\item Experiment 4: $\{\hat{u}_{t}\}$  is the estimation residuals of the following EGARCH(1,1) model:
	\begin{equation}\label{simu:egarch}
		X_{t}=h_{t}e_{t}, \  \log{h^2_{t}}=0.01(|X_{t-1}|+0.15X_{t-1})/h_{t-1}+0.9\log{h_{t-1}^2},\ e_{t}\overset{\mathrm{i.i.d.}}{\thicksim} \mathrm{N}(0,1).
	\end{equation}
\end{itemize}
The data generating processes considered in Experiment 3 and 4 are representative of the type of the estimates obtained with high frequency stock market returns, which are chosen with reference to De Lima (1996). For each model, we also refer to De Lima (1996) and set other parameters as follows: the $\epsilon=\epsilon_0\sqrt{\mathbf{Var}(u_{t})}$, $\epsilon_0=0.50, 1.00, 1.25$, and $T=500, 1000$, which is the length of $\{\hat{u}_t\}$. Besides, we compute BDS test and RBDS test on four different sequences: the innovations $\{u_t\}$, the estimation residuals $\{\hat{u}_t\}$, the square estimation residuals $\{\hat{u}_t^2\}$  and the logarithm of the square estimation residuals $\{\log{\hat{u}_t^2}\}$ of these two models. The frequency of rejections of the null hypothesis of iid over 1000 simulated realization of each model are given in Table \ref{tab:GARCH} and \ref{tab:EGARCH}, respectively. Note that the innovation sequence is observed without estimation and can therefore be used for comparison purpose.

As mentioned above, the empirical sizes and powers of the two tests are presented in Table \ref{size} and Table \ref{power} respectively. From Table \ref{size}, we see that the BDS test has a over-rejection problem and the smaller the length is, the more serious the problem is. While the RBDS test efficiently improves this problem. On the other hand, the results shown in Table \ref{power} reflect that the empirical power of RBDS test is slightly inferior to that of BDS test when $T$ is small, but with the increase of $T$, it increases at a higher rate to be similar to that of the BDS test.

Table \ref{tab:GARCH} and \ref{tab:EGARCH} present evidence of the size distortions of both BDS test and RBDS test that arise from applying them to the estimated residuals of ARCH-type models. But the distortion of RBDS test is not as serious as BDS test. In all situations of Experiment 3 and 4, both BDS test and RBDS test tend to over-reject the null hypothesis that an ARCH-type specification is appropriate. However, the logarithmic transformation correct this problem. Therefore, similar to BDS test, if the goodness-of-fit for a model is to be tested with RBDS test, especially for ARCH-type models, we recommend logarithmic transformation preprocessing of the estimate residuals of the model.

\begin{table}[h]
	\centering
	\caption{Empirical Sizes of BDS Test and RBDS Test}
	\label{size}
\begin{threeparttable}
		\begin{tabular}{@{}c|ccccc|ccccc@{}}
			\toprule[1.5pt]
			\multicolumn{1}{c|}{}& \multicolumn{5}{c|}{$\alpha=0.05$}   & \multicolumn{5}{c}{$\alpha=0.10$} \\
		\cmidrule(l){2-6}\cmidrule(l){7-11}
			&\multicolumn{2}{c}{$m=2$}&
			&\multicolumn{2}{c|}{$m=3$} &\multicolumn{2}{c}{$m=2$}
			&&\multicolumn{2}{c}{$m=3$} \\
			\cmidrule(lr){2-3} \cmidrule(lr){5-6} \cmidrule(lr){7-8}\cmidrule(l){10-11}
			$u_{t}\stackrel{\mathrm{i.i.d.}}{\sim}N(0,1)$&BDS&RBDS &&BDS&RBDS &BDS&RBDS &&BDS&RBDS \\ \midrule
			$T=200$   & 0.1776 & 0.0507 && 0.2225 & 0.0446  
			& 0.2598 & 0.1004 && 0.3045 & 0.0917\\
			$T=400$  & 0.1109 & 0.0495 && 0.1397 & 0.0454 
			& 0.1807 & 0.0987 && 0.2123 & 0.0935\\
			$T=600$ & 0.0894 & 0.0479 && 0.1082 & 0.0477
			& 0.1528 & 0.0969 && 0.1759 & 0.0947\\
			$T=800$ & 0.0810 & 0.0503 && 0.0898 & 0.0460
			& 0.1374 & 0.0957 && 0.1530 & 0.0924\\
			$T=1000$ & 0.0732 & 0.0478 && 0.0841 & 0.0484
			& 0.1273 & 0.0967 && 0.1479 & 0.0991\\
			$T=2000$ & 0.0665 & 0.0528 && 0.0680 & 0.0520 
			& 0.1205 & 0.1035 && 0.1190 & 0.0940\\
			$T=3000$ & 0.0564 & 0.0509 && 0.0581 & 0.0521 
			& 0.1145 & 0.1009 && 0.1175 & 0.0953 \\
			\bottomrule[1.5pt]
	\end{tabular}
\end{threeparttable}
\end{table}

\begin{table}[h]
	\centering
	\caption{Empirical Powers of BDS Test and RBDS Test}
	\label{power}
	\begin{threeparttable}
		\begin{tabular}{c|ccccc}
			\toprule[1.5pt]
			\multicolumn{1}{c|}{$\alpha=0.10$}   &\multicolumn{2}{c}{$m=2$}&&\multicolumn{2}{c}{$m=3$}\\
			\cmidrule{2-3}\cmidrule{5-6}
			\multicolumn{1}{c|}{GARCH(1,1)}&BDS&RBDS&&BDS&RBDS\\ \midrule
			$T=200$   & 0.3803 & 0.2074 & &0.3942 & 0.1599 \\
			$T=400$  & 0.4660 & 0.3427 && 0.4431 & 0.2920 \\
			$T=600$ & 0.5751 & 0.4854 & &0.5464 & 0.4213\\
			$T=800$ & 0.6677 & 0.5949 & &0.6295 & 0.5320
			\\
			$T=1000$ & 0.7600 & 0.7210 & &0.6890 &0.5920\\
			$T=2000$ & 0.9410 & 0.9300 & &0.9300 &0.9111\\
			$T=3000$ & 0.9843 & 0.9829 & &0.9818 & 0.9764\\
			\bottomrule[1.5pt]
	\end{tabular}
	\end{threeparttable}
\end{table}

% Please add the following required packages to your document preamble:
% \usepackage{booktabs}
% \usepackage{multirow}
% \usepackage{graphicx}
\begin{table}[h]
	\centering
	\caption{Different Performances of BDS Test and RBDS Test in GARCH(1,1) Model Diagnosis}
	\label{tab:GARCH}
	\begin{threeparttable}
	\resizebox{\textwidth}{!}{%
		\begin{tabular}{@{}cc|ccccccccccc|ccccccccccc@{}}
			\toprule[1.5pt]
			\multicolumn{2}{c|}{\multirow{3}{*}{$m=2$}}                           & \multicolumn{11}{c|}{$T=500$}                                                                                                                       & \multicolumn{11}{c}{$T=1000$}                                                                                                                      \\ \cmidrule(l){3-24} 
			\multicolumn{2}{c|}{}                                                 & \multicolumn{2}{c}{$u_t$} &  & \multicolumn{2}{c}{$\hat{u}_t$} &  & \multicolumn{2}{c}{$\hat{u}_t^2$} &  & \multicolumn{2}{c|}{$\log{\hat{u}_t^2}$} & \multicolumn{2}{c}{$u_t$} &  & \multicolumn{2}{c}{$\hat{u}_t$} &  & \multicolumn{2}{c}{$\hat{u}_t^2$} &  & \multicolumn{2}{c}{$\log{\hat{u}_t^2}$} \\ \cmidrule(lr){3-4} \cmidrule(lr){6-7} \cmidrule(lr){9-10} \cmidrule(lr){12-15} \cmidrule(lr){17-18} \cmidrule(lr){20-21} \cmidrule(l){23-24} 
			\multicolumn{2}{c|}{}                                                 & BDS         & RBDS        &  & BDS            & RBDS           &  & BDS             & RBDS            &  & BDS                 & RBDS               & BDS         & RBDS        &  & BDS            & RBDS           &  & BDS             & RBDS            &  & BDS                & RBDS               \\ \cmidrule(r){1-4} \cmidrule(lr){6-7} \cmidrule(lr){9-10} \cmidrule(lr){12-15} \cmidrule(lr){17-18} \cmidrule(lr){20-21} \cmidrule(l){23-24} 
			\multicolumn{1}{c|}{\multirow{2}{*}{$\epsilon_0=0.50$}}  & $\alpha=0.05$ & 0.093       & 0.047       &  & 0.078          & 0.038          &  & 0.051           & 0.023           &  & 0.079               & 0.033              & 0.077       & 0.046       &  & 0.073          & 0.047          &  & 0.068           & 0.039           &  & 0.056              & 0.033              \\
			\multicolumn{1}{c|}{}                                 & $\alpha=0.10$  & 0.155       & 0.088       &  & 0.131          & 0.073          &  & 0.940            & 0.041           &  & 0.135               & 0.073              & 0.138       & 0.098       &  & 0.139          & 0.104          &  & 0.105           & 0.075           &  & 0.111              & 0.070               \\ \midrule
			\multicolumn{1}{c|}{\multirow{2}{*}{$\epsilon_0=1.00$}}    & $\alpha=0.05$ & 0.065       & 0.050        &  & 0.052          & 0.042          &  & 0.045           & 0.050            &  & 0.069               & 0.028              & 0.059       & 0.048       &  & 0.062          & 0.051          &  & 0.063           & 0.022           &  & 0.071              & 0.036              \\
			\multicolumn{1}{c|}{}                                 & $\alpha=0.10$  & 0.115       & 0.100         &  & 0.100            & 0.075          &  & 0.080            & 0.020            &  & 0.123               & 0.057              & 0.106       & 0.100         &  & 0.116          & 0.103          &  & 0.118           & 0.055           &  & 0.127              & 0.082              \\ \midrule
			\multicolumn{1}{c|}{\multirow{2}{*}{$\epsilon_0=1.25$}} & $\alpha=0.05$ & 0.077       & 0.053       &  & 0.062          & 0.043          &  & 0.038           & 0.004           &  & 0.079               & 0.02               & 0.057       & 0.051       &  & 0.059          & 0.052          &  & 0.059           & 0.022           &  & 0.062              & 0.029              \\
			\multicolumn{1}{c|}{}                                 & $\alpha=0.10$  & 0.141       & 0.115       &  & 0.098          & 0.077          &  & 0.088           & 0.015           &  & 0.144               & 0.055              & 0.100         & 0.092       &  & 0.117          & 0.096          &  & 0.106           & 0.051           &  & 0.111              & 0.065              \\ \midrule[1pt]
			\multicolumn{2}{c|}{\multirow{3}{*}{$m=3$}}                           & \multicolumn{11}{c|}{$T=500$}                                                                                                                       & \multicolumn{11}{c}{$T=1000$}                                                                                                                      \\ \cmidrule(l){3-24} 
			\multicolumn{2}{c|}{}                                                 & \multicolumn{2}{c}{$u_t$} &  & \multicolumn{2}{c}{$\hat{u}_t$} &  & \multicolumn{2}{c}{$\hat{u}_t^2$} &  & \multicolumn{2}{c|}{$\log{\hat{u}_t^2}$} & \multicolumn{2}{c}{$u_t$} &  & \multicolumn{2}{c}{$\hat{u}_t$} &  & \multicolumn{2}{c}{$\hat{u}_t^2$} &  & \multicolumn{2}{c}{$\log{\hat{u}_t^2}$} \\ \cmidrule(lr){3-4} \cmidrule(lr){6-7} \cmidrule(lr){9-10} \cmidrule(lr){12-15} \cmidrule(lr){17-18} \cmidrule(lr){20-21} \cmidrule(l){23-24} 
			\multicolumn{2}{c|}{}                                                 & BDS         & RBDS        &  & BDS            & RBDS           &  & BDS             & RBDS            &  & BDS                 & RBDS               & BDS         & RBDS        &  & BDS            & RBDS           &  & BDS             & RBDS            &  & BDS                & RBDS               \\ \midrule
			\multicolumn{1}{c|}{\multirow{2}{*}{$\epsilon_0=0.50$}}  & $\alpha=0.05$ & 0.128       & 0.041       &  & 0.201          & 0.099          &  & 0.138           & 0.103           &  & 0.078               & 0.040               & 0.083       & 0.045       &  & 0.286          & 0.218          &  & 0.204           & 0.189           &  & 0.076              & 0.053              \\
			\multicolumn{1}{c|}{}                                 & $\alpha=0.10$  & 0.195       & 0.090        &  & 0.295          & 0.169          &  & 0.188           & 0.152           &  & 0.140                & 0.081              & 0.155       & 0.094       &  & 0.361          & 0.299          &  & 0.275           & 0.253           &  & 0.123              & 0.097              \\ \midrule
			\multicolumn{1}{c|}{\multirow{2}{*}{$\epsilon_0=1.00$}}    & $\alpha=0.05$ & 0.060        & 0.046       &  & 0.153          & 0.118          &  & 0.122           & 0.071           &  & 0.066               & 0.035              & 0.057       & 0.047       &  & 0.289          & 0.262          &  & 0.209           & 0.169           &  & 0.069              & 0.055              \\
			\multicolumn{1}{c|}{}                                 & $\alpha=0.10$  & 0.113       & 0.089       &  & 0.228          & 0.191          &  & 0.181           & 0.115           &  & 0.113               & 0.078              & 0.109       & 0.082       &  & 0.361          & 0.330           &  & 0.269           & 0.223           &  & 0.118              & 0.091              \\ \midrule
			\multicolumn{1}{c|}{\multirow{2}{*}{$\epsilon_0=1.25$}} & $\alpha=0.05$ & 0.068       & 0.050        &  & 0.170           & 0.138          &  & 0.113           & 0.058           &  & 0.064               & 0.029              & 0.052       & 0.048       &  & 0.302          & 0.271          &  & 0.206           & 0.136           &  & 0.064              & 0.040               \\
			\multicolumn{1}{c|}{}                                 & $\alpha=0.10$  & 0.129       & 0.108       &  & 0.237          & 0.198          &  & 0.183           & 0.098           &  & 0.120                & 0.053              & 0.097       & 0.088       &  & 0.364          & 0.346          &  & 0.287           & 0.212           &  & 0.115              & 0.085              \\ \bottomrule[1.5pt]
		\end{tabular}%
	}
\begin{tablenotes}
	\footnotesize
	\item[-] $u_{t}$ is innovation, which is i.i.d. sample from $N(0,1)$.
	\item[-] $\hat{u}_t$ is the estimate residual of GARCH(1,1) model, $\hat{u}_t^2$ is the square of $\hat{u}_t$ and $\log{\hat{u}_t^2}$ is the logarithm of $\hat{u}_t^2$.
\end{tablenotes}
	\end{threeparttable}
\end{table}

% Please add the following required packages to your document preamble:
% \usepackage{booktabs}
% \usepackage{multirow}
\begin{table}[H]
	\centering
	\caption{Different Performances of BDS Test and RBDS Test in EGARCH(1,1) Model Diagnosis}
	\label{tab:EGARCH}
	\begin{threeparttable}
	\resizebox{\textwidth}{!}{%
	\begin{tabular}{@{}cc|ccccccccccc|ccccccccccc@{}}
		\toprule[1.5pt]
		\multicolumn{2}{c|}{\multirow{3}{*}{$m=2$}} &
		\multicolumn{11}{c|}{$T=500$} &
		\multicolumn{11}{c}{$T=1000$} \\ \cmidrule(l){3-24} 
		\multicolumn{2}{c|}{} &
		\multicolumn{2}{c}{$u_t$} &
		&
		\multicolumn{2}{c}{$\hat{u}_t$} &
		&
		\multicolumn{2}{c}{$\hat{u}_t^2$} &
		&
		\multicolumn{2}{c|}{$\log{\hat{u}_t^2}$} &
		\multicolumn{2}{c}{$u_t$} &
		&
		\multicolumn{2}{c}{$\hat{u}_t$} &
		&
		\multicolumn{2}{c}{$\hat{u}_t^2$} &
		&
		\multicolumn{2}{c}{$\log{\hat{u}_t^2}$} \\ \cmidrule(lr){3-4} \cmidrule(lr){6-7} \cmidrule(lr){9-10} \cmidrule(lr){12-15} \cmidrule(lr){17-18} \cmidrule(lr){20-21} \cmidrule(l){23-24} 
		\multicolumn{2}{c|}{} &
		BDS &
		RBDS &
		&
		BDS &
		RBDS &
		&
		BDS &
		RBDS &
		&
		BDS &
		RBDS &
		BDS &
		RBDS &
		&
		BDS &
		RBDS &
		&
		BDS &
		RBDS &
		&
		BDS &
		RBDS \\ \cmidrule(r){1-4} \cmidrule(lr){6-7} \cmidrule(lr){9-10} \cmidrule(lr){12-15} \cmidrule(lr){17-18} \cmidrule(lr){20-21} \cmidrule(l){23-24} 
		\multicolumn{1}{c|}{\multirow{2}{*}{$\epsilon_0=0.50$}} &
		$\alpha=0.05$ &
		0.093 &
		0.047 &
		&
		0.315 &
		0.230 &
		&
		0.226 &
		0.063 &
		&
		0.070 &
		0.033 &
		0.077 &
		0.046 &
		&
		0.496 &
		0.423 &
		&
		0.346 &
		0.169 &
		&
		0.068 &
		0.046 \\
		\multicolumn{1}{c|}{} &
		$\alpha=0.10$ &
		0.155 &
		0.088 &
		&
		0.405 &
		0.303 &
		&
		0.308 &
		0.100 &
		&
		0.121 &
		0.071 &
		0.138 &
		0.098 &
		&
		0.608 &
		0.546 &
		&
		0.447 &
		0.261 &
		&
		0.119 &
		0.086 \\ \cmidrule(r){1-4} \cmidrule(lr){6-7} \cmidrule(lr){9-10} \cmidrule(lr){12-15} \cmidrule(lr){17-18} \cmidrule(lr){20-21} \cmidrule(l){23-24} 
		\multicolumn{1}{c|}{\multirow{2}{*}{$\epsilon_0=1.00$}} &
		$\alpha=0.05$ &
		0.065 &
		0.050 &
		&
		0.313 &
		0.202 &
		&
		0.196 &
		0.014 &
		&
		0.071 &
		0.034 &
		0.059 &
		0.048 &
		&
		0.561 &
		0.45 &
		&
		0.273 &
		0.054 &
		&
		0.057 &
		0.029 \\
		\multicolumn{1}{c|}{} &
		$\alpha=0.10$ &
		0.115 &
		0.100 &
		&
		0.395 &
		0.288 &
		&
		0.262 &
		0.035 &
		&
		0.122 &
		0.069 &
		0.106 &
		0.100 &
		&
		0.678 &
		0.602 &
		&
		0.348 &
		0.102 &
		&
		0.107 &
		0.066 \\ \cmidrule(r){1-4} \cmidrule(lr){6-7} \cmidrule(lr){9-10} \cmidrule(lr){12-15} \cmidrule(lr){17-18} \cmidrule(lr){20-21} \cmidrule(l){23-24} 
		\multicolumn{1}{c|}{\multirow{2}{*}{$\epsilon_0=1.25$}} &
		$\alpha=0.05$ &
		0.077 &
		0.053 &
		&
		0.335 &
		0.159 &
		&
		0.185 &
		0.013 &
		&
		0.074 &
		0.024 &
		0.057 &
		0.051 &
		&
		0.559 &
		0.436 &
		&
		0.264 &
		0.032 &
		&
		0.069 &
		0.031 \\
		\multicolumn{1}{c|}{} &
		$\alpha=0.10$ &
		0.141 &
		0.115 &
		&
		0.438 &
		0.264 &
		&
		0.239 &
		0.025 &
		&
		0.129 &
		0.054 &
		0.100 &
		0.092 &
		&
		0.657 &
		0.553 &
		&
		0.324 &
		0.068 &
		&
		0.113 &
		0.068 \\ \midrule[1pt]
		\multicolumn{2}{c|}{\multirow{3}{*}{$m=3$}} &
		\multicolumn{11}{c|}{$T=500$} &
		\multicolumn{11}{c}{$T=1000$} \\ \cmidrule(l){3-24} 
		\multicolumn{2}{c|}{} &
		\multicolumn{2}{c}{$u_t$} &
		&
		\multicolumn{2}{c}{$\hat{u}_t$} &
		&
		\multicolumn{2}{c}{$\hat{u}_t^2$} &
		&
		\multicolumn{2}{c|}{$\log{\hat{u}_t^2}$} &
		\multicolumn{2}{c}{$u_t$} &
		&
		\multicolumn{2}{c}{$\hat{u}_t$} &
		&
		\multicolumn{2}{c}{$\hat{u}_t^2$} &
		&
		\multicolumn{2}{c}{$\log{\hat{u}^2_t}$} \\ \cmidrule(lr){3-4} \cmidrule(lr){6-7} \cmidrule(lr){9-10} \cmidrule(lr){12-15} \cmidrule(lr){17-18} \cmidrule(lr){20-21} \cmidrule(l){23-24} 
		\multicolumn{2}{c|}{} &
		BDS &
		RBDS &
		&
		BDS &
		RBDS &
		&
		BDS &
		RBDS &
		&
		BDS &
		RBDS &
		BDS &
		RBDS &
		&
		BDS &
		RBDS &
		&
		BDS &
		RBDS &
		&
		BDS &
		RBDS \\ \cmidrule(r){1-4} \cmidrule(lr){6-7} \cmidrule(lr){9-10} \cmidrule(lr){12-15} \cmidrule(lr){17-18} \cmidrule(lr){20-21} \cmidrule(l){23-24} 
		\multicolumn{1}{c|}{\multirow{2}{*}{$\epsilon_0=0.50$}} &
		$\alpha=0.05$ &
		0.128 &
		0.041 &
		&
		0.466 &
		0.383 &
		&
		0.303 &
		0.173 &
		&
		0.096 &
		0.049 &
		0.083 &
		0.045 &
		&
		0.645 &
		0.609 &
		&
		0.453 &
		0.327 &
		&
		0.075 &
		0.049 \\
		\multicolumn{1}{c|}{} &
		$\alpha=0.10$ &
		0.195 &
		0.090 &
		&
		0.564 &
		0.488 &
		&
		0.384 &
		0.252 &
		&
		0.149 &
		0.093 &
		0.155 &
		0.094 &
		&
		0.747 &
		0.707 &
		&
		0.562 &
		0.439 &
		&
		0.123 &
		0.102 \\ \cmidrule(r){1-4} \cmidrule(lr){6-7} \cmidrule(lr){9-10} \cmidrule(lr){12-15} \cmidrule(lr){17-18} \cmidrule(lr){20-21} \cmidrule(l){23-24} 
		\multicolumn{1}{c|}{\multirow{2}{*}{$\epsilon_0=1.00$}} &
		$\alpha=0.05$ &
		0.06 &
		0.046 &
		&
		0.426 &
		0.338 &
		&
		0.252 &
		0.051 &
		&
		0.067 &
		0.032 &
		0.057 &
		0.047 &
		&
		0.705 &
		0.652 &
		&
		0.385 &
		0.147 &
		&
		0.063 &
		0.049 \\
		\multicolumn{1}{c|}{} &
		$\alpha=0.10$ &
		0.113 &
		0.089 &
		&
		0.518 &
		0.444 &
		&
		0.34 &
		0.102 &
		&
		0.13 &
		0.073 &
		0.109 &
		0.082 &
		&
		0.782 &
		0.749 &
		&
		0.465 &
		0.233 &
		&
		0.105 &
		0.089 \\ \cmidrule(r){1-4} \cmidrule(lr){6-7} \cmidrule(lr){9-10} \cmidrule(lr){12-15} \cmidrule(lr){17-18} \cmidrule(lr){20-21} \cmidrule(l){23-24} 
		\multicolumn{1}{c|}{\multirow{2}{*}{$\epsilon_0=1.25$}} &
		$\alpha=0.05$ &
		0.068 &
		0.05 &
		&
		0.462 &
		0.354 &
		&
		0.255 &
		0.035 &
		&
		0.073 &
		0.04 &
		0.052 &
		0.048 &
		&
		0.715 &
		0.641 &
		&
		0.323 &
		0.098 &
		&
		0.065 &
		0.042 \\
		\multicolumn{1}{c|}{} &
		$\alpha=0.10$ &
		0.129 &
		0.108 &
		&
		0.556 &
		0.457 &
		&
		0.327 &
		0.076 &
		&
		0.114 &
		0.075 &
		0.097 &
		0.088 &
		&
		0.798 &
		0.756 &
		&
		0.396 &
		0.165 &
		&
		0.116 &
		0.084 \\ \bottomrule[1,5pt]
	\end{tabular}%
}
\begin{tablenotes}
	\footnotesize
	\item[-] $u_{t}$ is innovation, which is i.i.d. sample from $N(0,1)$.
\item[-] $\hat{u}_t$ is the estimate residual of EGARCH(1,1) model, $\hat{u}_t^2$ is the square of $\hat{u}_t$ and $\log{\hat{u}_t^2}$ is the logarithm of $\hat{u}_t^2$.
	\end{tablenotes}
\end{threeparttable}
\end{table}

\section{Empirical Analysis}
In this section, we aim to compare the performance of BDS test and RBDS test as model diagnostic tools in practice.  For this purpose, referring to Cryper and Chan (2008) and Tsay (2005), we select the following three datasets:
\begin{itemize}
	\item The daily values of a unit of the CREF stocks fund over the period from August 26, 2004 to August 15, 2006 , which is denoted as $\{P_{t,1}, t=1,\cdots,500\}$ and available at:  \url{http://homepage.divms.uiowa.edu/~kchan/TSA/Datasets/CREF.dat}. 
	%Figure \ref{series}\subref{stock_value} present the time series plot of the CREF data, which shows a generally increasing trend with a hint of higher variability with higher level of the stock value. 
	\item The monthly log returns of IBM stock from January 1926 to December 1997, which is denoted as $\{P_{t,2}, t=1,\cdots,864\}$ and available at: \url{https://faculty.chicagobooth.edu/-/media/faculty/ruey-s-tsay/teaching/fts3/m-ibmvwew2697.txt}.
	\item The percentage changes of the exchange rate between mark and dollar in 10-minute intervals from June 5, 1989 to June 19, 1989, which is denoted as $\{P_{t,3}, t=1,\cdots,2488\}$ and available at: \url{https://faculty.chicagobooth.edu/-/media/faculty/ruey-s-tsay/teaching/fts3/exch-perc.txt}.
	\end{itemize}

For the CREF dataset, we let $R_{t,1}=100\times{(\log{P_{t,1}}-\log{P_{t-1,1}})}$, then $\{R_{1,t}\}$ represents the log return series of CREF, which is plotted in Figure \ref{series}\subref{stock_return}. The other two datasets are plotted in Figure \ref{series}\subref{series_IBM} and \ref{series}\subref{series_Exch_Rate} respectively. From these figures, it's easy to find that the data were more volatile over some time periods, in other words, there are ARCH effects in these datasets.  This phenomenon is conformed by the sample ACF of various functions of these data, which are shown in Figure \ref{ACF_series}.

\begin{figure}[!]
	\centering
	\vspace{-0.3cm}
	\subfloat[Log Returns of CREF]{
		\begin{minipage}[b]{0.45\textwidth}
			\centering
			\label{stock_return}
			\includegraphics[width=1\textwidth]{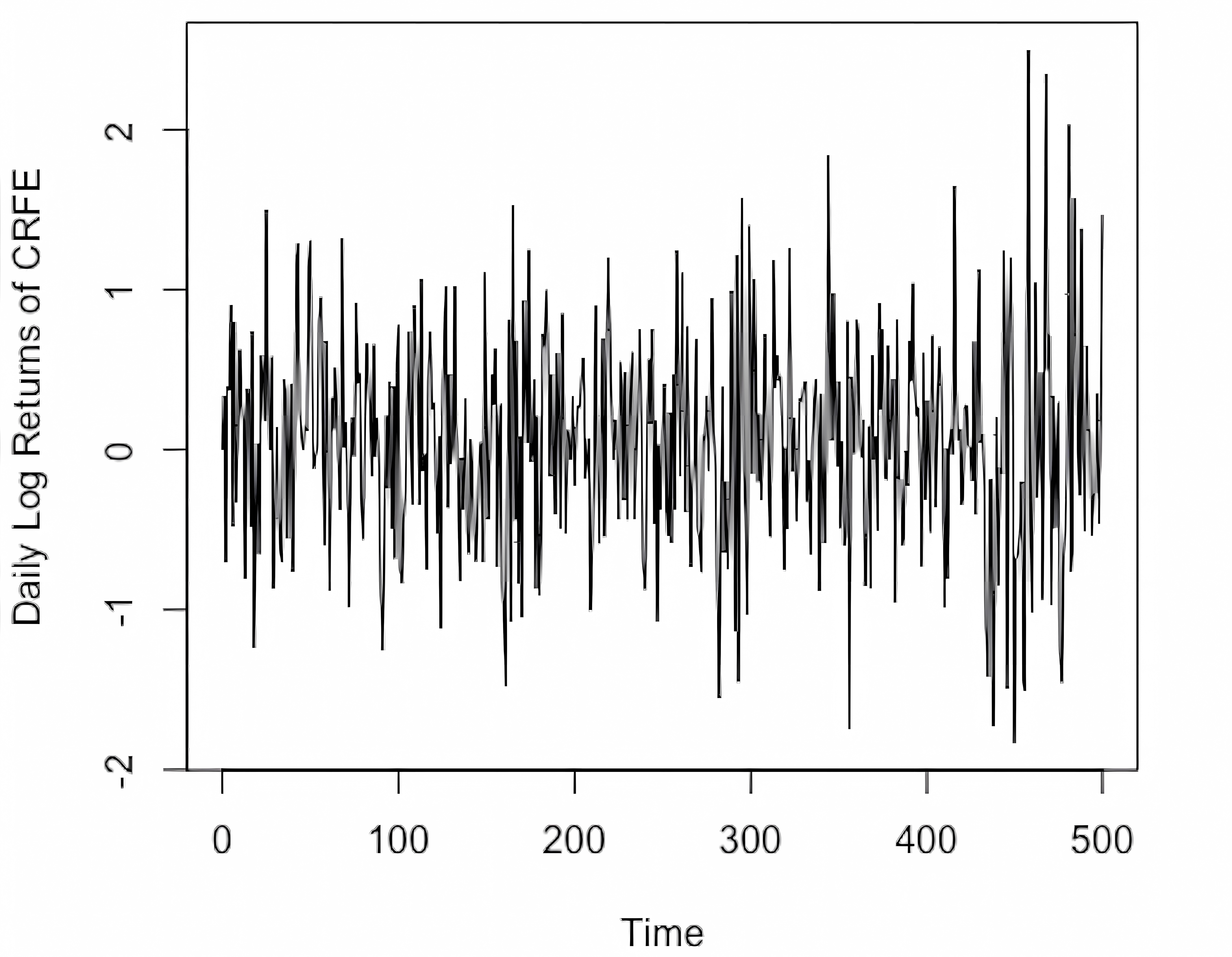}\\
	\end{minipage}}
	\subfloat[Log Returns of IBM]{
	\begin{minipage}[b]{0.45\textwidth}
		\centering
		\label{series_IBM}
		\includegraphics[width=1\textwidth]{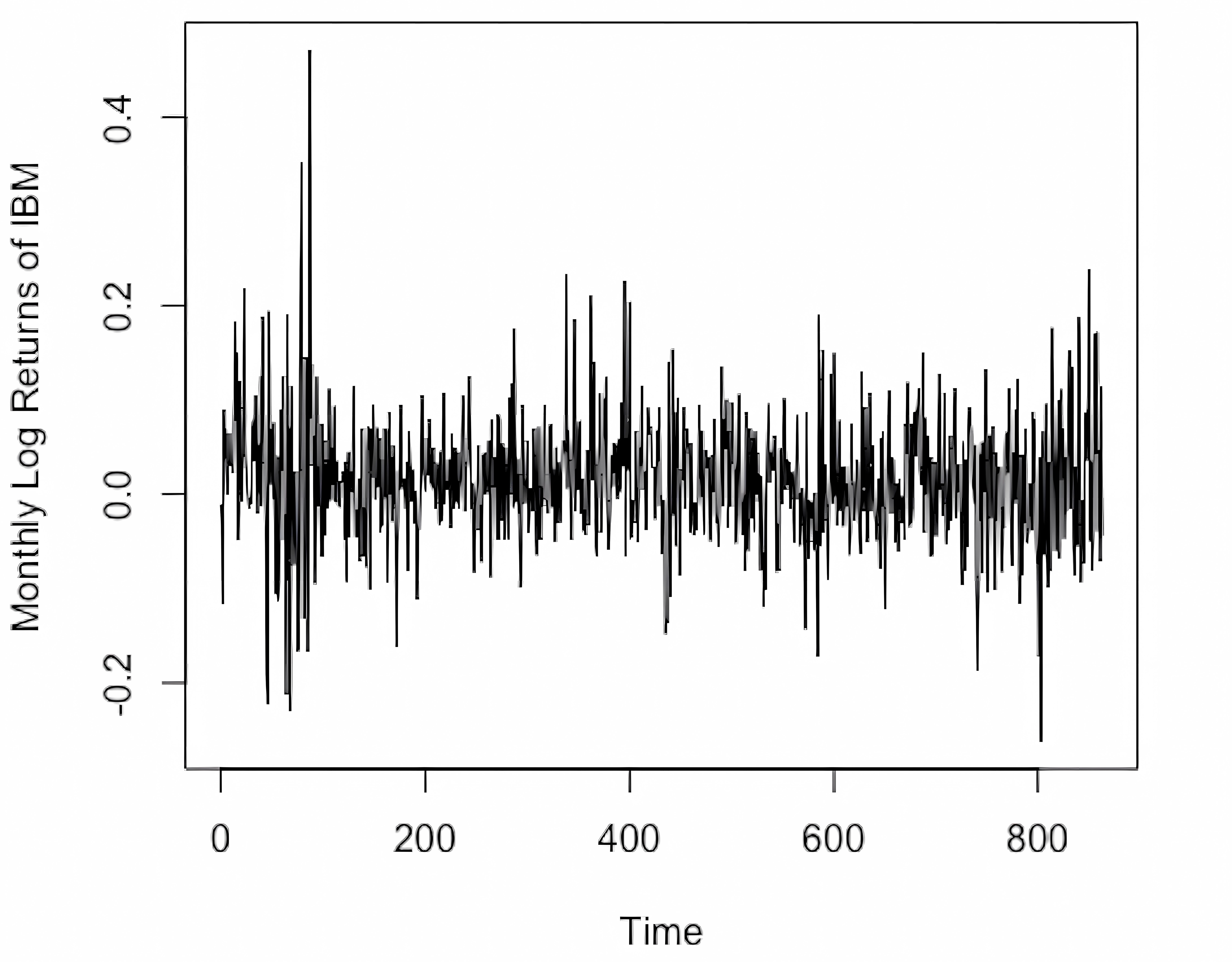}\\
\end{minipage}}\\
	\subfloat[Percentage Changes of Exchange Rate between Mark and Dollar]{
	\begin{minipage}[b]{0.4\textwidth}
		\centering
		\label{series_Exch_Rate}
		\includegraphics[width=1\textwidth]{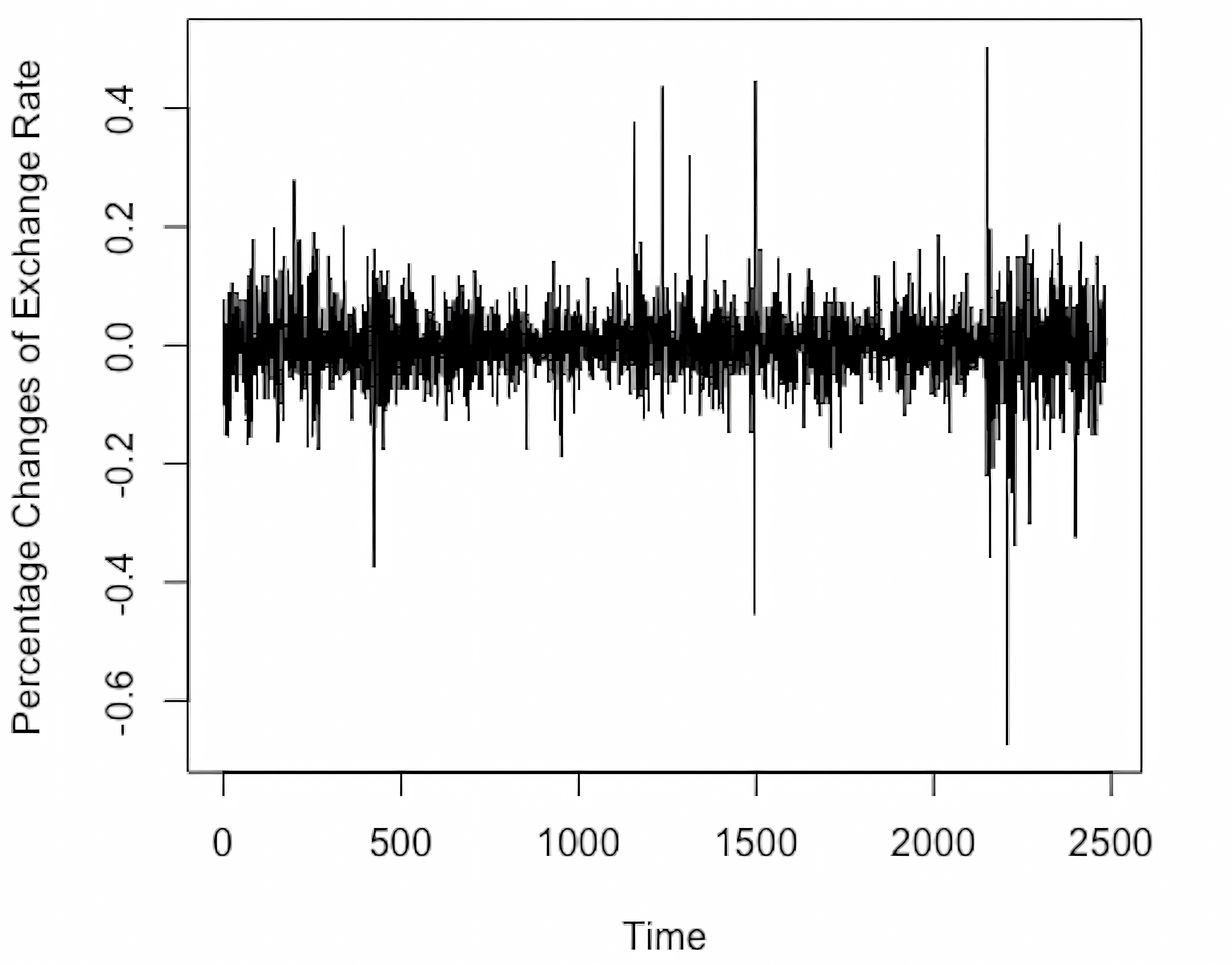}\\
\end{minipage}}
	\caption{Time Plots of Three Databases}
	\label{series}
\end{figure}

\begin{figure}[!]
	\centering
	\subfloat[]{
		\begin{minipage}[b]{0.4\textwidth}
			\centering
			\label{ACF_s_CREF}
			\includegraphics[width=1\textwidth]{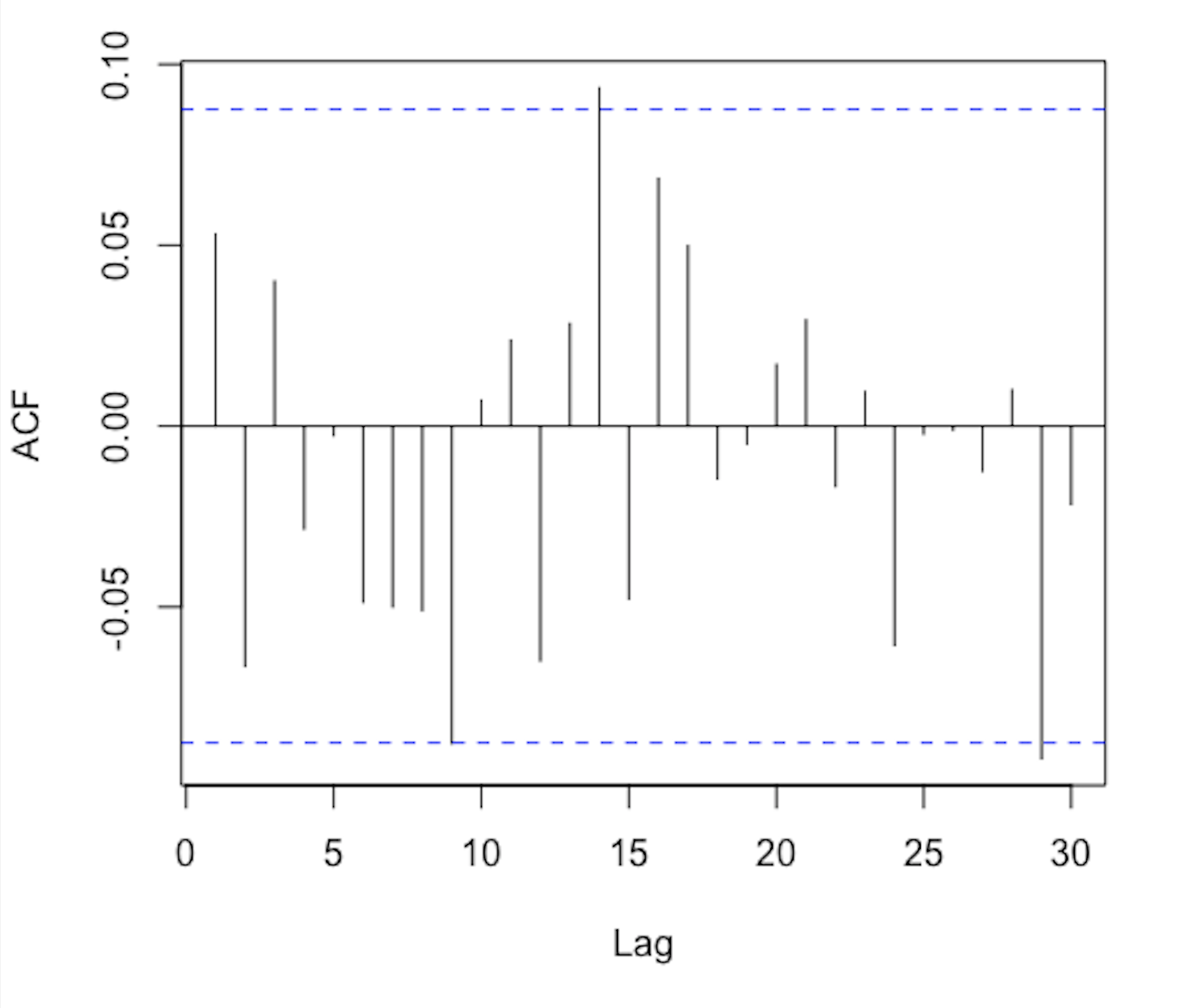}\\
	\end{minipage}}
	\subfloat[]{
		\begin{minipage}[b]{0.4\textwidth}
			\centering
			\label{ACF_s2_CREF}
			\includegraphics[width=1\textwidth]{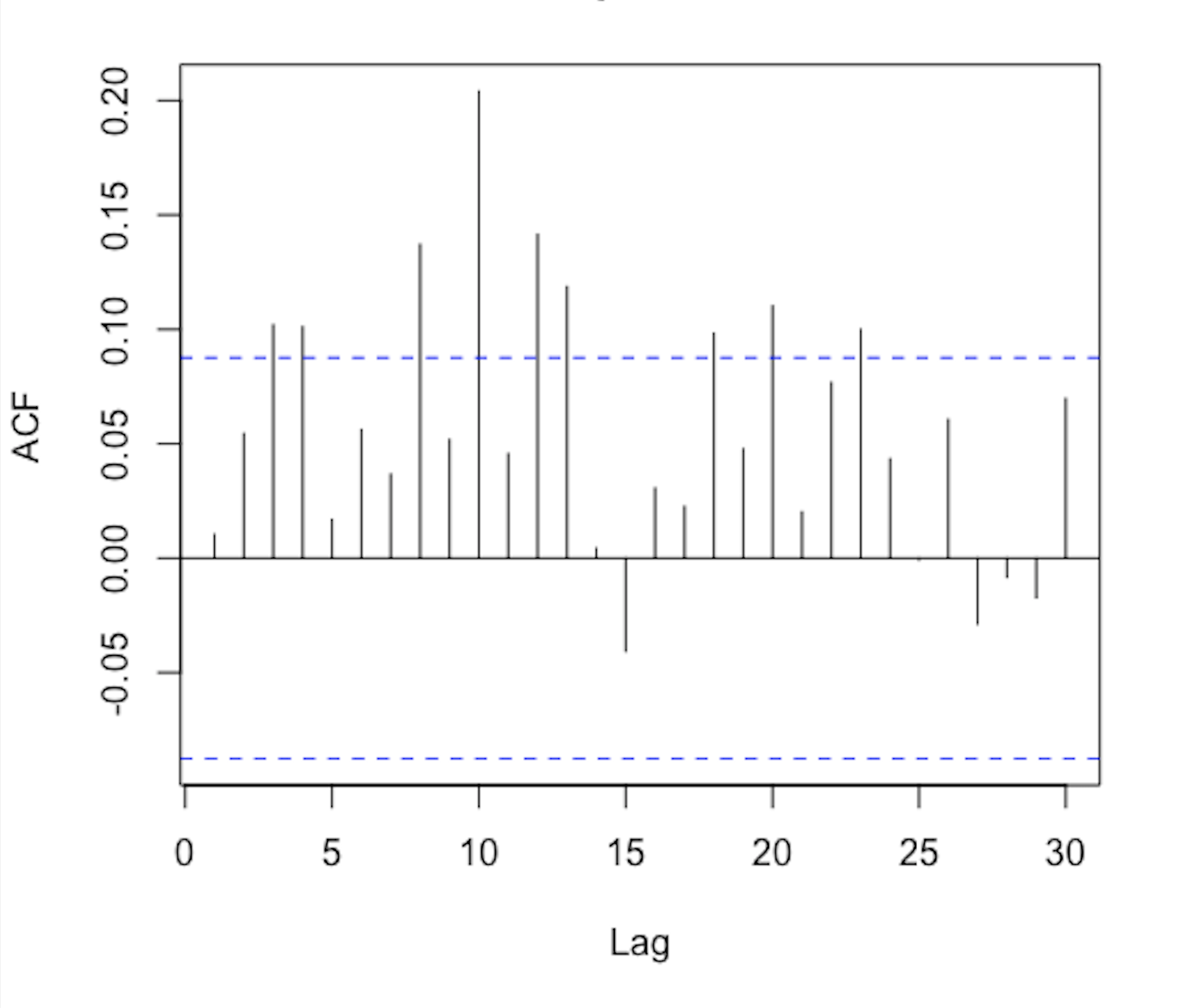}\\
	\end{minipage}}\\
	\subfloat[]{
		\begin{minipage}[b]{0.4\textwidth}
			\centering
			\label{ACF_s_IBM}
			\includegraphics[width=1\textwidth]{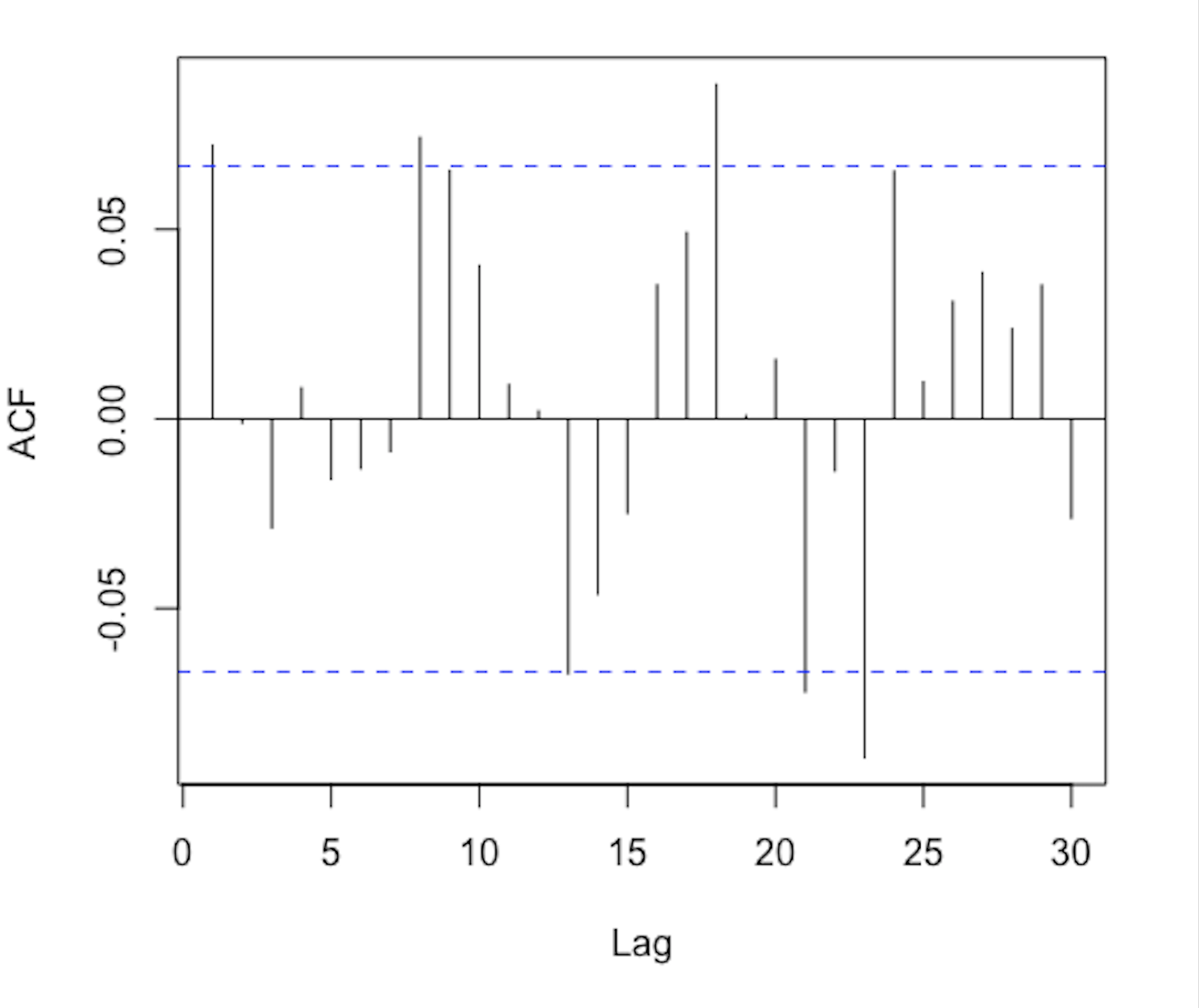}\\
	\end{minipage}}
	\subfloat[]{
	\begin{minipage}[b]{0.4\textwidth}
		\centering
		\label{ACF_s2_IBM}
		\includegraphics[width=1\textwidth]{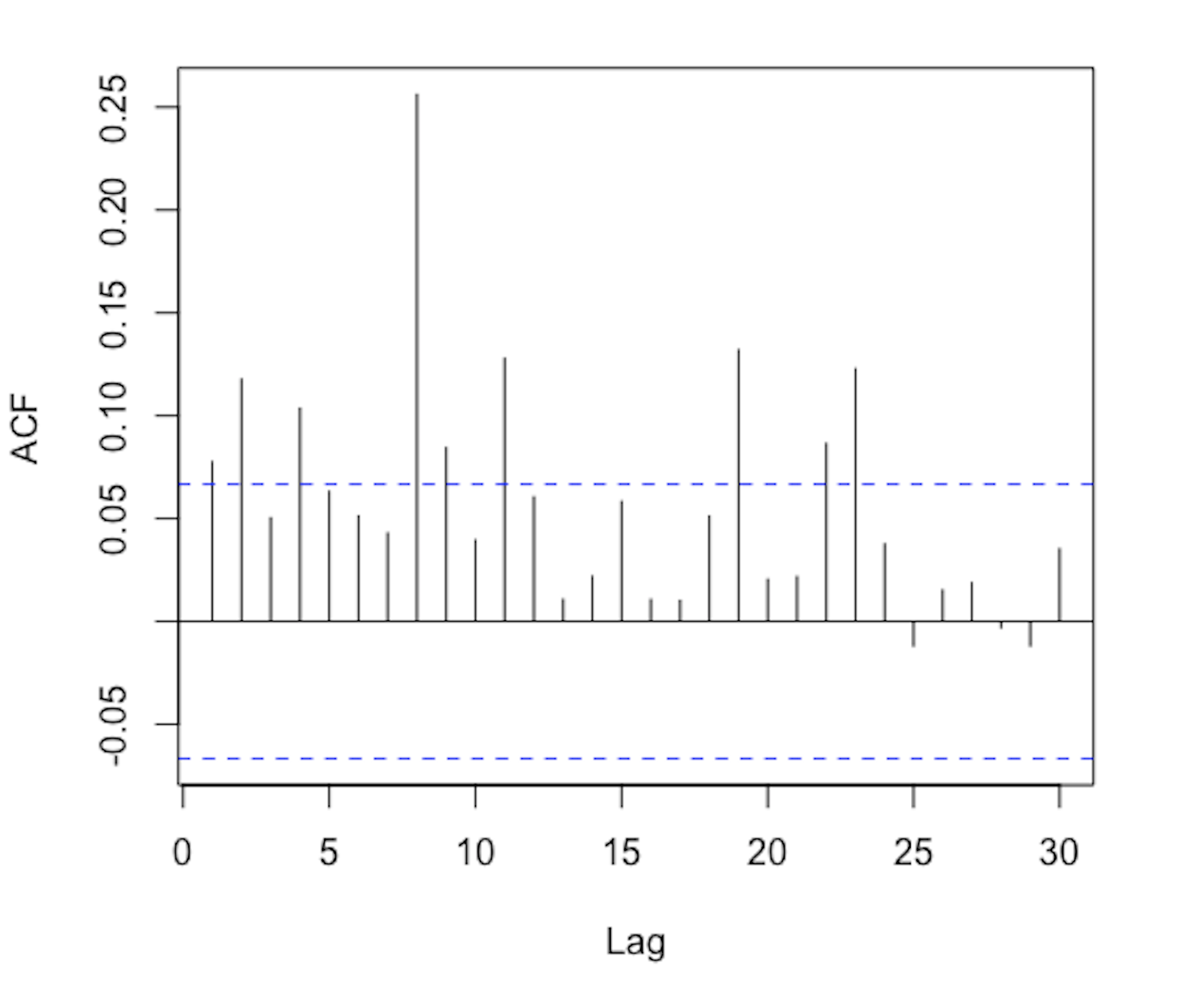}\\
\end{minipage}}\\
	\subfloat[]{
	\begin{minipage}[b]{0.4\textwidth}
		\centering
		\label{ACF_s_exch}
		\includegraphics[width=1\textwidth]{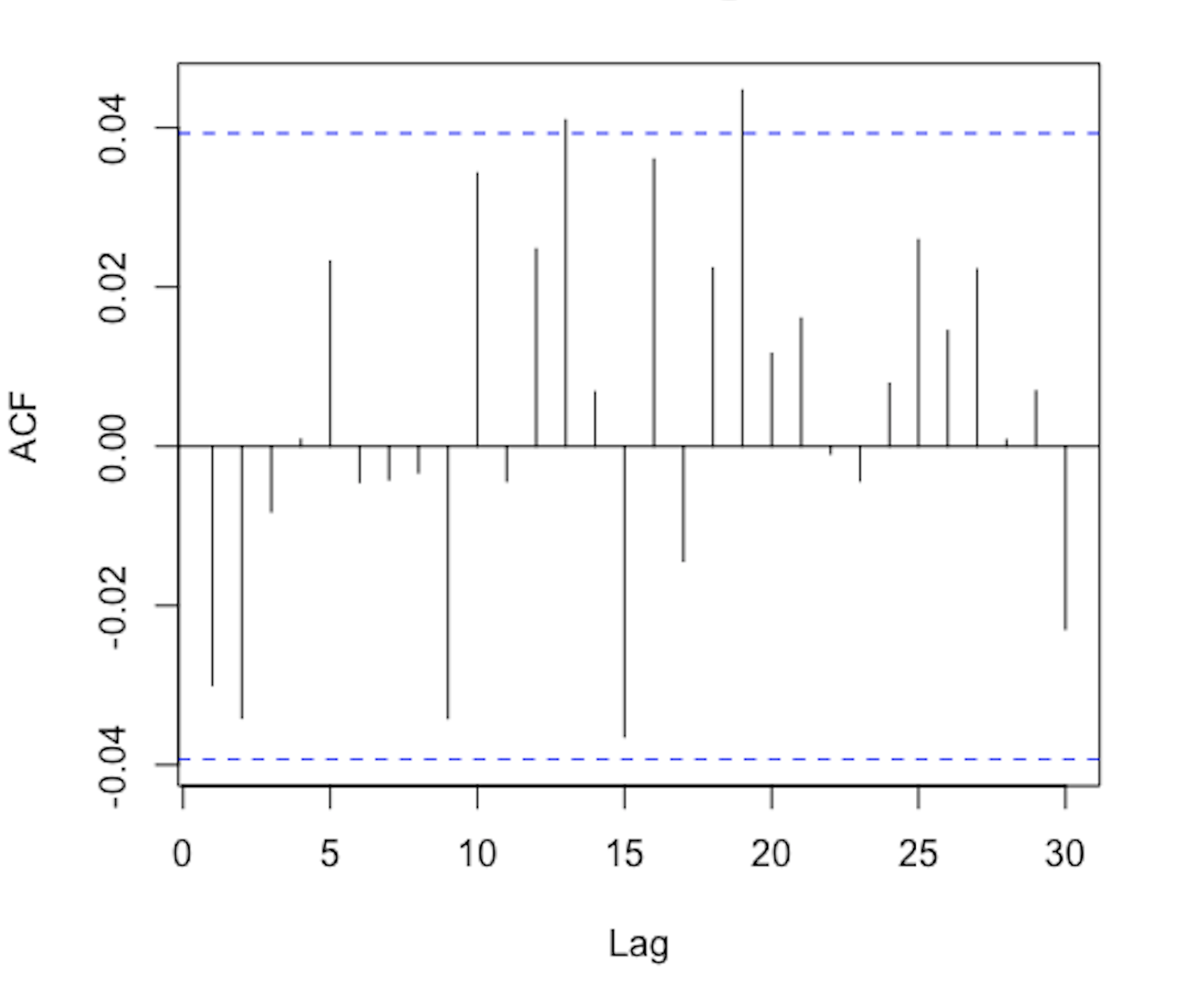}\\
\end{minipage}}
\subfloat[]{
	\begin{minipage}[b]{0.4\textwidth}
		\centering
		\label{ACF_s2_exch}
		\includegraphics[width=1\textwidth]{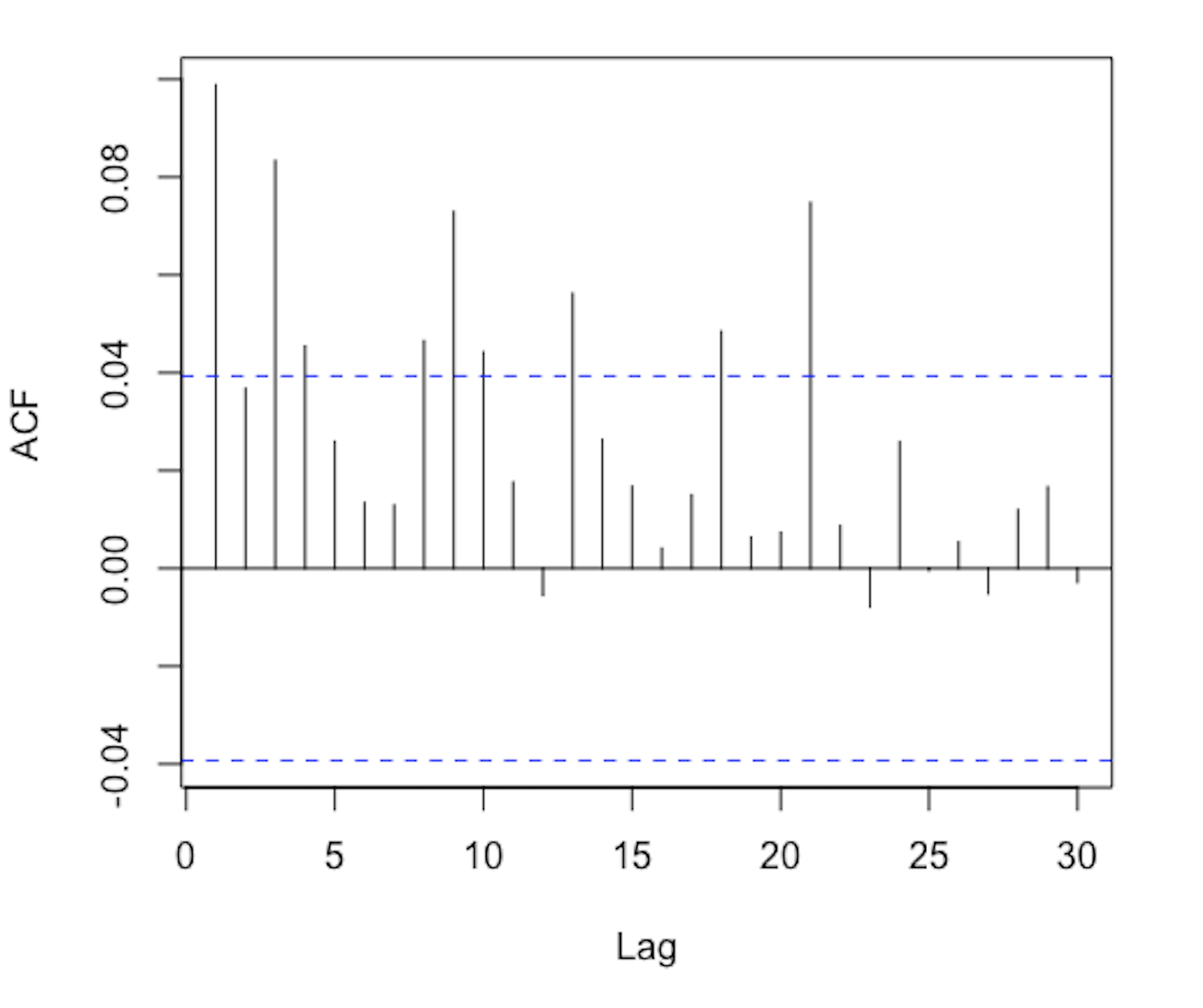}\\
\end{minipage}}
	\caption{Sample ACF and PACF of various functions of three datasets: (a) ACF of the log returns of CREF, (b) ACF of the squared log returns of CREF, (c) ACF of the log returns of IBM, (d) ACF of the squared log returns of IBM, (e) ACF of the exchange rate, (f) ACF of the squared exchange rate. }
	\label{ACF_series}
\end{figure}

According to the  Chapter 12 of Cryper and Chan (2008), the log returns of CREF can be well fitted by the GARCH(1,1) model. And in Chapter 3 of Tsay (2005), it has been demonstrated that the Log Returns of IBM and the Percentage Changes of Exchange Rate between Mark and Dollar can be well fitted by the EGARCH(1,1) and the ARCH(3) respectively.
%\begin{equation}
%\sigma^2_{t|t-1}=a_0+a_1R_{t-1}^2+b_1\sigma^2_{t-1|t-2}, 
%\end{equation} 
%where $\sigma^2_{t|t-1}=\mathbf{E}(R_t^2|R_{t-1},R_{t-2},\cdots)$ and $\sigma^2_{t-1|t-2}=\mathbf{E}(R_{t-1}^2|R_{-2},R_{t-3},\cdots)$.

Since the model specification and the model fitting are not our research subject here, we won't expand and only show the maximum likelihood estimations of parameters contained by GARCH(1,1), EGARCH(1,1) and ARCH(3) in Table \ref{estimation_models}, the ARCH effect test results for the estimate residuals and the squared estimate residuals in Table \ref{ARCH_test}, and the ACF of the estimate residuals and the squared estimate residuals in Figure \ref{ACF_res_2}. According to Table \ref{ARCH_test}, we find that, for three models we considered, the estimate residuals doesn't include ARCH effect any more. Besides, Figure \ref{ACF_res_2} provides the sample ACFs of the standardized residuals and the squared standardized residuals. These ACFs fail to suggest any significant serial correlation of conditional heteroscedasticity in the standardized residual series. Thus, these models appear to be adequate in describing the dependence in the volatility series. 

In the following, we focus on comparing the different performances of BDS test and  RBDS test in evaluating the goodness-of-fit of the above three models and comparing the results with the conclusions given by Cryper and Chan (2008) and Tsay (2005). In Table \ref{BDS_RBDS_GFT}, we presents the values of BDS and RBDS tests when detecting the goodness-of-fit for the above three models. The results show that at 0.05 significant level, the RBDS test performs better than the BDS test in the case where the sample size is hundreds. Specifically, for the daily log returns of CREF containing 500 observations and the monthly log returns of IBM containing 864 observations, RBDS test believes that fitting them with the GARCH(1,1) model and EGARCH(1,1) model respectively is sufficient, which is consistent with the views of Cryper and Chan (2008) and Tsay (2005).  But the BDS test give some opposite conclusions. Besides, it is worth mentioning that the results in Table \ref{BDS_RBDS_GFT} also show that when testing the goodness-of-fit of a model with such sample size, taking the logarithm of the estimate residuals can effectively reduce the parameter estimation impact of the BDS test, thereby making the BDS test give more reasonable results.
Moreover, according to the results in Table \ref{BDS_RBDS_GFT}, both BDS test and RBDS test agree that the ARCH(3) model can fully fit the percentage changes of the exchange rate containing 2488 observations, which is consistent with Tsay (2005). This further verifies that when the sample size is as large as several thousand, the over-rejection problem of BDS test is weakened to disappear and it has comparable performance to the RBDS test.

% Please add the following required packages to your document preamble:
% \usepackage{booktabs}
% \usepackage{graphicx}
\begin{table}[]
	\centering
	\caption{Estimated Coefficients of Three Models}
	\label{estimation_models}
	\begin{threeparttable}
	%\resizebox{\textwidth}{!}{%
		\begin{tabular}{@{}cccccc@{}}
			\toprule[2pt]
			\multicolumn{6}{c}{\begin{tabular}[c]{@{}c@{}}Estimated GARCH(1,1) Model of Daily Log Returns of \\ the CREF: August 2004 - August 2006\end{tabular}} \\ \midrule
			Coefficients & Estimation  & Std         & t-value  & p-value    & Significant Level \\ \midrule
			$a_0$        & $0.01633$   & $0.01237$   & $1.320$  & $0.1869$   &                   \\
			$a_1$        & $0.04414$   & $0.02097$   & $2.105$  & $0.0353$   & *                 \\
			$b_1$        & $0.91704$   & $0.04570$   & $20.066$ & $2e-16$    & ***               \\ \midrule[1.5pt]
			\multicolumn{6}{c}{\begin{tabular}[c]{@{}c@{}}Estimated EGARCH(1,1) Model of Monthly Log Returns of \\ the IBM: January 1926 - December 1997\end{tabular}} \\ \midrule
			Coefficients & Estimation  & Std         & t-value  & p-value    & Significant Level \\ \midrule
			$\mu$        & $0.0128351$ & $0.0021233$ & $6.045$  & $1.49e-09$ & ***               \\
			$\omega$     & $0.0003624$ & $0.0001440$ & $2.517$  & $0.0119$  & *                 \\
			$\alpha_1$   & $0.0895498$ & $0.0275172$ & $3.254$  & $0.0011$  & **                \\
			$\gamma_1$   & $0.2658391$ & $0.1237391$ & $2.148$  & $0.0317$  & *                 \\
			$\beta_1$    & $0.8261261$ & $0.0500977$ & $16.490$ & $<2e-16$   & ***               \\
			$\theta$     & $1.4662640$ & $0.0916414$ & $16.000$ & $<2e-16$   & ***               \\ \midrule[1.5pt]
			\multicolumn{6}{c}{\begin{tabular}[c]{@{}c@{}}Estimated ARCH(3) Model of Percentage Changes of the Exchange Rate \\ between Mark and Dollar in 10-minutes Intervals: June 5, 1989 - June 19, 1989\end{tabular}} \\ \midrule
			Coefficients & Estimation  & Std         & t-value  & p-value    & Significant Level \\ \midrule
			$a_0$        & $2.237e-03$ & $4.587e-05$ & $48.765$ & $<2e-16$   & ***               \\
			$a_1$        & $3.283e-01$ & $1.629e-02$ & $20.146$ & $<2e-16$   & ***               \\
			$a_2$        & $7.301e-02$ & $1.596e-02$ & $4.575$  & $4.75e-06$ & ***               \\
			$a_3$        & $1.026e-01$ & $1.469e-02$ & $6.986$  & $2.82e-12$ & ***               \\ 
 \bottomrule[2pt]
		\end{tabular}%
%	}
\begin{tablenotes}
	\footnotesize
\item[-] Notes: 0 '***', 0.001 '**', 0.01 '*', 0.05 '.', 0.1 ' '.
\end{tablenotes}
\end{threeparttable}
\end{table}

% Please add the following required packages to your document preamble:
% \usepackage{booktabs}
\begin{table}[]
	\centering
	\caption{ARCH Effect Test Results of the Estimate Residuals of 
		the Three Models}
	\label{ARCH_test}
	\begin{threeparttable}
	\begin{tabular}{@{}cccc@{}}
		\toprule[1pt]
		& GARCH(1,1)        & EGARCH(1,1)      & ARCH(3)           \\ \midrule
		$\hat{\epsilon}_t$   & $13.465(0.3362)$  & $10.664(0.5579)$ & $15.740(0.2035)$   \\
		$\hat{\epsilon}^2_t$ & $21.872(0.0390)$ & $10.148(0.6029)$ & $0.299(0.9998)$ \\  \bottomrule[1pt]
	\end{tabular}
\begin{tablenotes}
	\footnotesize
	\item[-] Shown in the table are the values of the LM test ($df=12$), with the corresponding P-values in parentheses.
	\item[-] $\hat{\epsilon}_t$ is the estimation residuals of the model. $\hat{\epsilon}^2_t$ is the square of $\hat{\epsilon}_t$.
\end{tablenotes}
\end{threeparttable}
\end{table}

\begin{figure}[!]
	\centering
	\subfloat[]{
	\begin{minipage}[b]{0.4\textwidth}
		\centering
		\label{ACF_res_GARCH}
		\includegraphics[width=1\textwidth]{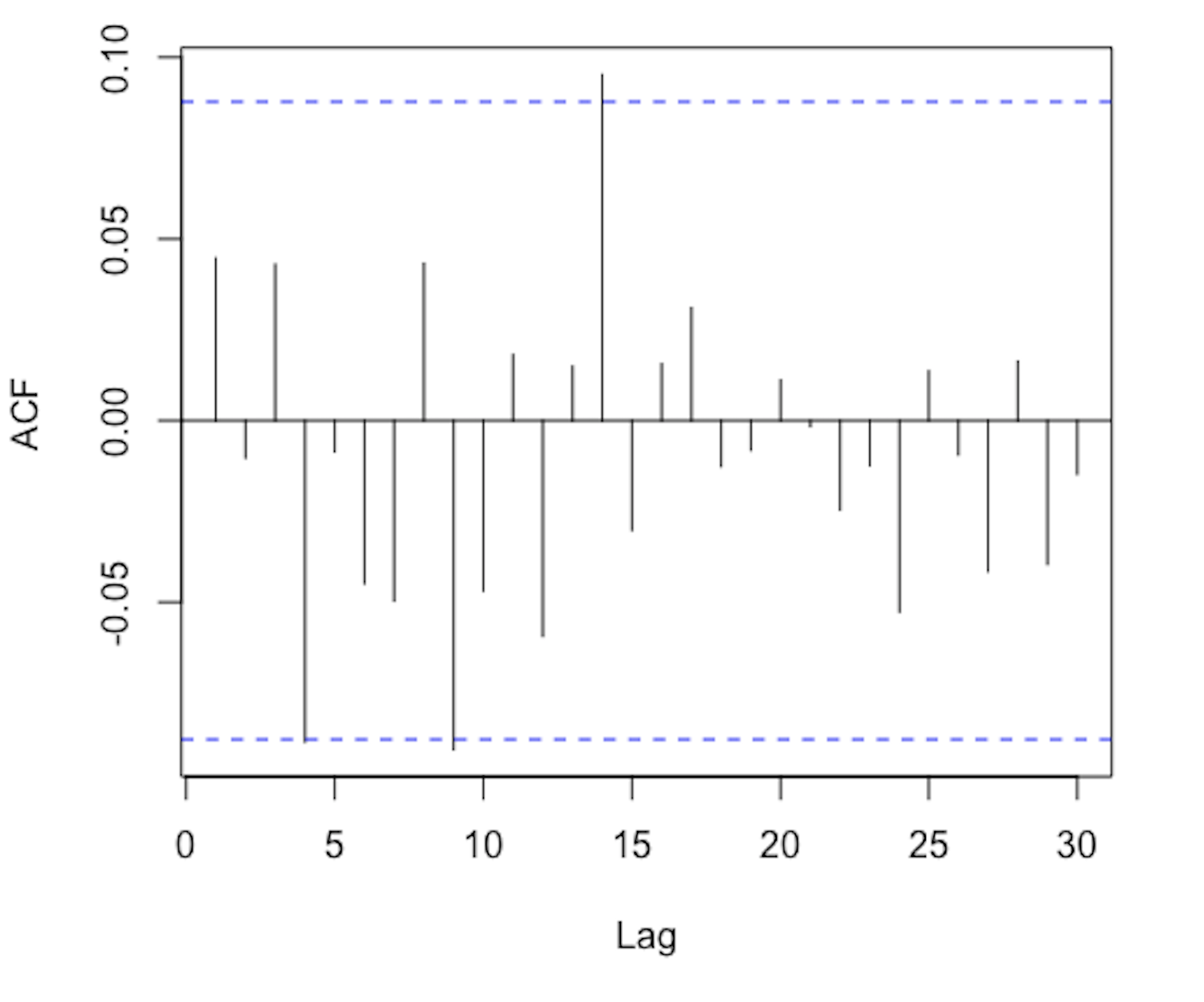}\\
\end{minipage}}
\subfloat[]{
	\begin{minipage}[b]{0.4\textwidth}
		\centering
		\label{ACF_res2_GARCH}
		\includegraphics[width=1\textwidth]{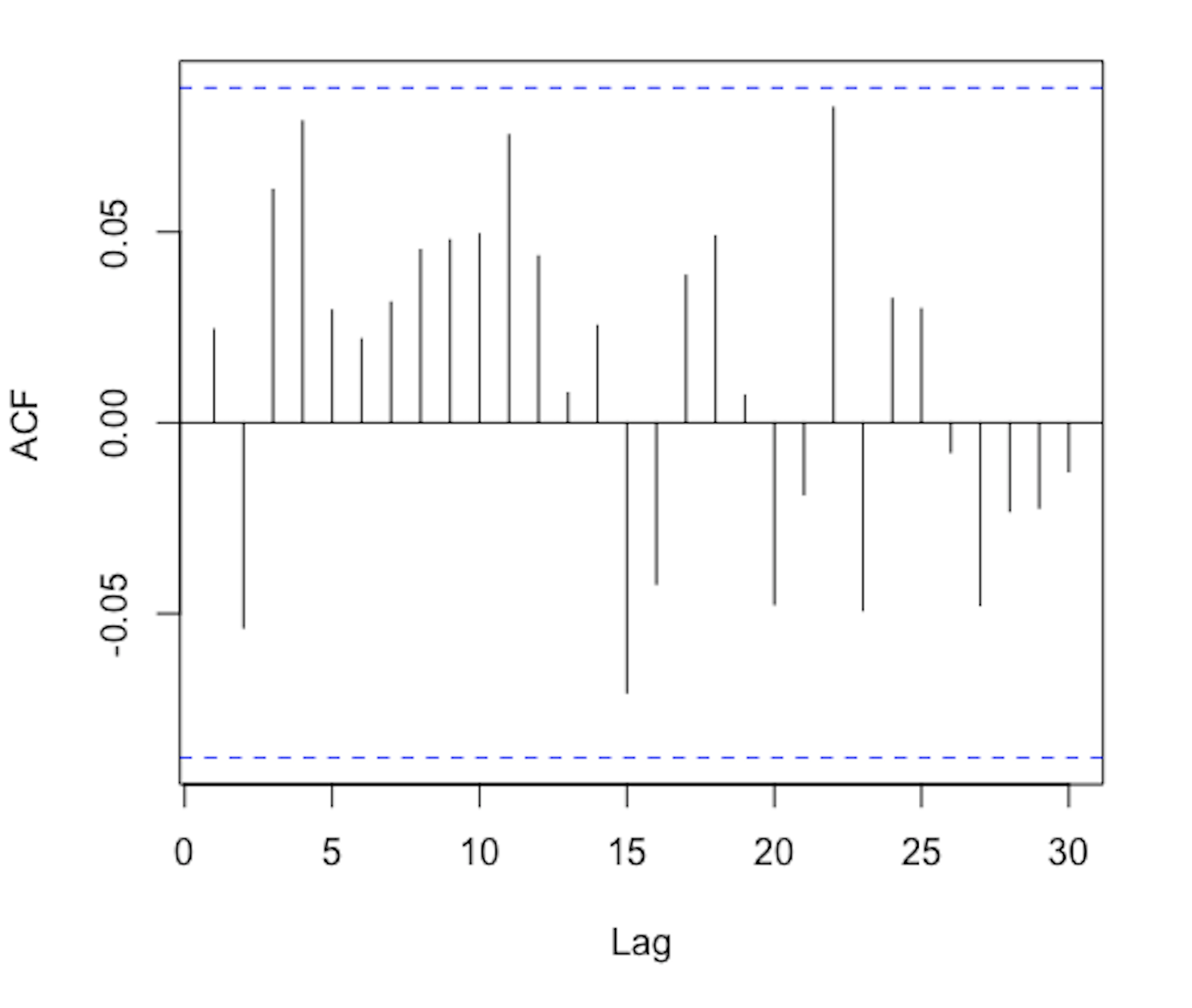}\\
\end{minipage}}\\
\subfloat[]{
	\begin{minipage}[b]{0.4\textwidth}
		\centering
		\label{ACF_res_EGARCH}
		\includegraphics[width=1\textwidth]{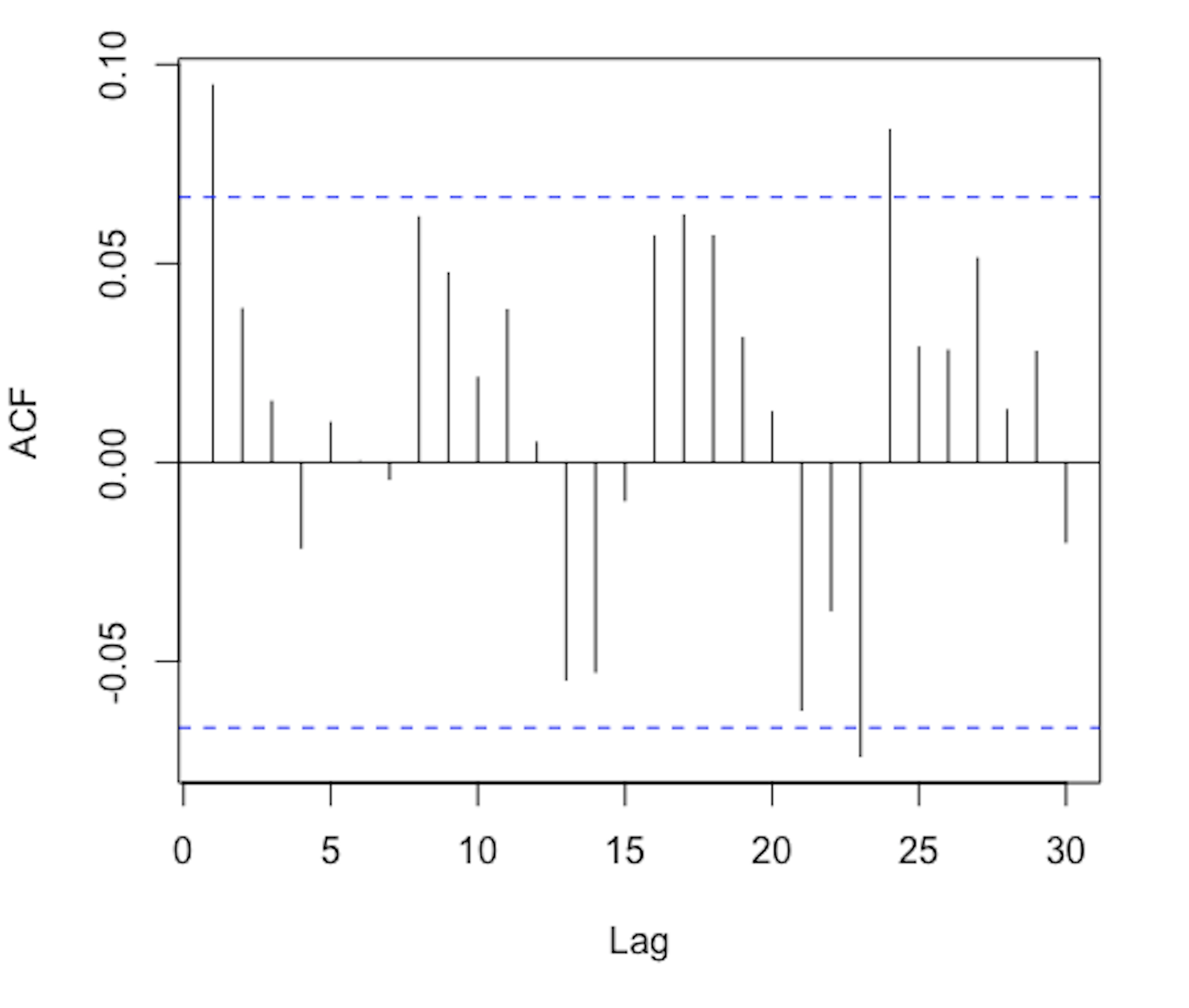}\\
\end{minipage}}
\subfloat[]{
	\begin{minipage}[b]{0.4\textwidth}
		\centering
		\label{ACF_res2_EGARCH}
		\includegraphics[width=1\textwidth]{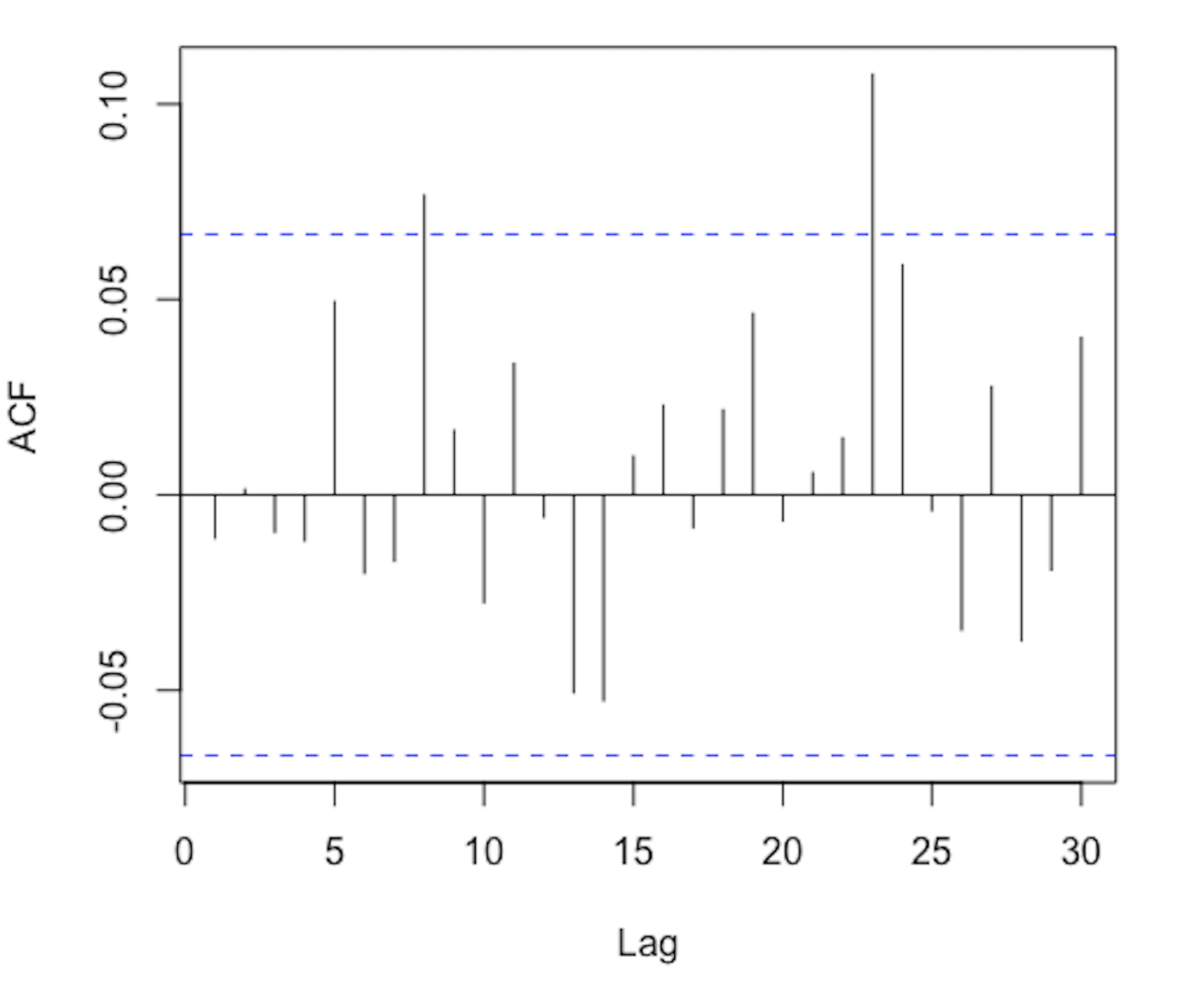}\\
\end{minipage}}\\
\subfloat[]{
	\begin{minipage}[b]{0.4\textwidth}
		\centering
		\label{ACF_res_ARCH}
		\includegraphics[width=1\textwidth]{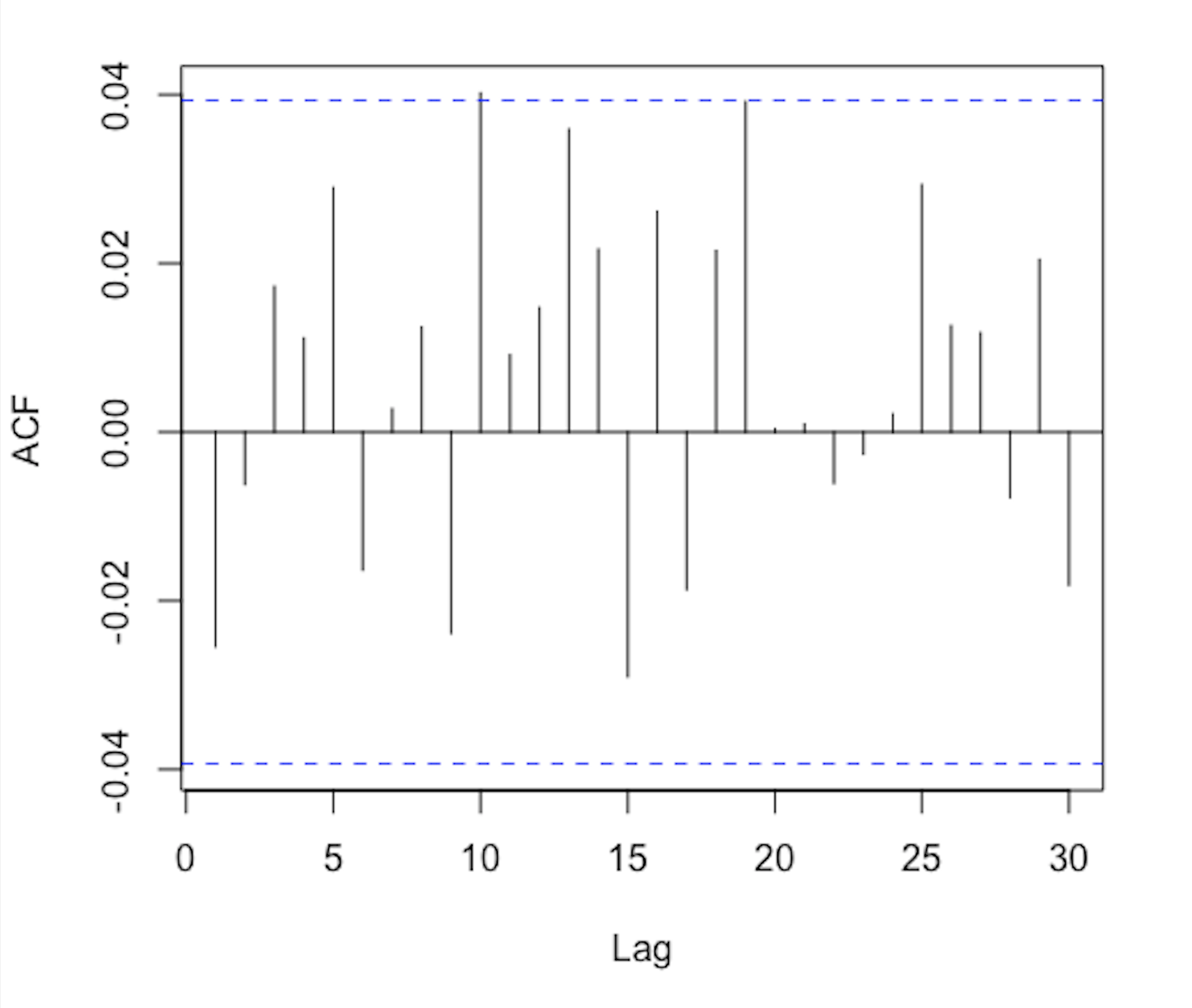}\\
\end{minipage}}
\subfloat[]{
	\begin{minipage}[b]{0.4\textwidth}
		\centering
		\label{ACF_res2_ARCH}
		\includegraphics[width=1\textwidth]{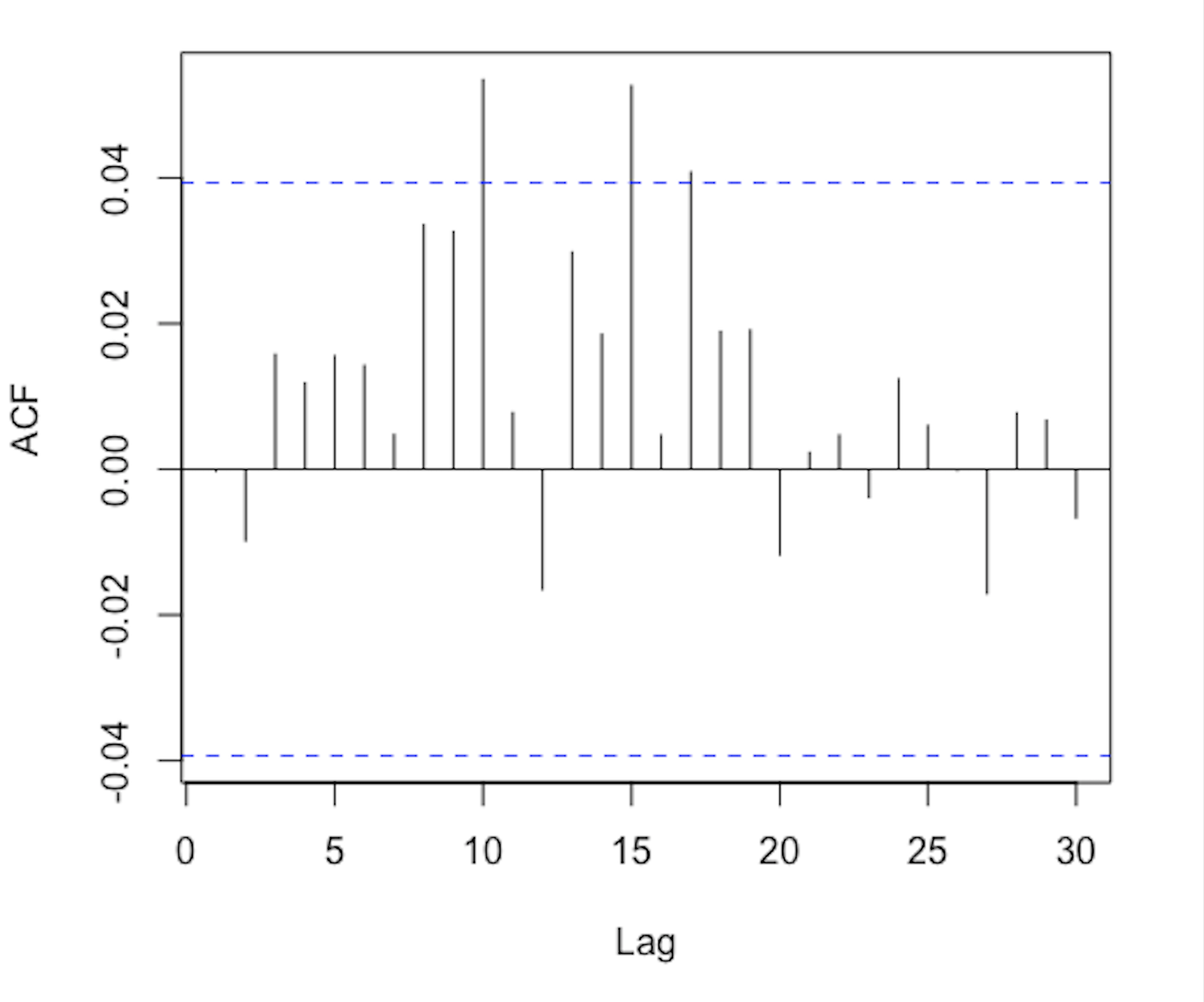}\\
\end{minipage}}
	\caption{Model checking of GARCH(1,1), EGARCH(1,1) and ARCH(3): (a) ACF of standardized residuals of GARCH(1,1), (b) ACF of squared standardized residuals of GARCH(1,1), (c) ACF of standardized residuals of EGARCH(1,1), (d) ACF of standardized squared residuals of EGARCH(1,1), (e) ACF of standardized residuals of ARCH(3), (f) ACF of standardized squared residuals of ARCH(3).}
\label{ACF_res_2}
\end{figure}

% Please add the following required packages to your document preamble:
% \usepackage{booktabs}
% \usepackage{multirow}
\begin{table}[]
	\centering
	\caption{The Values of BDS Test and RBDS Test for Evaluating the Goodness-of-Fit of Three Models}
	\label{BDS_RBDS_GFT}
	\begin{threeparttable}
	\begin{tabular}{@{}cc|ccc|cc@{}}
		\toprule[1.5pt]
		\multicolumn{3}{c}{\multirow{2}{*}{}} & \multicolumn{2}{c|}{$\hat{\epsilon}_t$} & \multicolumn{2}{c}{$\log(\hat{\epsilon}_t^2)$} \\ \cmidrule(l){4-7} 
		\multicolumn{3}{c}{}                 & BDS       & RBDS     & BDS       & RBDS      \\
		\cmidrule(l){4-7}
		\multirow{2}{*}{GARCH(1,1)} &\multirow{2}{*}{$T=500$} & $m=2$ & $2.3009$  & $1.7632$ & $-1.1988$ & $-0.9190$ \\
		& & $m=3$ & $2.0807$  & $1.1576$ & $-1.3627$ & $-1.0441$ \\ \midrule[1pt]
		\multirow{2}{*}{EGARCH(1,1)} 
	 &\multirow{2}{*}{$T=864$}& $m=2$ & $1.7990$  & $1.4461$ & $-0.6990$ & $-0.6395$ \\
		& & $m=3$ & $2.0186$  & $1.5901$ & $-0.9484$ & $-0.8610$ \\ \midrule[1pt]
		\multirow{2}{*}{ARCH(3)}  &\multirow{2}{*}{$T=2488$}   & $m=2$ & $-0.4219$ & $0.4160$ & $-$       & $-$       \\
		& & $m=3$ & $0.7423$  & $0.7060$ & $-$       & $-$       \\ \bottomrule[1.5pt]
	\end{tabular}
\begin{tablenotes}
	\footnotesize
	\item[-] $T$ is the length of the sequence $\{\hat{\epsilon}_t\}$.
	\item[-] $\hat{\epsilon}_t$ is the standardized estimate residuals of the specified model, $\log(\hat{\epsilon}_t^2)$ is logarithm of the squared $\hat{\epsilon}_t$.
\end{tablenotes}
\end{threeparttable}
\end{table}

\section{Conclusions}

In this paper, we are committed to improving the over-rejection problem of BDS test. Since BDS test originated from the correlation integral $C_{m,T}(\epsilon)$ in the chaos theory, we found that treating $C_{m,T}(\epsilon)$ as a U-statistic to obtain the BDS test is the inducement for the over-rejection problem. We give the exact expectation of the correlation integral $C_{m,T}(\epsilon)$ and recalculate the asymptotic variance of $C_{m,T}(\epsilon)$. And based on the modified asymptotic theory of the correlation integral, a revised BDS (RBDS) test was proposed. The RBDS test is still a nonparametric test for independence, also comes from the correlation integral. Therefore, it has the same advantages as the BDS test. What's more, from the results of the simulation experiments we designed, the RBDS test effectively eliminates the over-rejection problem of BDS test. Similar to BDS test, RBDS test is also affected by the estimation of model parameters, resulting in size distortion, which can be alleviated by logarithmic transformation preprocessing of the estimate residuals of the model. Besides, empirical analysis shows that RBDS test is more reliable in evaluating the goodness-of-fit of the model, especially when sample size is hundreds, and BDS test is more likely to cause overfitting of the model. Therefore, we suggest that, in practice, multiple methods should be combined in evaluating the goodness-of-fit for the model to prevent misjudgment caused by using a single method.

\section{Acknowledgments}

We thank two anonymous referees for helpful comments and suggestions and Z.D. Bai for theoretical guidance about this topic. Funding for this work was provided by National Science Foundation of China (Grant no. 12271536, 12171198), First Class Discipline of Zhejiang-A (Zhejiang University of Finance and Economics - Statistics) (Grant no. 10344921011/003).

\section{Appendix A}
\begin{proof}[Proof of Theorem \ref{ThCLTCI}:]
	
	Since the correlation integral $C_{m,T}(\epsilon)$ can be decomposed according to \eqref{DivitionCI}, the expectation of $C_{m,T}(\epsilon)$ can be obtained by calculating the expectation of $\breve{C}_{m,T}(\epsilon)$ and $\widetilde{C}_{m,T}(\epsilon)$ respectively. For $\widetilde{C}_{m,T}(\epsilon)$, when $\{u_{t}\}$ is a sequence of i.i.d. random variables, its expectation can be easily  figured out
	\begin{equation}\label{EtildeCI}
		\mathbf{E}\widetilde{C}_{m,T}(\epsilon)=\left(\omega_{1}^0\right)^m=W_m(m,0).
	\end{equation}
	And for the other part, after some routine calculation, we get
	\begin{equation}\label{EbreveCI}
		\mathbf{E}\breve{C}_{m,T}(\epsilon)= \frac{1}{N}\sum\limits_{k=1}^{m-1}(T_{m}-k) \left(\omega_{h}^0\right)^{k-i}\left(\omega_{h+1}^0\right)^{i}=\frac{1}{N}\sum_{k=1}^{m-1}(T_{m}-k)W_m(k,0),
	\end{equation}
	where $h=\left\lfloor{\frac{m}{k}}\right\rfloor$, $i=m-hk$.
	Further, we get \eqref{ECI}.
	
	Next, the variance of $C_{m,T}(\epsilon)$ is our focus. Also according to the decomposition in \eqref{DivitionCI}, we have:
	\begin{equation}\label{VarCI}
		\begin{aligned}
			\mathbf{Var}(C_{m,T}(\epsilon))&=\mathbf{E}\left[C_{m,T}(\epsilon)-\mathbf{E}C_{m,T}(\epsilon)\right]^2\\ &=\mathbf{E}\left[\breve{C}_{m,T}(\epsilon)-\mathbf{E}\breve{C}_{m,T}(\epsilon) +\frac{N_{0}}{N}\left(\widetilde{C}_{m,T}(\epsilon)-\mathbf{E}\widetilde{C}_{m,T}(\epsilon) \right)\right]^2\\
			&=\uppercase\expandafter{\romannumeral1}_{m,m}  +2\left(\frac{N_{0}}{N}\right)\uppercase\expandafter{\romannumeral2}_{m,m} +\left(\frac{N_{0}}{N}\right)^2\uppercase\expandafter{\romannumeral3}_{m,m},
		\end{aligned}
	\end{equation}
	where
	\begin{eqnarray}\label{VarCII}
		\uppercase\expandafter{\romannumeral1}_{m,m}&=& \mathbf{E}\left[\breve{C}_{m,T}(\epsilon) -\mathbf{E}\breve{C}_{m,T}(\epsilon)\right]^2,\qquad \qquad \qquad \qquad \qquad \qquad \qquad \qquad \qquad
	\end{eqnarray}
	\begin{eqnarray}\label{VarCIII}
		\uppercase\expandafter{\romannumeral2}_{m,m}&=&\mathbf{E} \left[\breve{C}_{m,T}(\epsilon)-\mathbf{E}\breve{C}_{m,T}(\epsilon)\right] \left[\widetilde{C}_{m,T}(\epsilon)-\mathbf{E}\widetilde{C}_{m,T}(\epsilon)\right],\qquad \qquad \qquad \qquad
	\end{eqnarray}
	\begin{eqnarray}\label{VarCIIII}
		\uppercase\expandafter{\romannumeral3}_{m,m} &=&\mathbf{E}\left[\widetilde{C}_{m,T}(\epsilon)-\mathbf{E}\widetilde{C}_{m,T}(\epsilon)\right]^2. \qquad \qquad \qquad \qquad \qquad \qquad \qquad \qquad \quad
	\end{eqnarray}
	Among them, $\uppercase\expandafter{\romannumeral3}_{m,m}= \widetilde{\sigma}_{m}^2+O(T^{-3})$, where $\widetilde{\sigma}_{m}^2$ is equal to the variance part of Theorem 2.1. in Luo et al.(2020). Here we only need to care about the results of $\uppercase\expandafter{\romannumeral1}_{m,m}$ and $\uppercase\expandafter{\romannumeral2}_{m,m}$ respectively.
	
	For $\uppercase\expandafter{\romannumeral1}_{m,m}$, according to \eqref{breveCI} and \eqref{EbreveCI}, we have \begin{equation}\label{VarCIIC}
		\begin{aligned}
			\uppercase\expandafter{\romannumeral1}_{m,m}&= \mathbf{E}\left[\breve{C}_{m,T}(\epsilon) -\mathbf{E}\breve{C}_{m,T}(\epsilon)\right]^2\\
			&=\frac{1}{N^2}\mathbf{E}\left(  \sum_{k=1}^{m-1}\sum_{t=1}^{T_{m}-k}\left[I_{\epsilon}(\|Y_{t}^m-Y_{t+k}^m\|) -W_m(k,0)\right]\right)^2\\
			&=\frac{1}{N^2}\mathbf{E}\left(\sum_{k_{1}=1}^{m-1}\sum_{t_{1}=1}^{T_{m}-k_1} \left[I_{\epsilon}(\|Y_{t_{1}}^m-Y_{t_{1}+k_{1}}^m\|) -W_m(k_1,0)\right]\right)\\ &\quad \quad \quad \times{\left(\sum_{k_{2}=1}^{m-1}\sum_{t_{2}=1}^{T_{m}-k_2} \left[I_{\epsilon}(\|Y_{t_{2}}^m-Y_{t_{2}+k_{2}}^m\|) -W_m(k_2,0)\right]\right)}\\
			&=\frac{1}{N^2}\sum_{k_{1}=1}^{m-1}\sum_{k_{2}=1}^{m-1}\psi(k_{1},k_{2}),
		\end{aligned}
	\end{equation}
	where
	\begin{eqnarray}\label{psi}
		\begin{aligned}
			\psi(k_{1},k_{2})&=\sum_{t_{1}=1}^{T_{m}-k_1} \sum_{t_{2}=1}^{T_{m}-k_2} \left[\mathbf{E}I_{\epsilon}(\|Y_{t_{1}}^m-Y_{t_{1}+k_{1}}^m\|) I_{\epsilon}(\|Y_{t_{2}}^m-Y_{t_{2}+k_{2}}^m\|)-W_m(k_1,0)W_m(k_2,0)
			\right],
		\end{aligned}
	\end{eqnarray}
	Easy to understand that  $\uppercase\expandafter{\romannumeral1}_{m,m}$ is determined by $\mathbf{E}I_{\epsilon}(\|Y_{t_1}^m-Y_{t_1+k}^m\|)I_{\epsilon}(\|Y_{t_2}^m-Y_{t_2+k_2}^m\|)$, and $\mathbf{E}I_{\epsilon}(\|Y_{t_1}^m-Y_{t_1+k}^m\|)I_{\epsilon}(\|Y_{t_2}^m-Y_{t_2+k_2}^m\|)$ is closely related to the dependence between $I_{\epsilon}(\|Y_{t_{1}}^m-Y_{t_{1}+k_{1}}^m\|)$ and $I_{\epsilon}(\|Y_{t_{2}}^m-Y_{t_{2}+k_{2}}^m\|)$. Since $\|\cdot \|$ represents the maximum norm, recalling the definition of $Y_l^m$ and \eqref{indicator} we have
	\begin{equation}\label{Yu}
		\begin{aligned}
			&I_{\epsilon}(\|Y_{t_{1}}^m-Y_{t_{1}+k_{1}}^m\|)= \prod_{\rho=0}^{m-1} I_{\epsilon}(|u_{t_{1}+\rho}-u_{t_{1}+k_{1}+\rho}|),\\ &I_{\epsilon}(\|Y_{t_{2}}^m-Y_{t_{2}+k_{2}}^m\|)= \prod_{\rho=0}^{m-1} I_{\epsilon}(|u_{t_{2}+\rho}-u_{t_{2}+k_{2}+\rho}|).
		\end{aligned}
	\end{equation}
	For fixed $k_{1}$ and $k_{2}$, all random variables involved in $I_{\epsilon}(\|Y_{t_{1}}^m-Y_{t_{1}+k_{1}}^m\|)$
	form the set $$\{u_{t_1},u_{t_1+1},\cdots,u_{t_1+k_1},\cdots,u_{t_1+k_1+m-1}\},$$
	and all random variables involved in $I_{\epsilon}(\|Y_{t_{2}}^m-Y_{t_{2}+k_{2}}^m\|)$ form the set
	$$\{u_{t_2},u_{t_2+1},\cdots,u_{t_2+k_2},\cdots,u_{t_2+k_2+m-1}\}.$$
	We correspond these two sets to the following two sets respectively:
	\begin{equation}\label{At1k1At2k2}
		\begin{aligned}
			&A_{t_{1},k_{1}}=\left\{t_{1},t_{1}+1,\cdots,t_{1}+k_{1}, \cdots, t_{1}+k_{1}+m-1\right\},\\
			&A_{t_{2},k_{2}}=\left\{t_{2},t_{2}+1,\cdots,t_{2}+k_{2}, \cdots, t_{2}+k_{2}+m-1\right\}.
		\end{aligned}
	\end{equation}
	So $I_{\epsilon}(\|Y_{t_{1}}^m-Y_{t_{1}+k_{1}}^m\|)$ and  $I_{\epsilon}(\|Y_{t_{2}}^m-Y_{t_{2}+k_{2}}^m\|)$ are independent if and only if $A_{t_{1},k_{1}}\cap{A_{t_{2},k_{2}}}={\emptyset}.$
	Further, we divide $\psi(k_1,k_2)$ into two parts as follows:
	\begin{equation}\label{Divitionpsi}
		\psi(k_{1},k_{2})=\psi_{1}(k_{1},k_{2})+\psi_{2}(k_{1},k_{2}),
	\end{equation}
	where
	\begin{equation*}
		\begin{aligned}
			\psi_{1}(k_{1},k_{2})&= \sum_{A_{t_{1},k_{1}}\cap{A_{t_{2},k_{2}}}={\emptyset}} \left[\mathbf{E} I_{\epsilon}(\|Y_{t_{1}}^m-Y_{t_{1}+k_{1}}^m\|) I_{\epsilon}(\|Y_{t_{2}}^m-Y_{t_{2}+k_{2}}^m\|)  -W_m(k_1,0)W_m(k_2,0)\right],
		\end{aligned}
	\end{equation*}
	\begin{equation*}
		\begin{aligned}
			\psi_{2}(k_{1},k_{2}) &=\sum_{A_{t_{1},k_{1}}\cap{A_{t_{2},k_{2}}}\neq{\emptyset}} \left[\mathbf{E} I_{\epsilon}(\|Y_{t_{1}}^m-Y_{t_{1}+k_{1}}^m\|) I_{\epsilon}(\|Y_{t_{2}}^m-Y_{t_{2}+k_{2}}^m\|)-W_m(k_1,0)W_m(k_2,0)\right].
		\end{aligned}
	\end{equation*}
	It's easy to obtain that 
	\begin{equation}\label{psi1}
		\psi_{1}(k_{1},k_{2})=0.
	\end{equation}
	As for $\psi_2(k_1,k_2)$, since $A_{t_{1},k_{1}}\cap{A_{t_{2},k_{2}}}\neq{\emptyset}$, $t_{1}$ and  $t_{2}$ are restricted by each other, so there are $O(T)$ terms involved in $\psi_{2}(k_{1},k_{2})$, that is,
	\begin{equation}\label{psi2}
		\psi_{2}(k_{1},k_{2})=O(T).
	\end{equation}
	By \eqref{VarCIIC}, \eqref{Divitionpsi}, \eqref{psi1} and \eqref{psi2}, we have
	\begin{equation}\label{RVarCII}
		\begin{aligned}
			&\uppercase\expandafter{\romannumeral1}_{m,m}=\frac{1}{N^2} \sum_{k_{1}=1}^{m-1} \sum_{k_{2}=1}^{m-1}\left[\psi_{1}(k_{1},k_{2}) +\psi_{2}(k_{1},k_{2})\right]=O(T^{-3}).
		\end{aligned}
	\end{equation}
	
	From now on, we introduce the calculation of $\uppercase\expandafter{\romannumeral2}_{m,m}$, which can be rewritten as
	\begin{equation}\label{RVarCIIIPhi}
		\begin{aligned}
			\uppercase\expandafter{\romannumeral2}_{m,m}&=\mathbf{E}\left[ \breve{C}_{m,T}(\epsilon)-\mathbf{E}\breve{C}_{m,T}(\epsilon)\right] \left[\widetilde{C}_{m,T}(\epsilon)-\mathbf{E}\widetilde{C}_{m,T}(\epsilon)\right]\\
			&=\frac{1}{NN_{0}}\mathbf{E}\left(\sum_{k=1}^{m-1}\sum_{t_{1}=1}^{T_{m}-k} \left[I_{\epsilon}(\|Y_{t_{1}}^m-Y_{t_{1}+k}^m\|)-W_m(k,0)\right]\right)\\
			&\quad \quad \qquad \times{ \left(\sum_{t_{2}=1}^{T_{m}-m}\sum_{s_{2}=t_{2}+m}^{T_{m}}\left[ I_{\epsilon}(\|Y_{t_{2}}^m-Y_{s_{2}}^m\|) -W_m(m,0)\right]\right)}\\
			&=\frac{1}{NN_{0}}\sum_{k=1}^{m-1}\Phi(k),
		\end{aligned}
	\end{equation}
	where
	\begin{equation*}
		\begin{aligned}
			\Phi(k)&=\sum_{t_{1}=1}^{T_{m}-k}\sum_{t_{2}=1}^{T_{m}-m} \sum_{s_{2}=t_{2}+m}^{T_{m}} \left[\mathbf{E}I_{\epsilon}\left(\|Y_{t_{1}}^m-Y_{t_{1}+k}^m\|\right) I_{\epsilon}(\|Y_{t_{2}}^m-Y_{s_{2}}^m\|) -W_m(k,0)W_m(m,0)\right].
		\end{aligned}
	\end{equation*}
	Hence we get $\uppercase\expandafter{\romannumeral2}_{m,m}$ once we calculate the results of $\mathbf{E}I_{\epsilon}(\|Y_{t_{1}}^m-Y_{t_{1}+k}^m\|)I_{\epsilon}(\|Y_{t_2}^m-Y_{s_2}^m\|)$, which depends on  the relationship between $I_{\epsilon}(\|Y_{t_1}^m-Y_{t_1+k}^m\|)$ and $I_{\epsilon}(\|Y_{t_{2}}^m-Y_{s_{2}}^m\|)$. Similar to the above method, we correspond the set of all random variables involved in $I_{\epsilon}(\|Y_{t_1}^m-Y_{t_1+k}^m\|)$ and $I_{\epsilon}(\|Y_{t_2}^m-Y_{s_2}^m\|)$ to the following two sets respectively:
	\begin{equation}
		\begin{aligned}
			&A_{t_{1},k}=\left\{t_{1},t_{1}+1,\cdots,t_{1}+k,\cdots,t_{1}+m+k-1 \right\},\\
			&B_{t_{2},s_{2}}
			=\left\{t_{2},\cdots,t_{2}+m-1,s_{2},\cdots,s_{2}+m-1 \big|s_{2}-t_{2}\geq{m}\right\},
		\end{aligned}
	\end{equation}
	where $B_{t_{2},s_{2}}$ can be rewritten as below: 
	\begin{equation}\label{Bdivision}
	B_{t_{2},s_{2}}=B_{t_{2}}+B_{s_{2}}, 
	B_{t_{2}}=\left\{t_{2},\cdots,t_{2}+m-1\right\} ,B_{s_{2}}=\left\{s_{2},\cdots,s_{2}+m-1\right\}.
	\end{equation}
	
	Thus, the number of all elements contained in  $A_{t_1,k}\cap{B_{t_2,s_2}}$ reflects the relationship between $I_{\epsilon}(\|Y_{t_1}^m-Y_{t_1+k}^m\|)$ and $I_{\epsilon}(\|Y_{t_2}^m-Y_{s_2}^m\|)$, for example, $I_{\epsilon}(\|Y_{t_1}^m-Y_{t_1+k}^m\|)$ is independent with $I_{\epsilon}(\|Y_{t_2}^m-Y_{s_2}^m\|)$ if and only if $^{\#}{(A_{t_{1},k}\cap{B_{t_{2},s_{2}}})}=0$.
	Therefore, according to the value of $^{\#}{(A_{t_{1},k}\cap{B_{t_{2},s_{2}}})}$, we divive $\Phi(k)$ into the forllowing three parts:
	\begin{equation}\label{DivitionPhi}
		\Phi(k)=\Phi_{1}(k)+\Phi_{2}(k)+\Phi_{3}(k),
	\end{equation}
	where
	\begin{equation*}
		\begin{aligned}
			\Phi_{1}(k)&=\sum_{\mbox{\scriptsize{$\begin{array}{c}
							^{\#}{(A_{t_{1},k}\cap{B_{t_{2},s_{2}}})}=0\end{array}$}}} \left[\mathbf{E}I_{\epsilon}(\|Y_{t_{1}}^{m}-Y_{t_{1}+k}^m\|) I_{\epsilon}(\|Y_{t_{2}}^m-Y_{s_{2}}^m\|)  -W_m(k,0)W_m(m,0) \right],
		\end{aligned}
	\end{equation*}
	\begin{equation*}
		\begin{aligned}
			\Phi_{2}(k)&=\sum_{\mbox{\scriptsize{$\begin{array}{c}
							j=1\\
							^{\#}{(A_{t_{1},k}\cap{B_{t_{2},s_{2}}})}=j\end{array}$}}}^{m} \left[\mathbf{E}I_{\epsilon}(\|Y_{t_{1}}^m-Y_{t_{1}+k}^m\|) I_{\epsilon}(\|Y_{t_{2}}^m-Y_{s_{2}}^m\|) -W_m(k,0)W_m(m,0)\right],
		\end{aligned}
	\end{equation*}
	\begin{equation*}
		\begin{aligned}
			\Phi_{3}(k)&=\sum_{\mbox{\scriptsize{$\begin{array}{c}
							j=m+1\\
							^{\#}{(A_{t_{1},k}\cap{B_{t_{2},s_{2}}})}=j\end{array}$}}}^{m+k}\left[\mathbf{E}I_{\epsilon}(\|Y_{t_{1}}^{m}-Y_{t_{1}+k}^m\|) I_{\epsilon}(\|Y_{t_{2}}^m-Y_{s_{2}}^m\|) -W_m(k,0)W_m(m,0)\right].
		\end{aligned}
	\end{equation*}
	Obviouly, we have 
	\begin{equation}\label{Phi1}
		\Phi_{1}(k)=0,
	\end{equation}
	because the independence is satisfied in this case.
	
	As for $\Phi_{3}(k)$,  $\left\{^{\#}{(A_{t_{1},k}\cap{B_{t_{2},s_{2}}})}=j,j=m+1,\cdots,m+k\right\}$ and $\{A_{t_{1},k}\cap{B_{t_{2}}}\neq{\emptyset}, A_{t_{1}}\cap{B_{s_{2}}}\neq{\emptyset}\}$ are equivalent, which means that $t_1$, $t_2$ and $s_2$ are mutually restricted, therefore, there are $O(T)$ terms in $\Phi_3(k)$, that is, 
	\begin{equation}\label{Phi3}
		\Phi_{3}(k)=O(T).
	\end{equation}
	
	$\Phi_{2}(k)$ contains all the terms when $^{\#}{(A_{t_{1},k}\cap{B_{t_{2},s_{2}}})}=1,\cdots, m$. Considering 
	$^{\#}{(A_{t_{1},k})}=m+k$, $^{\#}{(B_{t_{2}})}=m$ and  $^{\#}{(B_{s_{2}})}=m$, $^{\#}{(A_{t_{1},k}\cap{B_{t_{2},s_{2}}})}=1,\cdots, m$ can be divided into the following two sub-cases according to  \eqref{Bdivision}:
	$$^{\#}{(A_{t_{1},k}\cap{B_{s_{2}}})}=0, ^{\#}{(A_{t_{1},k}\cap{B_{t_{2}}})}=1,\cdots,m,$$
	and 
	$$^{\#}{(A_{t_{1},k}\cap{B_{t_{2}}})}=0, ^{\#}{(A_{t_{1},k}\cap{B_{s_{2}}})}=1,\cdots,m,$$
	further, we have
	\begin{equation*}
		\Phi_{2}(k)=\Phi_{2,1}(k)+\Phi_{2,2}(k),
	\end{equation*}
	where
	\begin{equation*}
		\begin{aligned}
			\Phi_{2,1}(k)&=\sum\limits_{\mbox{\scriptsize{$\begin{array}{c}
							j=1\\
							^{\#}{(A_{t_{1},k}\cap{B_{t_{2}}})}=j\\
							^{\#}{(A_{t_{1},k}\cap{B_{s_{2}}})}
							=0\end{array}$}}}^m\left[\mathbf{E}I_{\epsilon}(\|Y_{t_{1}}^m-Y_{t_{1}+k}^m\|) I_{\epsilon}(\|Y_{t_{2}}^m-Y_{s_{2}}^m\|) -W_m(k,0)W_m(m,0)\right],
		\end{aligned}
	\end{equation*}
	and
	\begin{equation*}
		\begin{aligned}
			\Phi_{2,2}(k)&=\sum\limits_{\mbox{\scriptsize{$\begin{array}{c}
							j=1\\
							^{\#}{(A_{t_{1},k}\cap{B_{s_{2}}})}=j\\
							^{\#}{(A_{t_{1},k}\cap{B_{t_{2}}})}=0\end{array}$}}}^m \left[\mathbf{E}I_{\epsilon}(\|Y_{t_{1}}^m-Y_{t_{1}+k}^m\|) I_{\epsilon}(\|Y_{t_{2}}^m-Y_{s_{2}}^m\|) -W_m(k,0)W_m(m,0)\right].
		\end{aligned}
	\end{equation*}
	In the cases of $\Phi_{2,1}(k)$ and $\Phi_{2,2}(k)$, the dependence of $I_{\epsilon}(\|Y_{t_{1}}^m-Y_{t_{1}+k}^m\|)$ and $I_{\epsilon}(\|Y_{t_{2}}^m-Y_{s_{2}}^m\|)$ is little complicated.
	To have a more intuitive understanding of $\Phi_{2,1}(k)$ and $\Phi_{2,2}(k)$, for fixed $k$ and $m$, we associate $I_{\epsilon}(\|Y_{t_{1}}^m-Y_{t_{1}+k}^m\|)$ with the graph $G_{t_{1},k}$ in Figure \ref{FigureGt1k},
	\begin{figure}[H]
		\centering
		\includegraphics[height=2.47cm,width=11.0cm]{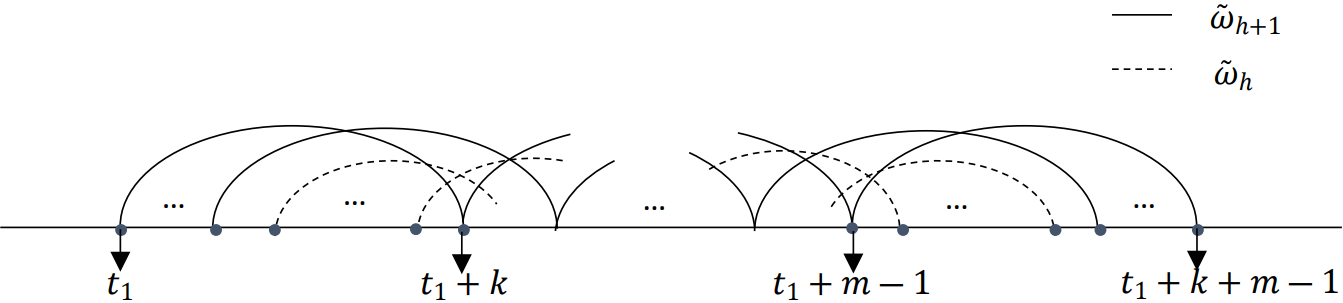}
		\caption{Graph $G_{t_{1},k}$ for $I_{\epsilon}(\|Y_{t_{1}}^m-Y_{t_{1}+k}^m\|)$.}
		\label{FigureGt1k}
	\end{figure}
	where $\tilde{\omega}_{h+1}$ is a chain of length $h+1$; $\tilde{\omega}_{h}$ is a chain of length $h$, $A_{t_{1},k}=\{t_{1},t_{1}+1\cdots,t_{1}+k+m-1\}$ introduced above is the set of vertices of all chains included in graph $G_{t_{1},k}$. In short, $G_{t_{1},k}$ contains $m+k$ points, and take them as vertices to obtain $k-i$ chains of length $h$ and $i$ chains of length $h+1$ in the order shown in Figure \ref{FigureGt1k}. For example, when $m=5,$ $k=3$, the graph associated with $I_{\epsilon}(\|Y_{t_{1}}^5-Y_{t_{1}+3}^5\|)$ is presented in Figure \ref{FigureG53}:
	
	\begin{figure}[H]
		\centering
		\includegraphics[height=2.25cm,width=9.58cm]{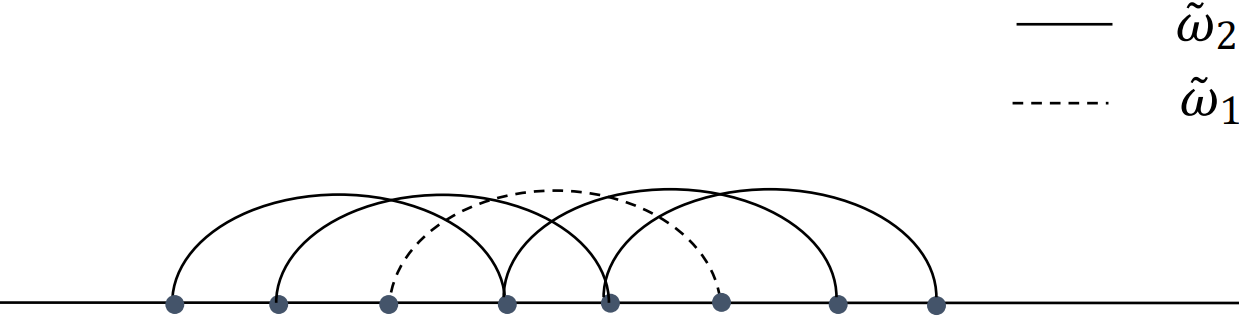}
		\caption{Graph for $I_{\epsilon}(\|Y_{t_{1}}^5-Y_{t_{1}+3}^5\|)$}
		\label{FigureG53}
	\end{figure}
	
	In addition, for fixed $m$, we associate $I_{\epsilon}(\|Y_{t_{2}}^m-Y_{s_{2}}^m\|)$ with the graph $G_{t_{2},s_{2}}$ in Figure \ref{FigureGt2s2},
	
	\begin{figure}[H]
		\centering
		\includegraphics[height=1.85cm,width=9.58cm]{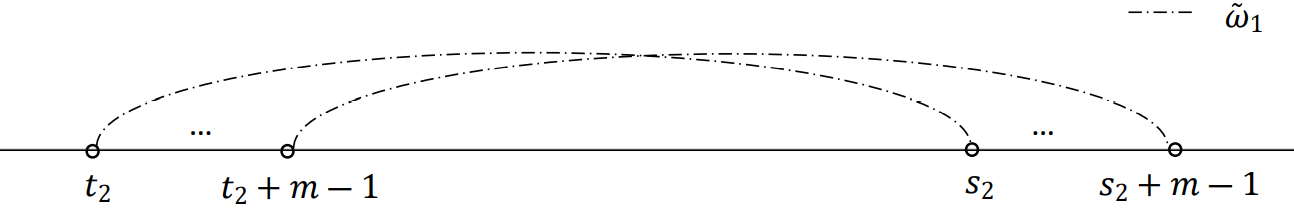}
		\caption{Graph $G_{t_{2},s_{2}}$ for $I_{\epsilon}(\|Y_{t_{2}}^m-Y_{s_{2}}^m\|)$.}
		\label{FigureGt2s2}
	\end{figure}
	where $\tilde{\omega}_{1}$ is a chain of length 1, $B_{t_{2}}$ and $B_{s_{2}}$ are two sets of vertices of $G_{t_{2},s_{2}}$. It should be noted that $B_{t_{2}}$ and $B_{s_{2}}$ never overlap and $B_{t_{2}}$ is always on the left side of $B_{s_{2}}$. To put it simply, $G_{t_{2},s_{2}}$ is composed of $m$ chains of length 1, with adjacent chains spaced by equal 1 unit. For example, when $m=5$, the graph $G_{t_{2},s_{2}}$ associated with $I_{\epsilon}(\|Y_{t_{2}}^5-Y_{s_{2}}^5\|)$ is as follows:
	
	\begin{figure}[H]
		\centering
		\includegraphics[height=1.55cm,width=10.28cm]{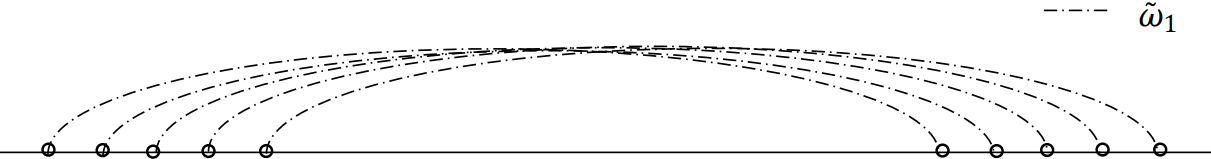}
		\caption{Graph for $I_{\epsilon}(\|Y_{t_{2}}^5-Y_{s_{2}}^5\|)$.}
	\end{figure}
	
	Combining these two graphs $G_{t_{1},k}$ and $G_{t_{2},s_{2}}$, we know that terms involved in $\Phi_{2,1}(k)$ and $\Phi_{2,2}(k)$ correspond to the following two situations respectively:
	\begin{itemize}
		\item For a fixed graph $G_{t_{1},k}$, its vertex set $A_{t_{1},k}$ only has an intersection with the vertex set $B_{t_{2}}$ of graph $G_{t_{2},s_{2}}$, but has no intersection with $B_{s_{2}}$, and $B_{s_{2}}$ will always be located on the right side of these two sets;
		\item For a fixed graph $G_{t_{1},k}$, its vertex set $A_{t_{1},k}$ only has an intersection with the vertex set $B_{s_{2}}$ of graph $G_{t_{2},s_{2}}$, but has no intersection with $B_{t_{2}}$, and $B_{t_{2}}$ will always be located on the left side of these two sets.
	\end{itemize}
	It's easy to understand that $\Phi_{2,1}(k)=\Phi_{2,2}(k)$, so
	\begin{equation}\label{DivitionPhi2}
		\Phi_{2}(k)=\Phi_{2,1}(k)+\Phi_{2,2}(k)=2\Phi_{2,1}(k).
	\end{equation}
	Thus we only need to study $\Phi_{2,1}(k)$, specifically:
	
	\begin{equation*}
		\begin{aligned}
			&\Phi_{2,1}(k)=\sum\limits_{\mbox{\scriptsize{$\begin{array}{c}
							j=1\\
							^{\#}{(A_{t_{1},k}\cap{B_{t_{2}}})}=j\\
							^{\#}{(A_{t_{1},k}\cap{B_{s_{2}}})}=0\end{array}$}}}^{m} \left[\mathbf{E}I_{\epsilon}(\|Y_{t_{1}}^m-Y_{t_{1}+k}^m\|) I_{\epsilon}(\|Y_{t_{2}}^m-Y_{s_{2}}^m\|)- W_m(k,0)W_m(m,0)\right]\\
			&=\sum_{t_{1}=m}^{\mbox{\tiny{$\begin{array}{c}
							T-3m-\\
							k+2\end{array}$}}}\sum_{\mbox{\tiny{$\begin{array}{c}
							s_{2}=t_{1}+2m\\
							+k-1 \end{array}$}}}^{T-m+1}\left[\phi_{1}(t_{1},s_{2}) +\phi_{2}(t_{1},s_{2})+\phi_{3}(t_{1},s_{2})\right]+O(T),
		\end{aligned}
	\end{equation*}
	where
	\begin{equation*}
		\begin{aligned}
			\phi_{1}(t_{1},s_{2})&=\sum_{t_{2}=t_{1}-m+1}^{t_{1}} \left[\mathbf{E}I_{\epsilon}(\|Y_{t_{1}}^m-Y_{t_{1}+k}^m\|) I_{\epsilon}(\|Y_{t_{2}}^m-Y_{s_{2}}^m\|)
			-W_m(k,0)W_m(m,0)\right],
		\end{aligned}
	\end{equation*}
	\begin{equation*}
		\begin{aligned}
			\phi_{2}(t_{1},s_{2})&=\sum_{t_{2}=t_{1}+1}^{t_{1}+k-1} \left[\mathbf{E}I_{\epsilon}(\|Y_{t_{1}}^m-Y_{t_{1}+k}^m\|) I_{\epsilon}(\|Y_{t_{2}}^m-Y_{s_{2}}^m\|) -W_m(k,0)W_m(m,0)\right],
		\end{aligned}
	\end{equation*}
	
	\begin{equation*}
		\begin{aligned}
			\phi_{3}(t_{1},s_{2})&=\sum_{t_{2}=t_{1}+k}^{t_{1}+k+m-1} \left[\mathbf{E}I_{\epsilon}(\|Y_{t_{1}}^m-Y_{t_{1}+k}^m\|) I_{\epsilon}(\|Y_{t_{2}}^m-Y_{s_{2}}^m\|) -W_m(k,0)W_m(m,0)\right].
		\end{aligned}
	\end{equation*}
	According to the symmetry of Figure \ref{FigureGt1k}, it's easy to find that $\phi_{1}(t_{1},s_{2})=\phi_{3}(t_{1},s_{2})$, since, for each term in $\phi_{1}(t_{1},s_{2})$, there is a term in $\phi_{3}(t_{1},s_{2})$ so that the relative position between $G_{t_{1},k}$ in Figure \ref{FigureGt1k} and $G_{t_{2},s_{2}}$ in Figure \ref{FigureGt2s2} is essentially the same. In other words, for each term in $\phi_{1}(t_{1},s_{2})$ there is always a term in $\phi_{3}(t_{1},s_{2})$ satisfies that the dependence between $\{Y_{t_{1}}^m,Y_{t_{1}+k_{1}}^m\}$ and $\{Y_{t_{2}}^m,Y_{t_{2}+k_{2}}^m\}$ is exactly the same.
	Then, we have
	\begin{equation}\label{DivitionPsi21}
		\begin{aligned}
			\Phi_{2,1}(k)&=\sum_{t_{1}=m}^{T-3m-k+2}\sum_{\mbox{\tiny{$\begin{array}{c}
							s_{2}=t_{1}+2m\\
							+k-1 \end{array}$}}}^{T-m+1}\left[2\phi_{1}(t_{1},s_{2}) +\phi_{2}(t_{1},s_{2})\right]+O(T).
		\end{aligned}
	\end{equation}
	So far, we only need to focus on the results of $\phi_{1}(t_{1},s_{2})$ and $\phi_{2}(t_{1},s_{2})$. For $\phi_{1}(t_{1},s_{2})$, combining the graph $G_{t_{1},k}$ in Figure \ref{FigureGt1k} and $G_{t_{2},s_{2}}$ in Figure \ref{FigureGt2s2}, easy to verify that each term of it can be associated with a graph shaped like $G_{1}(t_{1},s_{2})$ given in Figure \ref{FigureG1t1s2}. 
	\begin{figure}[H]
		\centering
		\includegraphics[height=3.57cm,width=10.1cm]{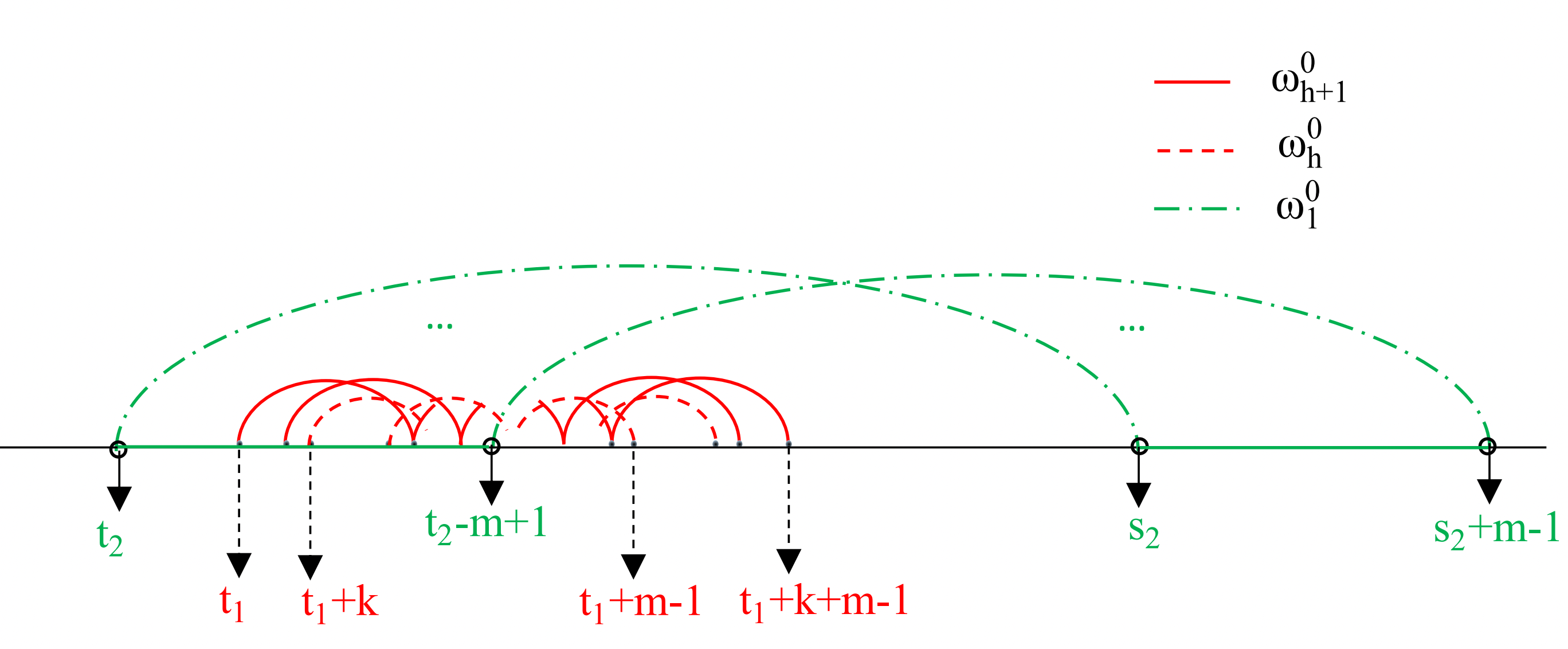}
		\caption{$G_{1}(t_{1},s_{2})$}
		\label{FigureG1t1s2}
	\end{figure}
	Then, let $d_{1}=t_{2}+m-t_{1}$, we have:
	\begin{equation}\label{phi1t1s2}
		\begin{aligned}
			\phi_{1}(t_{1},s_{2})&=\sum_{t_{2}=t_{1}-m+1}^{t_{1}} \left[\mathbf{E}I_{\epsilon}(\|Y_{t_{1}}^m-Y_{t_{1}+k}^m\|) I_{\epsilon}(\|Y_{t_{2}}^m-Y_{s_{2}}^m\|) 
			-W_m(k,0)W_m(m,0)\right]\\
			&=\sum_{d_{1}=1}^{m}\left[\mathbf{E} I_{\epsilon}(\|Y_{t_{1}}^m-Y_{t_{1}+k}^m\|) I_{\epsilon}(\|Y_{t_{1}+d_{1}-m}^m-Y_{s_{2}}^m\|)
			-W_m(k,0)W_m(m,0)\right]\\
			&=\sum_{d_{1}=1}^{m}[W_m(k,d_1)-W_m(k,0)W_m(m,0)].
		\end{aligned}
	\end{equation}
	
	Similarly, each terms of $\phi_{2}(t_{1},s_{2})$ can be associated with a graph shaped like $G_{2}(t_{1},s_{2})$ given in Figure \ref{FigureG2t1s2}. 
	\begin{figure}[H]
		\centering
		\includegraphics[height=3.57cm,width=10.1cm]{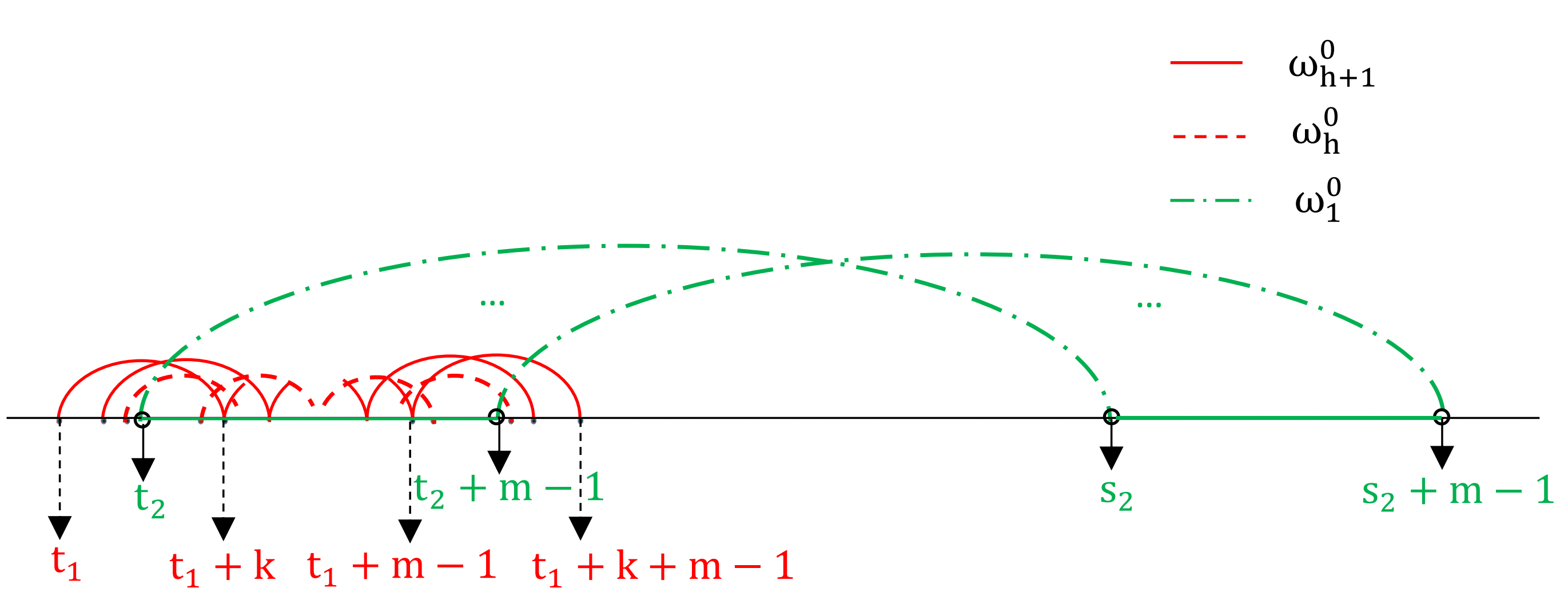}
		\caption{$G_{2}(t_{1},s_{2})$}
		\label{FigureG2t1s2}
	\end{figure}
	Then, let $d_{2}=t_{2}-t_{1}$, combining with \eqref{Umk}, we get:
	\begin{equation}\label{phi2t1s2}
		\begin{aligned}
			\phi_{2}(t_{1},s_{2})&=\sum_{t_{2}=t_{1}+1}^{t_{1}+k-1} \left[\mathbf{E}I_{\epsilon}(\|Y_{t_{1}}^m-Y_{t_{1}+k}^m\|) I_{\epsilon}(\|Y_{t_{2}}^m-Y_{s_{2}}^m\|)  -W_m(k,0)W_m(m,0)\right]\\
			&=\sum_{d_{2}=1}^{k-1}\left[\mathbf{E} I_{\epsilon}(\|Y_{t_{1}}^m-Y_{t_{1}+k}^m\|) I_{\epsilon}(\|Y_{t_{1}+d_{2}}^m-Y_{s_{2}}^m\|) -W_m(k,0)W_m(m,0)\right]\\
			&=\left[1-I_0(k-2i)\right]\left[\sum_{d_2=1}^iU_m(k,d_2)+ \sum_{d_2=i+1}^{k-i}U_m(k,i)+\sum_{d_2=k-i+1}^{k-1}U_m(k,k-d_2)\right]\\
			&+
			I_0(k-2i)\left[\sum_{d_2=1}^{k-i}U_m(k,d_2)+\sum_{d_2=k-i+1}^iU_m(k,k-i)+\sum_{d_2=i+1}^{k-1}U_m(k,k-d_2)\right]
		\end{aligned}
	\end{equation}
	Based on \eqref{DivitionPhi2}, \eqref{DivitionPsi21}, \eqref{phi1t1s2} and \eqref{phi2t1s2}, we get
	\begin{equation}\label{Phi2}
		\begin{aligned}
			\Phi_{2}(k)
			&=2\sum_{t_{1}=m}^{T-3m-k+2}\sum_{\mbox{\tiny{$\begin{array}{c}
							s_{2}=t_{1}+2m\\
							+k-1 \end{array}$}}}^{T-m+1}\left[2\phi_{1}(t_{1},s_{2}) +\phi_{2}(t_{1},s_{2})\right]+O(T)\\
			&=\mathcal{M}_{T,m}(k)\left[2\times{}\eqref{phi1t1s2}+\eqref{phi2t1s2}\right]+O(T),
		\end{aligned}
	\end{equation}
	where $\mathcal{M}_{T,m}(k)=(T-4m-k+3)(T-4m-k+4)$.
	Further,  according to \eqref{RVarCIIIPhi}, \eqref{DivitionPhi}, \eqref{Phi1}, \eqref{Phi3}, \eqref{Phi2}, we have
	\begin{equation}\label{RVarCIII}
		\begin{aligned}
			\uppercase\expandafter{\romannumeral2}_{m,m}&=\frac{1}{NN_{0}} \sum_{k=1}^{m-1}\mathcal{M}_{T,m}(k)\left[2\times{}\eqref{phi1t1s2}+\eqref{phi2t1s2}\right]+O\left(T^{-3}\right),
		\end{aligned}
	\end{equation}
	As a result,
	$\eqref{VarCI}=\breve{\sigma}_m^2+O(T^{-3})$, where
	\begin{equation}
		\begin{aligned}
			\breve{\sigma}_m^2=	\frac{2}{N^2} \sum_{k=1}^{m-1}\mathcal{M}_{T,m}(k)\left[2\times{}\eqref{phi1t1s2}+\eqref{phi2t1s2}\right].
		\end{aligned}
	\end{equation}
	The proof of Theorem \ref{ThCLTCI} is completed.
\end{proof}

\section{Appendix B}
\begin{proof}[Proof of Theorem \ref{ThCLTKCI}:]
	Based on the Delta method, the asymptotic variance of $\mathcal{K}_{m,T}(\epsilon)$ is approximately equal to the variance of $\widetilde{\mathcal{K}}_{m,T}(\epsilon)$. Therefore, we only need to focus on the variance of $\widetilde{\mathcal{K}}_{m,T}(\epsilon)$. According to \eqref{tildeKmT}, we have
	\begin{eqnarray}\label{DivitionVartildeKmT}
		\begin{aligned}
			&\mathbf{Var}\left(\widetilde{\mathcal{K}}_{m,T}(\epsilon)\right)= V_{1}+V_{2}+V_{3}+V_{4}-V_{5}-V_{6}-V_{7}+V_{8}+V_{9}+V_{10},
		\end{aligned}
	\end{eqnarray}
	where
	\begin{equation*}
		\begin{aligned}
			&V_{1}=\mathbf{E}\left(C_{m,T}(\epsilon)-\mu_{m}\right)^2, \\
			&V_{2}=\frac{1}{N^2}\mathbf{E}\left(\sum_{k=1}^{m-1}(T_m-k)W_m^{(h)}(k,0) \left[\left(\widehat{\omega}_{h}^0\right) -\left(\omega_{h}^0\right)\right]\right)^2,\\
			&V_{3}=\frac{1}{N^2}\mathbf{E}\left(\sum_{k=1}^{m-1}(T_m-k)W_m^{(h+1)}(k,0) \left[\left(\widehat{\omega}_{h+1}^0\right) -\left(\omega_{h+1}^0\right)\right]\right)^2, \\
			&V_{4}=\left(\frac{N_{0}}{N}W_m^{(1)}(m,0)\right)^2 \mathbf{E}\left[\left(\widehat{\omega}_{1}^0\right) -\left(\omega_{1}^0\right)\right]^2\\
			&V_{5}=\frac{2}{N}\sum_{k=1}^{m-1}(T_m-k)W_m^{(h)}(k,0) \mathbf{E}\left[C_{m,T}(\epsilon)-\mu_{m}\right] \left[\left(\widehat{\omega}_{h}^0\right) -\left(\omega_{h}^0\right)\right],\\
			&V_{6}=\frac{2}{N}\sum_{k=1}^{m-1}(T_m-k)W_m^{(h+1)}(k,0) \mathbf{E}\left[C_{m,T}(\epsilon)-\mu_{m}\right] \left[\left(\widehat{\omega}_{h+1}^0\right) -\left(\omega_{h+1}^0\right)\right],\\
			&V_{7}=\frac{2N_{0}}{N}W_m^{(1)}(m,0) \mathbf{E}\left[C_{m,T}(\epsilon)-\mu_{m}\right] \left[\left(\widehat{\omega}_{1}^0\right) -\left(\omega_{1}^0\right)\right], \\
			&V_{8}=\frac{2}{N^2}\mathbf{E}\left(\sum_{k=1}^{m-1}(T_m-k)W_m^{(h)}(k,0) \left[\left(\widehat{\omega}_{h}^0\right) -\left(\omega_{h}^0\right)\right]\right)\\
					\end{aligned}
			\end{equation*}
				\begin{equation*}
			\begin{aligned}
			&\qquad \qquad \times{} \left(\sum_{k=1}^{m-1}(T_m-k)W_m^{(h+1)}(k,0) \left[\left(\widehat{\omega}_{h+1}^0\right) -\left(\omega_{h+1}^0\right)\right]\right), \\
			&V_{9}=\frac{2N_{0}}{N^2}W_m^{(1)}(m,0)\sum_{k=1}^{m-1}(T_m-k)W_m^{(h)}(k,0) \mathbf{E}\left[\widehat{\omega}_{1}^0-\omega_{1}^0\right] \left[\left(\widehat{\omega}_{h}^0\right) -\left(\omega_{h}^0\right)\right],\\
			&V_{10}=\frac{2N_{0}}{N^2}W_m^{(1)}(m,0)\sum_{k=1}^{m-1}(T_m-k)W_m^{(h+1)}(k,0)\mathbf{E}\left[\widehat{\omega}_{1}^0-\omega_{1}^0\right] \left[\left(\widehat{\omega}_{h+1}^0\right) -\left(\omega_{h+1}^0\right)\right].
		\end{aligned}
	\end{equation*}
	Among them,
	\begin{equation}\label{V1}
		V_{1}=\mathbf{E}\left(C_{m,T}(\epsilon)-\mu_{m}\right)^2 =\sigma_{m,m}^2+O(T^{-3}),
	\end{equation}
	where $\sigma_{m,m}^2=\eqref{CIsigma}$ in Theorem \ref{ThCLTCI}.
	\begin{eqnarray}\label{V20}
		\begin{aligned}
			V_{2}&=\frac{1}{N^2}\mathbf{E}\left(\sum\limits_{k=1}^{m-1}(T_m-k)W_m^{(h)}(k,0) \left[\left(\widehat{\omega}_{h}^0\right) -\left(\omega_{h}^0\right)\right]\right)^2\\
			&=\frac{1}{N^2}\mathbf{E} \left(\sum\limits_{k_{1}=1}^{m-1}(T_m-k_1)W_m^{(h_1)}(k_1,0) \left[\left(\widehat{\omega}_{h_{1}}^0\right) -\left(\omega_{h_{1}}^0\right)\right]\right)\\
			&\qquad \times{} \left(\sum\limits_{k_{2}=1}^{m-1}(T_m-k_2)W_m^{(h_2)}(k_2,0) \left[\left(\widehat{\omega}_{h_{2}}^0\right) -\left(\omega_{h_{2}}^0\right)\right]\right)\\
			&=\frac{1}{N^2}\sum_{k=1}^{m-1}(T_m-k)^2\left(W_m^{(h)}(k,0)\right)^2\mathbf{E} \left[\left(\widehat{\omega}_{h}^0\right) -\left(\omega_{h}^0\right)\right]^2\\
			&+\frac{1}{N^2}\sum_{k_{1}\neq{k_{2}}}(T_m-k_1)(T_m-k_2)W_m^{(h_1)}(k_1,0)W_m^{(h_2)}(k_2,0)
			\\ & \qquad \qquad \times{} 
			\mathbf{E} \left[\left(\widehat{\omega}_{h_{1}}^0\right) -\left(\omega_{h_{1}}^0\right)\right] \left[\left(\widehat{\omega}_{h_{2}}^0\right) -\left(\omega_{h_{2}}^0\right)\right],
		\end{aligned}
	\end{eqnarray}
	where $h=\lfloor{\frac{m}{k}}\rfloor$,  $h_{1}=\lfloor{\frac{m}{k_{1}}}\rfloor$,  $h_{2}=\lfloor{\frac{m}{k_{2}}}\rfloor$.
	Note that the results of $V_{2}$ is determined by the results of $\mathbf{E}\left[\left(\widehat{\omega}_{h}^0\right) -\left(\omega_{h}^0\right)\right]^2$ and $\mathbf{E} \left[\left(\widehat{\omega}_{h_{1}}^0\right)-\left(\omega_{h_{1}}^0\right)\right] \left[\left(\widehat{\omega}_{h_{2}}^0\right)-\left(\omega_{h_{2}}^0\right)\right]$, so we focus on these two parts. Reviewing the definition of $\widehat{\omega}_{h}^0$ in \eqref{estomegal0}, we have
	
	\begin{equation}\label{Eomegah0}
		\begin{aligned}
			&\mathbf{E}\left[\left(\widehat{\omega}_{h}^0\right) -\left(\omega_{h}^0\right)\right]^2\\ &=C_{T,h}^2\mathbf{E}\left[\sum \limits_{t_{0},t_{1},\cdots,t_{h} \atop\mbox{\tiny distrinct}} {\prod_{\rho=0}^{h-1}I_{\epsilon}(|u_{t_{\rho}}-u_{t_{\rho+1}}|)} -\left(\omega_{h}^0\right)\right]^2\\
			&=C_{T,h}^2 \sum\limits_{t_{0},t_{1},\cdots,t_{h} \atop\mbox{\tiny distrinct}} \sum\limits_{s_{0},s_{1},\cdots,s_{h} \atop\mbox{\tiny distrinct}}\left[ \mathbf{E} {\prod_{\rho_1=0}^{h-1} I_{\epsilon}(|u_{t_{\rho_1}}-u_{t_{\rho_1+1}}|)}\prod_{\rho_2=0}^{h-1} I_{\epsilon}(|u_{s_{\rho_2}}-u_{s_{\rho_2+1}}|) -\left(\omega_{h}^0\right)^2\right].
		\end{aligned}
	\end{equation}
	For simplify, we associate all random variables invovled in
	$$\prod_{\rho_1=0}^{h-1}I_{\epsilon}(|u_{t_{\rho_1}}-u_{t_{\rho_1+1}}|) \text{and} \prod_{\rho_2=0}^{h-1}I_{\epsilon}(|u_{s_{\rho_2}}-u_{s_{\rho_2+1}}|)$$ with the following two sets respectively:
	\begin{equation*}
		A_{t}^h=\{t_{0},t_{1},\cdots,t_{h}\}\qquad \mbox{and} \qquad
		A_{s}^h=\{s_{0},s_{1},\cdots,s_{h}\}.
	\end{equation*}
	Then the number of elements contained by $A_{t}^{h}\cap{A_{s}^{h}}$ can reflect the relationship between $\prod_{\rho_1=0}^{h-1}I_{\epsilon}(|u_{t_{\rho_1}}-u_{t_{\rho_1+1}}|) $ and $\prod_{\rho_2=0}^{h-1}I_{\epsilon}(|u_{s_{\rho_2}}-u_{s_{\rho_2+1}}|)$, for instance, when $^{\#}{\left(A_{t}^{h}\cap{A_{s}^h}\right)}=0$, $\prod_{\rho_1=0}^{h-1}I_{\epsilon}(|u_{t_{\rho_1}}-u_{t_{\rho_1+1}}|) $ is independent with $\prod_{\rho_2=0}^{h-1}I_{\epsilon}(|u_{s_{\rho_2}}-u_{s_{\rho_2+1}}|)$. So we can rewrite \eqref{Eomegah0} as follows:
	\begin{equation}\label{Eomegah00}
		\begin{aligned}
			\eqref{Eomegah0}&=C_{T,h}^2\sum_{\mbox{\tiny{$\begin{array}{c}
							j=0\\
							^{\#}{\left(A_{t}^{h}\cap{A_{s}^{h}}\right)}=j\end{array}$}}}^h \mathbf{E}\left[\prod_{\rho_1=0}^{h-1} I_{\epsilon}(|u_{t_{\rho_1}}-u_{t_{\rho_1+1}}|) \prod_{\rho_2=0}^{h-1} I_{\epsilon}(|u_{s_{\rho_2}}-u_{s_{\rho_2+1}}|) -\left(\omega_{h}^0\right)^2\right].
		\end{aligned}
	\end{equation}
	When $^{\#}{\left(A_{t}^{h}\cap{A_{s}^h}\right)}=0$, the independence is satisified, so we get 
	\begin{equation}\label{Eomegah1}
		\begin{aligned}
			&\quad \sum_{\mbox{\tiny{$\begin{array}{c}
							^{\#}{\left(A_{t}^{h}\cap{A_{s}^{h}}\right)}=0\end{array}$}}} \left[\mathbf{E}\prod_{\rho_1=0}^{h-1} I_{\epsilon}(|u_{t_{\rho_1}}-u_{t_{\rho_1+1}}|) 
			\prod_{\rho_2=0}^{h-1}
			I_{\epsilon}(|u_{s_{\rho_2}}-u_{s_{\rho_2+1}}|) -\left(\omega_{h}^0\right)^2\right]\\
			&=\sum_{\mbox{\tiny{$\begin{array}{c}
							^{\#}{\left(A_{t}^{h}\cap{A_{s}^{h}}\right)}=0\end{array}$}}} \left[\mathbf{E}\prod_{\rho_1=0}^{h-1} I_{\epsilon}(|u_{t_{\rho_1}}-u_{t_{\rho_1+1}}|)
			\mathbf{E}\prod_{\rho_2=0}^{h-1} I_{\epsilon}(|u_{s_{\rho_2}}-u_{s_{\rho_2+1}}|)-\left(\omega_{h}^0\right)^2\right]\\
			&=0.
		\end{aligned}
	\end{equation}
	When $^{\#}{\left(A_{t}^{h}\cap{A_{s}^{h}}\right)}=j,j=1,\cdots,h$,  there are $2h+2-j,j=1,\cdots,h$ points that are free because $^{\#}{A_{t}^{h}}=^{\#}{A_{s}^{h}}=h+1$. Hence, for fixed $j$, there are $O(T^{2h+2-j})$ terms included in this case. Further we have:
	\begin{equation}\label{Eomegah2}
		\begin{aligned}
			&\quad \sum_{\mbox{\tiny{$\begin{array}{c}
							j=1\\
							^{\#}{\left(A_{t}^{h}\cap{A_{s}^{h}}\right)}=j\end{array}$}}}^{h+1} \mathbf{E}\left[\prod_{\rho_1=0}^{h-1} I_{\epsilon}(|u_{t_{\rho_1}}-u_{t_{\rho_1+1}}|) \prod_{\rho_2=0}^{h-1} I_{\epsilon}(|u_{s_{\rho_2}}-u_{s_{\rho_2+1}}|) -\left(\omega_{h}^0\right)^2\right]\\
			&=\sum_{j=1}^{h+1}O\left(T^{2h+2-j}\right)=O\left(T^{2h+1}\right),
		\end{aligned}
	\end{equation}
	Based on \eqref{Eomegah00}, \eqref{Eomegah1} and \eqref{Eomegah2}, we have
	\begin{equation}\label{Eomegah1h1}
		\mathbf{E}\left[\left(\widehat{\omega}_{h}^0\right) -\left(\omega_{h}^0\right)\right]^2=O\left(T^{-1}\right).
	\end{equation}
	With the same argument, we get
	\begin{equation}\label{Eomegah1h2}
		\mathbf{E} \left[\left(\widehat{\omega}_{h_{1}}^0\right) -\left(\omega_{h_{1}}^0\right)\right] \left[\left(\widehat{\omega}_{h_{2}}^0\right) -\left(\omega_{h_{2}}^0\right)\right] =O\left(T^{-1}\right).
	\end{equation}
	According to \eqref{V20}, \eqref{Eomegah1h1} and \eqref{Eomegah1h2}, we have
	\begin{equation}\label{V2}
		V_{2}=O\left(T^{-3}\right).
	\end{equation}
	Analogously, we obain
	\begin{equation}\label{V3}
		V_{3}=O\left(T^{-3}\right).
	\end{equation}
	Hence, $V_8\leq{V_2+V_3}=O(T^{-3})$. 
	
	For $V_{4}$, we just need to focus on  $\mathbf{E}\left[\left(\widehat{\omega}_{1}^0\right) -\left(\omega_{1}^0\right)\right]^2$, where $$\left(\widehat{\omega}_{1}^0\right)=\frac{1}{T(T-1)} \sum\limits_{i\neq{j}}I_{\epsilon}(|u_i-u_j|).$$ 
	Referring to the above skills, we divide $\mathbf{E}\left[\left(\widehat{\omega}_{1}^0\right) -\left(\omega_{1}^0\right)\right]^2$ into the following three parts according to the different values of $^{\#}{\left(A_{1}^2\cap{A_{2}^2}\right)}$:
	\begin{equation*}
		\begin{aligned}
			&\mathbf{E}\left[\left(\widehat{\omega}_{1}^0\right) -\left(\omega_{1}^0\right)\right]^2\\
			&=C_{T,1}^2\left\{\quad
			\sum_{^{\#}{\left(A_{1}^2\cap{A_{2}^2}\right)}=0}\left[\mathbf{E} I_{\epsilon}(|u_{i_1}-u_{j_1}|)I_{\epsilon}(|u_{i_{2}}-u_{j_{2}}|) -\left(\omega_{1}^0\right)^2\right] \right.\\
			& \left. \qquad \qquad \qquad \quad +\sum_{^{\#}{\left(A_{1}^2\cap{A_{2}^2}\right)}=1}\left[\mathbf{E} I_{\epsilon}(|u_{i_1}-u_{j_1}|)I_{\epsilon}(|u_{i_{2}}-u_{j_{2}}|)-\left(\omega_{1}^0\right)^2\right] \right. \\
			& \left. \qquad \qquad \qquad \quad +\sum_{^{\#}{\left(A_{1}^2\cap{A_{2}^2}\right)}=2}\left[\mathbf{E} I_{\epsilon}(|u_{i_1}-u_{j_1}|)I_{\epsilon}(|u_{i_{2}}-u_{j_{2}}|) -\left(\omega_{1}^0\right)^2\right]\right\},
		\end{aligned}
	\end{equation*}
	where $A_{1}^2=\{i_{1},j_{1}\}$ and $A_{2}^{2}=\{i_{2},j_{2}\}$, which correspond to $\{u_{i_1},u_{j_1}\}$ and $\{u_{i_2},u_{j_2}\}$.
	Easy to find that $^{\#}{\left(A_{1}^2\cap{A_{1}^2}\right)=0}$ indicates $\{u_{i_{1}},u_{j_{1}}\}$ is independent of $\{u_{i_{2}},u_{j_{2}}\}$, so that
	\begin{equation}\label{Eomega11}
		\begin{aligned}
			&\quad \sum_{^{\#}{\left(A_{1}^2\cap{A_{2}^2}\right)}=0}\left[\mathbf{E} I_{\epsilon}(|u_{i_1}-u_{j_1}|)I_{\epsilon}(|u_{i_{2}}-u_{j_{2}}|) -\left(\omega_{1}^0\right)^2\right]\\
			&=\sum_{^{\#}{\left(A_{1}^2\cap{A_{2}^2}\right)}=0}\left[\mathbf{E} I_{\epsilon}(|u_{i_1}-u_{j_1}|) \mathbf{E}I_{\epsilon}(|u_{i_{2}}-u_{j_{2}}|) -\left(\omega_{1}^0\right)^2\right]\\
			&=0.
		\end{aligned}
	\end{equation}
	For the sum of all terms under $^{\#}{\left(A_{1}^2\cap{A_{2}^2}\right)}=1$ and $^{\#}{\left(A_{1}^2\cap{A_{2}^2}\right)}=2$, through some elementary calculations, we obtain:
	\begin{equation}\label{Eomega12}
		\begin{aligned}
			&\quad \sum_{^{\#}{\left(A_{1}^2\cap{A_{2}^2}\right)}=2}\left[\mathbf{E} I_{\epsilon}(|u_{i_1}-u_{j_1}|)I_{\epsilon}(|u_{i_{2}}-u_{j_{2}}|) -\left(\omega_{1}^0\right)^2\right] \\
			&=\frac{2}{C_{T,1}}\left[\left(\omega_{1}^0\right) -\left(\omega_{1}^0\right)^2\right],
		\end{aligned}
	\end{equation}
	\begin{equation}\label{Eomega13}
		\begin{aligned}
			&\quad \sum_{^{\#}{\left(A_{1}^2\cap{A_{2}^2}\right)}=1}\left[\mathbf{E} I_{\epsilon}(|u_{i_1}-u_{j_1}|)I_{\epsilon}(|u_{i_{2}}-u_{j_{2}}|) -\left(\omega_{1}^0\right)^2\right] \\
			&=\frac{4}{C_{T,2}}\left[\left(\omega_{2}^0\right) -\left(\omega_{1}^0\right)^2\right].
		\end{aligned}
	\end{equation}
	Combining \eqref{Eomega11}, \eqref{Eomega12} and \eqref{Eomega13}, we obtain
	\begin{equation}\label{Eomega1}
		\begin{aligned}
			\mathbf{E}\left[\left(\widehat{\omega}_{1}^0\right) -\left(\omega_{1}^0\right)\right]^2=2C_{T,1}\left[\left(\omega_{1}^0\right)+2(T-2)\left(\omega_{2}^0\right) -(2T-3) \left(\omega_{1}^0\right)^2\right].
		\end{aligned}
	\end{equation}
	Furthermore, we have
	\begin{equation}\label{V4sigma11}
		\begin{aligned}
			V_{4}&=\sigma_{1,1}^2=2\left(\frac{N_{0}}{N}W_m^{(1)}(m,0)\right)^2C_{T,1} \left[\left(\omega_{1}^0\right)+2(T-2)\left(\omega_{2}^0\right) -(2T-3) \left(\omega_{1}^0\right)^2\right].
		\end{aligned}
	\end{equation}
	
	The calculation of $V_{5}$ and $V_6$ are completely similar. Thus, to avoid redundancy, we only introduce the calculation of $V_5$ in detail and directly give the result of $V_6$. As for $V_5$, based on \eqref{DivitionCI}, we divide $\mathbf{E}\left[C_{m,T}(\epsilon)-\mu_{m}\right] \left[\left(\widehat{\omega}_{h}^0\right) -\left(\omega_{h}^0\right)\right]$ into the following two parts, 
	\begin{equation}\label{DivitionV5}
		\begin{aligned}
			&\quad \mathbf{E}\left[C_{m,T}(\epsilon)-\mathbf{E}C_{m,T}(\epsilon)\right] \left[\left(\widehat{\omega}_{h}^0\right) -\left(\omega_{h}^0\right)\right]\\
			&=\mathbf{E}\left[\breve{C}_{m,T}(\epsilon) -\mathbf{E}\breve{C}_{m,T}(\epsilon)\right] \left[\left(\widehat{\omega}_{h}^0\right) -\left(\omega_{h}^0\right)\right]\\ &+\frac{N_{0}}{N}\mathbf{E}\left[\widetilde{C}_{m,T}(\epsilon) -\mathbf{E}\widetilde{C}_{m,T}(\epsilon)\right] \left[\left(\widehat{\omega}_{h}^0\right) -\left(\omega_{h}^0\right)\right].
		\end{aligned}
	\end{equation}
	For the first part $\mathbf{E}\left[\breve{C}_{m,T}(\epsilon) -\mathbf{E}\breve{C}_{m,T}(\epsilon)\right] \left[\left(\widehat{\omega}_{h}^0\right) -\left(\omega_{h}^0\right)\right]$, combing with \eqref{breveCI} and \eqref{EbreveCI}, we have
	\begin{equation}\label{EbreCIomegah0}
		\begin{aligned}
			&\quad\mathbf{E}\left[\breve{C}_{m,T}(\epsilon) -\mathbf{E}\breve{C}_{m,T}(\epsilon)\right] \left[\left(\widehat{\omega}_{h}^0\right) -\left(\omega_{h}^0\right)\right]\\
			&=\mathbf{E}\left[\frac{1}{N}\sum_{k=1}^{m-1}\sum_{t=1}^{T_m-k} \left(I_{\epsilon}(\|Y_{t}^m-Y_{t+k}^m\|) -W_m(k,0)\right)\right] \\
			&\quad \times{}\left[C_{T,h} \sum\limits_{\alpha_{1},\cdots,\alpha_{h} \atop\mbox{\tiny distrinct}} \left(\prod_{\rho=1}^{h+1} I_{\epsilon}(|u_{\alpha_{\rho}}-u_{\alpha_{\rho+1}}|) -\left(\omega_{h}^0\right)\right)\right]\\
		\end{aligned}
	\end{equation}
	Similar to the previous, we correspond all random variables invovled in 
	$$\left\{Y_{t}^m,Y_{t+k}^m\right\} \text{and} \left\{u_{\alpha_1},\cdots, u_{\alpha_{h+1}}\right\}$$ to sets
	$
	A_{t}^{m+k}=\{t,t+1,\cdots,t+k+m-1\}$ and $ B_{\alpha}^{h+1}=\{\alpha_{1},\cdots,\alpha_{h+1}\}$ by turn, where $^{\#}{\left(A_{t}^{m+k}\right)}=m+k$, $^{\#}{\left(B_{\alpha}^{h+1}\right)}=h+1$.
	Then \eqref{EbreCIomegah0} can be rewritten as
	\begin{equation*}
		\begin{aligned}
			&\mathbf{E}\left[\breve{C}_{m,T}(\epsilon) -\mathbf{E}\breve{C}_{m,T}(\epsilon)\right] \left[\left(\widehat{\omega}_{h}^0\right) -\left(\omega_{h}^0\right)\right]\\
			&=\frac{C_{T,h}}{N}\left\{ \sum_{k=1}^{m-1}\sum_{\mbox{\tiny{$\begin{array}{c}
							j=0\\
							^{\#}{\left(A_{t}^{m+k}\cap{B_{\alpha}^{h+1}}\right)}=j \end{array}$}}}^{(m+k)\wedge{(h+1)}} \mathbf{E}\left[ I_{\epsilon}(\|Y_{t}^m-Y_{t+k}^m\|) \prod_{\rho=1}^h I_{\epsilon}(|u_{\alpha_{\rho}}-u_{\alpha_{\rho+1}}|)\right. \right. \\
			& \qquad \qquad \qquad \qquad \qquad \qquad \qquad \qquad \qquad \qquad \qquad  -\left(\omega_{h}^0\right)W_m(k,0)\Bigg]\Bigg\},
		\end{aligned}
	\end{equation*}
	among them, when $^{\#}{\left(A_{t}^{m+k}\cap{B_{\alpha}^{h+1}}\right)}=0$, $\left\{Y_{t}^m,Y_{t+k}^m\right\}$ is independent of $\left\{u_{\alpha_{1}},\cdots,u_{\alpha_{h+1}}\right\}$, so
	\begin{equation}\label{EbreCIomegah1}
		\begin{aligned}
			&\quad 
			\sum_{\mbox{\tiny{$\begin{array}{c}
							^{\#}{\left(A_{t}^{m+k}\cap{B_{\alpha}^{h+1}}\right)}=0\end{array}$}}} \left[\mathbf{E}I_{\epsilon}(\|Y_{t}^m-Y_{t+k}^m\|)\prod_{\rho=1}^{h} I_{\epsilon}(|u_{\alpha_{\rho}}-u_{\alpha_{\rho+1}}|) -W_m(k,0)\left(\omega_{h}^0\right)\right] \\
			&=\sum_{\mbox{\tiny{$\begin{array}{c}
							^{\#}{\left(A_{t}^{m+k}\cap{B_{\alpha}^{h+1}}\right)}=0\end{array}$}}} \left[\mathbf{E}I_{\epsilon}(\|Y_{t}^m-Y_{t+k}^m\|) \mathbf{E}\prod_{\rho=1}^{h} I_{\epsilon}(|u_{\alpha_{\rho}}-u_{\alpha_{\rho+1}}|)
			-\left(\omega_{h}^0\right)W_m(k,0)\right]\\
			&=0.
		\end{aligned}
	\end{equation}
	On the other hand, when $^{\#}{\left(A_{t}^{m+k}\cap{B_{\alpha}^{h+1}}\right)}=j, j=1,\cdots,(m+k)\wedge{(h+1)}$, there's $h+2-j,j=1,\cdots,(m+k)\wedge{(h+1)}$ points that are free, hence we get
	\begin{equation}\label{EbreCIomegah2}
		\begin{aligned}
			&\quad  \sum_{\mbox{\tiny{$\begin{array}{c}
							j=1\\
							^{\#}{\left(A_{t}^{m+k}\cap{B_{\alpha}^{h+1}}\right)}=j\end{array}$}}}^{(m+k)\wedge{(h+1)}} \left[\mathbf{E}I_{\epsilon}(\|Y_{t}^m-Y_{t+k}^m\|) \prod_{\rho=1}^{h} I_{\epsilon}(|u_{\alpha_{\rho}}-u_{\alpha_{\rho+1}}|) -\left(\omega_{h}^0\right)W_m(k,0)\right] \\
			&=  \sum_{j=1}^{(m+k)\wedge{(h+1)}}O\left(T^{h+2-j}\right)\\
			&=O\left(T^{h+1}\right)
		\end{aligned}
	\end{equation}
	Based on \eqref{EbreCIomegah1} and \eqref{EbreCIomegah2}, we have
	\begin{equation}\label{EbreCIomegah}
		\begin{aligned}
			\quad \mathbf{E}\left[\breve{C}_{m,T}(\epsilon) -\mathbf{E}\breve{C}_{m,T}(\epsilon)\right] \left[\left(\widehat{\omega}_{h}^0\right) -\left(\omega_{h}^0\right)\right]=O\left(T^{-2}\right).
		\end{aligned}
	\end{equation}
	Now, let's concern about the second part  $\mathbf{E}\left[\widetilde{C}_{m,T}-\left(\omega_{1}^0\right)^{m}\right] \left[\left(\hat{\omega}_{h}^0\right)-\left(\omega_{h}^0\right)\right]$, as detailed below
	\begin{equation}\label{EtilCIomegah0}
		\begin{aligned}
			&\quad\mathbf{E}\left[\widetilde{C}_{m,T}-W_m(m,0)\right] \left[\left(\hat{\omega}_{h}^0\right)-\left(\omega_{h}^0\right)\right]\\
			&=\mathbf{E}\left[\frac{1}{N_{0}}\sum_{t=1}^{T_{m}-m}\sum_{s=t+m}^{T_{m}} \left(I_{\epsilon}(\|Y_{t}^m-Y_{s}^m\|)-W_m(m,0)\right)\right] \\
			&\quad \times{}\left[C_{T,h} \sum\limits_{\alpha_{1},\cdots,\alpha_{h} \atop\mbox{\tiny distrinct}} \left(\prod_{\rho=1}^{h+1} I_{\epsilon}(|u_{\alpha_{\rho}}-u_{\alpha_{\rho+1}}|) -\left(\omega_{h}^0\right)\right)\right]
		\end{aligned}
	\end{equation}
	Similarly, we respectively make the set of all random variables involved in
	$\left\{Y_{t}^m,Y_{s}^m|s-t\geq{m}\right\}$, $\{Y_t^m\}$ and $\{Y_s^m\}$ correspond to three sets below:
	\begin{equation}
		\begin{aligned}
			&A_{t,s}^{2m}=\left\{t,t+1,\cdots,t+m-1,s,s+1,\cdots,s+m-1 |s-t\geq{m}\right\},\\
			&A_{t}^{m}=\left\{t,t+1,\cdots,t+m-1\right\}\quad \mbox{and} \quad A_{s}^{m}=\left\{s,s+1,\cdots,s+m-1\right\}.
		\end{aligned}
	\end{equation}
	Then \eqref{EtilCIomegah0} can be simplified as
	
	\begin{equation*}
		\begin{aligned}
			&\quad \mathbf{E}\left[\widetilde{C}_{m,T}(\epsilon) -W_m(m,0)\right] \left[\left(\widehat{\omega}_{h}^0\right) -\left(\omega_{h}^0\right)\right]\\
			&=\frac{C_{T,h}}{N_0}\sum_{\mbox{\tiny{$\begin{array}{c}
							j=0\\
							^{\#}{\left((A_{t,s}^{2m}\cap{B_{\alpha}^{h+1}}\right)}=j\end{array}$}}}^{(m)\wedge{(h+1)}} \left[\mathbf{E}I_{\epsilon}(\|Y_{t}^m-Y_{s}^m\|) \prod_{\rho=1}^{h} I_{\epsilon}(|u_{\alpha_{\rho}}-u_{\alpha_{\rho+1}}|) -W_m(m,0) \left(\omega_{h}^0\right)\right] .
		\end{aligned}
		\end {equation*}
		Among them, due to independence, it's easy to have 
		\begin{equation}\label{EtilCIomegah2_0}
		\begin{aligned}
			& \quad \sum_{\mbox{\tiny{$\begin{array}{c}
							^{\#}{\left(A_{t,s}^{2m}\cap{B_{\alpha}^{h+1}}\right)}=0\end{array}$}}} \left[\mathbf{E}I_{\epsilon}(\|Y_{t}^m-Y_{s}^m\|) \prod_{\rho=1}^{h} I_{\epsilon}(|u_{\alpha_{\rho}}-u_{\alpha_{\rho+1}}|)-W_m(m,0) \left(\omega_{h}^0\right)\right]\\
			&=0.
		\end{aligned}
		\end{equation} 
		And for cases when $^{\#}{\left(A_{t,s}^{2m}\cap{B_{\alpha}^{h+1}}\right)}=j, j=2,\cdots,(m)\wedge{(h+1)}$, there are $h+3-j$ points are free, thus, we have
		\begin{equation}\label{EtilCIomegah2_2}
		\begin{aligned}
			&\quad+\sum_{\mbox{\tiny{$\begin{array}{c}
							j=2\\
							^{\#}{\left(A_{t,s}^{2m}\cap{B_{\alpha}^{h+1}}\right)}=j\end{array}$}}}^{(m)\wedge{(h+1)}} \left[\mathbf{E}I_{\epsilon}(\|Y_{t}^m-Y_{s}^m\|) \prod_{\rho=1}^{h} I_{\epsilon}(|u_{\alpha_{\rho}}-u_{\alpha_{\rho+1}}|) -W_m(m,0) \left(\omega_{h}^0\right)\right]\\
			&=O(T^{h+1}).
		\end{aligned}
		\end{equation}
		After some routine calculation, we obtain
		\begin{equation}
			\begin{aligned}
				&\quad\sum_{\mbox{\tiny{$\begin{array}{c}
								^{\#}{\left(A_{t,s}^{2m}\cap{B_{\alpha}^{h+1}}\right)}=1\end{array}$}}} \left[\mathbf{E}I_{\epsilon}(\|Y_{t}^m-Y_{s}^m\|) \prod_{\rho=1}^{h} I_{\epsilon}(|u_{\alpha_{\rho}}-u_{\alpha_{\rho+1}}|)  -W_m(m,0) \left(\omega_{h}^0\right)\right] \\
				&=2M_{h+2} \left[W_m^{(1)}(m,0)\left( 2\left(\omega_{h}^1\right) +\sum_{i=1}^{h-1}\left(\xi_{h}^i\right)\right) -m(h+1)W_m(m,0) \left(\omega_{h}^0\right)\right]
			\end{aligned}
		\end{equation}
		where $M_{h+2}=\frac{(T-2m+2)!}{(T-2m+2-(h+2))!}$. Then, we get 
		\begin{equation} \label{EtilCIomegah2}
		\begin{aligned}
			&\mathbf{E}\left[\widetilde{C}_{m,T}(\epsilon)-W_m(m,0)\right] \left[\left(\widehat{\omega}_{h}^0\right)-\left(\omega_{h}^0\right)\right]\\
			&=\frac{2C_{T,h}M_{h+2}}{N_0} \left[W_m^{(1)}(m,0)\left( 2\left(\omega_{h}^1\right) +\sum_{\kappa =1}^{h-1}\left(\xi_{h}^{\kappa}\right)\right) -m(h+1)W_m(m,0) \left(\omega_{h}^0\right)\right]\\
			&\quad +O(T^{-2}).
		\end{aligned}
		\end{equation} 
		Further, 
		based on \eqref{EbreCIomegah}, \eqref{EtilCIomegah2} and \eqref{DivitionV5}, we obtain
		\begin{equation}\label{V5}
			\begin{aligned}
				V_{5}&=\frac{2}{N}\sum_{k=1}^{m-1}(T_m-k)W_m^{(h)}(k,0)\mathbf{E}\left[C_{m,T}(\epsilon)-\mu_{m}\right] \left[\left(\widehat{\omega}_{h}^0\right)-\left(\omega_{h}^0\right)\right]\\
				&=\sigma_{m,h}^2+O\left(T^{-3}\right),
			\end{aligned}
		\end{equation}
		where
		\begin{equation*}
			\begin{aligned}
				\sigma_{m,h}^2&=\frac{4}{N^2}\sum_{k=1}^{m-1}(T_m-k)W_m^{(h)}(k,0)C_{T,h}M_{h+2}\\
				&\quad  \times{}  \left[W_m^{(1)}(m,0)\left(2\left(\omega_{h}^1\right)+\sum_{\kappa=1}^{h-1}\left(\xi_{h}^{\kappa}\right)\right)-m(h+1)W_m(m,0) \left(\omega_{h}^0\right) \right],
			\end{aligned}
		\end{equation*}
		$T_{m}=T-m+1$, $h=\left\lfloor{\frac{m}{k}}\right\rfloor$, $N=\binom{T_{m}}{2}$, $C_{T,h}=\frac{(T-1-h)!}{T!}$, $M_{h+2}=\frac{(T-2m+2)!}{(T-2m+2-(h+2))!}$.
		
		By the same argument, we obtain
		\begin{equation}\label{V6}
			\begin{aligned}
				V_{6}&=\frac{2}{N}\sum_{k=1}^{m-1}(T_m-k)W_m^{(h+1)}(k,0) \mathbf{E}\left[C_{m,T}(\epsilon)-\mu_{m}\right] \left[\left(\widehat{\omega}_{h+1}^0\right)-\left(\omega_{h+1}^0\right)\right]\\
				&=\sigma_{m,h+1}^{2}+O\left(T^{-3}\right),
			\end{aligned}
		\end{equation}
		where
		\begin{equation}\label{sigmamh1}
			\begin{aligned}
				\sigma_{m,h+1}^2&=\frac{4}{N^2}\sum_{k=1}^{m-1}(T_m-k)W_m^{(h+1)}(k,0)C_{T,h+1}M_{h+3}\\
				&\ \times{}  \left[W_m^{(1)}(m,0)\left(2\left(\omega_{h+1}^1\right)+  \sum_{\kappa =1}^{h}\left(\xi_{h+1}^{\kappa}\right)\right)  -m(h+2)W_m(m,0) \left(\omega_{h+1}^0\right)\right],
			\end{aligned}
		\end{equation}
		$T_{m}=T-m+1$, $h=\left\lfloor{\frac{m}{k}}\right\rfloor$, $N=\binom{T_{m}}{2}$,  $C_{T,h+1}=\frac{(T-2-h)!}{T!}$, $M_{h+3}=\frac{(T-2m+2)!}{(T-2m+2-(h+3))!}$.
		
		Following \eqref{DivitionV5}, to get the result of $$V_{7}=\frac{2N_{0}}{N}W_m^{(1)}(m,0) \mathbf{E}\left[C_{m,T}(\epsilon)-\mu_{m}\right] \left[\left( \widehat{\omega}_{1}^0\right) -\left( \omega_{1}^0\right)\right],$$  we  only need to care about
		\begin{equation}\label{EbreCIomega10}
			\mathbf{E}\left[\breve{C}_{m,T}(\epsilon) -\mathbf{E}\breve{C}_{m,T}(\epsilon)\right] \left[\left(\widehat{\omega}_{1}^0\right) -\left(\omega_{1}^0\right)\right],
		\end{equation}
		and
		\begin{equation}\label{EtilCIomega10}
			\mathbf{E}\left[\widetilde{C}_{m,T}(\epsilon) -\mathbf{E}\widetilde{C}_{m,T}(\epsilon)\right] \left[\left(\widehat{\omega}_{1}^0 \right)-\left(\omega_{1}^0\right)\right].
		\end{equation}
		Among them,
		\begin{equation}\label{EbreCIomega100}
			\begin{aligned}
				&\quad\mathbf{E}\left[\breve{C}_{m,T}(\epsilon) -\mathbf{E}\breve{C}_{m,T}(\epsilon)\right] \left[\left(\widehat{\omega}_{1}^0\right) -\left(\omega_{1}^0\right)\right]\\
				&=\mathbf{E}\left[\frac{1}{N}\sum_{k=1}^{m-1}\sum_{t=1}^{T_{m}-k} \left(I_{\epsilon}(\|Y_{t}^m-Y_{t+k}^m\|) -W_m(k,0)\right) \right] \\
				&\quad \times{}\left[C_{T,1} \sum\limits_{\alpha_{1}\neq{\alpha_{2}}} \left[I_{\epsilon}(|u_{\alpha_{1}}-u_{\alpha_{2}}|) -\left(\omega_{1}^0\right)\right] \right. \\
				&=\frac{C_{T,1}}{N}\left\{\sum_{k=1}^{m-1}\sum_{t=1}^{T_{m}-k} \sum\limits_{\alpha_{1}\neq{\alpha_{2}}}\left[\mathbf{E} I_{\epsilon}(\|Y_{t}^m-Y_{t+k}^m\|) I_{\epsilon}(|u_{\alpha_{1}}-u_{\alpha_{2}}|) -\left(\omega_{1}^0\right)W_m(k,0)\right]\right\}\\
				&=\frac{C_{T,1}}{N}\left\{\sum_{k=1}^{m-1} \sum_{\mbox{\tiny{$\begin{array}{c}
								j=0\\
								^{\#}{\left(A_{t}^{m+k}\cap{B_{\alpha}^{2}}\right)}=j\end{array}$}}}^{2} \left[\mathbf{E}I_{\epsilon}(\|Y_{t}^m-Y_{t+k}^m\|) I_{\epsilon}(|u_{\alpha_{1}}-u_{\alpha_{2}}|)   -\left(\omega_{1}^0\right)W_m(k,0)\right] \right\},
			\end{aligned}
		\end{equation}
		where $A_{t}^{m+k}=\{t,t+1,\cdots,t+k+m-1\}$, $B_{\alpha}^2=\{\alpha_{1},\alpha_{2}\}$, $^{\#}{\left(A_{t}^{m+k}\right)}=m+k$, $^{\#}{\left(B_{\alpha}^{2}\right)}=2$.
		With routine calculation, it's not difficult to get:
		\begin{equation}\label{EbreCIomega11}
			\begin{aligned}
				&\sum_{\mbox{\tiny{$\begin{array}{c}
								^{\#}{\left(A_{t}^{m+k}\cap{B_{\alpha}^{2}}\right)}=1\end{array}$}}} \left[\mathbf{E}I_{\epsilon}(\|Y_{t}^m-Y_{t+k}^m\|) I_{\epsilon}(|u_{\alpha_{1}}-u_{\alpha_{2}}|) -W_m(k,0)\left(\omega_{1}^0\right)\right]\\
				&=2\mathcal{N}_{T_m,k}\left\{ W_m^{(h)}(k,0)\left[2\left(\omega_h^1\right)+\sum\limits_{\kappa=1}^{h_1}\left(\xi_h^{\kappa}\right)+i_1\left(\xi_h^{h-1}\right) \right] \right.\\
				&\left. +W_m^{(h+1)}(k,0) \left[2\left(\omega_{h+1}^1\right)+\sum\limits_{\kappa=1}^{h_1}\left(\xi_{h+1}^{\kappa}\right)\right] -(m-k)W_m(k,0)\left(\omega_1^0\right)\right\},
			\end{aligned}
		\end{equation}
		where $\mathcal{N}_{T_m,k}=\binom{T_m-k}{2} $, $h_1=\lfloor{\frac{m-k}{k}}\rfloor$, $i_1=m-k-h_1k$
		and
		\begin{equation}\label{EbreCIomega12}
			\begin{aligned}
				&\sum_{\mbox{\tiny{$\begin{array}{c}
								^{\#}{\left(A_{t}^{m+k}\cap{B_{\alpha}^{2}}\right)}=2\end{array}$}}} \left[\mathbf{E}I_{\epsilon}(\|Y_{t}^m-Y_{t+k}^m\|) I_{\epsilon}(|u_{\alpha_{1}}-u_{\alpha_{2}}|) -\omega_{1}^0W_m(k,0)\right] \\
				&=O(T).
			\end{aligned}
		\end{equation}
		
		In addition,
		\begin{equation}\label{EtildeCIomega0}
			\begin{aligned}
				&\quad\mathbf{E}\left[\widetilde{C}_{m,T}(\epsilon) -\mathbf{E}\widetilde{C}_{m,T}(\epsilon)\right] \left[\left(\widehat{\omega}_{1}^0\right)-\left(\omega_{1}^0\right)\right]\\
				&=\mathbf{E}\left[\frac{1}{N_{0}}\sum_{t=1}^{T_{m}-m} \sum_{s=t+m}^{T_{m}} \left(I_{\epsilon}(\|Y_{t}^m-Y_{s}^m\|) -\left(\omega_{1}^0\right)^{m}\right)\right] \\
				&\quad \times{}\left[C_{T,1}\sum\limits_{\alpha_{1}\neq{\alpha_{2}}} \left(I_{\epsilon}(|u_{\alpha_{1}}-u_{\alpha_{2}}|) -\left(\omega_{1}^0\right)\right)\right]\\
				&=\frac{C_{T,1}}{N_0}\sum_{t=1}^{T_{m}-m}\sum_{s=t+m}^{T_{m}} \sum\limits_{\alpha_{1}\neq{\alpha_{2}}}\left[\mathbf{E} I_{\epsilon}(\|Y_{t}^m-Y_{s}^m\|) I_{\epsilon}(|u_{\alpha_{1}}-u_{\alpha_{2}}|)  -\left(\omega_{1}^0\right)^{m+1}\right]\\
				&=\frac{C_{T,1}}{N_0}\sum_{\mbox{\tiny{$\begin{array}{c}
								j=0\\
								^{\#}{\left(A_{t,s}^{2m}\cap{B_{\alpha}^{2}}\right)}=j \end{array}$}}}^2\left[ \mathbf{E}I_{\epsilon}(\|Y_{t}^m-Y_{s}^m\|) I_{\epsilon}(|u_{\alpha_{1}}-u_{\alpha_{2}}|)  -\left(\omega_{1}^0\right)^{m+1}\right],
			\end{aligned}
		\end{equation}
		where $A_{t,s}^{2m}=\left\{t,t+1,\cdots,t+m-1,s,s+1,\cdots,s+m-1 |s-t\geq{m}\right\}$, $B_{\alpha}^{2}=\left\{\alpha_{1},\alpha_{2}\right\}$, $^{\#}{\left(A_{t,s}^{2m}\right)}=2m$, $^{\#}{\left(B_{\alpha}^{2}\right)}=2$,
		
		\begin{equation}\label{EtildeCIomega11}
			\begin{aligned}
				&\quad\sum_{\mbox{\tiny{$\begin{array}{c}
								^{\#}{\left(A_{t,s}^{2m}\cap{B_{\alpha}^{2}}\right)}=1\end{array}$}}} \left[ \mathbf{E}I_{\epsilon}(\|Y_{t}^m-Y_{s}^m\|) I_{\epsilon}(|u_{\alpha_{1}}-u_{\alpha_{2}}|)  -\left(\omega_{1}^0\right)^{m+1}\right]\\
				&=4m(T-2m+2)(T-2m+1)(T-2m)\left[\left(\omega_1^0\right)^{m-1}\left(\omega_1^1\right)-\left(\omega_1^0\right)^{m+1}\right]\\
				&=4M_3\left[W_m^{(1)}(m,0)\left(\omega_1^1\right)-mW_m(m,0)\left(\omega_1^0\right)\right]
			\end{aligned}
		\end{equation}
		\begin{equation}\label{EtildeCIomega12}
			\begin{aligned}
				&\quad\sum_{\mbox{\tiny{$\begin{array}{c}
								^{\#}{\left(A_{t,s}^{2m}\cap{B_{\alpha}^{2}}\right)}=2\end{array}$}}} \left[ \mathbf{E}I_{\epsilon}(\|Y_{t}^m-Y_{s}^m\|) I_{\epsilon}(|u_{\alpha_{1}}-u_{\alpha_{2}}|)  -\left(\omega_{1}^0\right)^{m+1}\right]\\
				&=M_2\left\{2m\left[\left(\omega_1^0\right)^m-\left(\omega_1^0\right)^{m+1}\right]+ 2m(m-1)\left[\left(\omega_1^0\right)^{m-1}\left(\omega_1^1\right)-\left(\omega_1^0\right)^{m+1}\right] \right. \\
				&\left. \qquad \quad  +2m(m-1)\left[\left(\eta_1^2\right)\left(\omega_1^0\right)^{m-2}-\left(\omega_1^0\right)^{m+1}\right] \right\}\\
				&=M_2\left\{2W_m^{(1)}(m,0)\left[\left(\omega_1^0\right)+(m-1)\left(\left(\omega_1^1\right)+\left(\eta_1^2\right)\left(\omega_1^0\right)^{-1}\right)\right] 
				-(4m-2)W_m(m,0)\left(\omega_1^0\right)
				\right\}.
			\end{aligned}
		\end{equation}
		Thus, we obtain 
		\begin{equation}\label{ECmomega1}
		\begin{aligned}
			&\mathbf{E}\left[C_{m,T}(\epsilon) -\mathbf{E}C_{m,T}(\epsilon)\right] \left[\left(\widehat{\omega}_{1}^0\right)-\left(\omega_{1}^0\right)\right]\\
			&=\frac{C_{T,1}}{N}\sum\limits_{k=1}^{m-1}2\mathcal{N}_{T_m,k}\left\{ W_m^{(h)}(k,0)\left[2\left(\omega_h^1\right)+\sum\limits_{\kappa=1}^{h_1}\left(\xi_h^{\kappa}\right)+\left(\xi_h^{h-1}\right) \right] \right.\\
			&\qquad  \qquad \qquad \qquad\qquad +W_m^{(h+1)}(k,0) \left[2\left(\omega_{h+1}^1\right)+\sum\limits_{\kappa=1}^{h_1}\left(\xi_{h+1}^{\kappa}\right)\right]\\ 
			& \left. \qquad \qquad \qquad \qquad \qquad -(m-k)W_m(k,0)\left(\omega_1^0\right)\right\}\\
			&+\frac{C_{T,1}}{N_0}\left\{2W_m^{(1)}(m,0)\left[\left(2M_3+(m-1)M_2\right)\left(\omega_1^1\right)\right.\right.\\
			&\left. \qquad \qquad +M_2\left(\left(\omega_1^0\right)+(m-1)\left(\eta_1^2\right)\left(\omega_1^0\right)^{-1}\right)\right]\\
			&\left. \qquad \qquad  -W_m(m,0)\left[\omega_1^0\right)\left[4mM_3+(4m-2)M_2\right]
			\right\}.
		\end{aligned}
		\end{equation}
		So far, we have
		\begin{equation}\label{V7}
			\begin{aligned}
				V_{7}=\sigma_{m,1}^2+O\left(T^{-3}\right),
			\end{aligned}
		\end{equation}
		where
		\begin{equation}\label{sigmam1}
			\sigma_{m,1}^2=\frac{2N_{0}}{N}W_m^{(1)}(m,0)
			\eqref{ECmomega1}.
		\end{equation}
		
		The calculation of $V_{9}$ and $V_{10}$ are similar. Next, we will only introduce that of $V_{9}$ in detail.
		\begin{equation}\label{V90}
			V_{9}=\frac{2N_{0}}{N^2}W_m^{(1)}(m,0) \sum_{k=1}^{m-1}(T_m-k)W_m^{(h)}(k,0) \mathbf{E}\left[\widehat{\omega}_{1}^0-\omega_{1}^0\right] \left[\left(\widehat{\omega}_{h}^0\right) -\left(\omega_{h}^0\right)\right],
		\end{equation}
		where
		\begin{equation}\label{Eomega1h0}
			\begin{aligned}
				&\quad \mathbf{E}\left[\widehat{\omega}_{1}^0-\omega_{1}^0\right] \left[\widehat{\omega}_{h}^0-\omega_{h}^0\right]\\
				&=\mathbf{E}\left(C_{T,1} \sum\limits_{\alpha_{1}\neq{\alpha_{2}}} \left[I_{\epsilon}(|u_{\alpha_{1}}-u_{\alpha_{2}}|) -\left(\omega_{1}^0\right)\right]\right)\\
				&\quad \times{} \left(C_{T,h}\sum \limits_{\beta_{1},\cdots,\beta_{h+1} \atop\mbox{\tiny distrinct}}\left[\prod_{\rho=1}^{h} I_{\epsilon}(|u_{\beta_{\rho}}-u_{\beta_{\rho+1}}|) -\left(\omega_{h}^0\right)\right]\right)\\
				&=C_{T,1}C_{T,h} \sum\limits_{\alpha_{1}\neq{\alpha_{2}}}
				\sum\limits_{\beta_{1},\cdots,\beta_{h+1} \atop\mbox{\tiny distrinct}} \left[\mathbf{E}I_{\epsilon}(|u_{\alpha_{1}}-u_{\alpha_{2}}|) \prod_{\rho=1}^{h} I_{\epsilon}(|u_{\beta_{\rho}}-u_{\beta_{\rho+1}}|) -\left(\omega_{1}^0\right)\left(\omega_{h}^0\right)\right]\\
				&=C_{T,1}C_{T,h} \sum_{\mbox{\tiny{$\begin{array}{c}
								j=0\\
								^{\#}{\left(A_{\alpha}^{2}\cap{B_{\beta}^{h+1}}\right)}=j\end{array}$}}}^{2} \left[\mathbf{E}I_{\epsilon}(|u_{\alpha_{1}}-u_{\alpha_{2}}|) \prod_{\rho=1}^{h} I_{\epsilon}(|u_{\beta_{\rho}}-u_{\beta_{\rho+1}}|) -\left(\omega_{1}^0\right)\left(\omega_{h}^0\right)\right],
			\end{aligned}
		\end{equation}
		where $
		A_{\alpha}^2=\{\alpha_{1},\alpha_{2}\}$, $^{\#}{\left(A_{\alpha}^{2}\right)}=2$ and $B_{\beta}^{h+1}=\{\beta_{1},\cdots,\beta_{h+1}\}$,
		$^{\#}{\left(B_{\beta}^{h+1}\right)}=h+1$. 
		Then, when $^{\#}{\left(A_{\alpha}^{2}\cap{B_{\beta}^{h+1}}\right)}=0$,$I_{\epsilon}(|u_{\alpha_{1}}-u_{\alpha_{2}}|)$ is independent of $\prod_{\rho=1}^{h}I_{\epsilon}(|u_{\rho}-u_{\rho+1}|)$, so that
		\begin{equation}\label{Eomega1h1}
			\begin{aligned}
				&\sum_{\mbox{\tiny{$\begin{array}{c}
								^{\#}{\left(A_{\alpha}\cap{B_{\beta}^{h+1}}\right)}=0\end{array}$}}} \left[\mathbf{E}I_{\epsilon}(|u_{\alpha_{1}}-u_{\alpha_{2}}|) \prod_{\rho=1}^{h} I_{\epsilon}(|u_{\beta_{\rho}}-u_{\beta_{\rho+1}}|) -\left(\omega_{1}^0\right)\left(\omega_{h}^0\right)\right]\\
				&=0.
			\end{aligned}
		\end{equation}
		When $^{\#}{\left(A_{\alpha}^{2}\cap{B_{\beta}^{h+1}}\right)}=2$, there's only $h+1$ points that are free, further, we have
		\begin{equation}\label{Eomega1h2}
			\begin{aligned}
				&\sum_{
					^{\#}{\left(A_{\alpha}^{2}\cap{B_{\beta}^{h+1}}\right)}=2} \left[\mathbf{E}I_{\epsilon}(|u_{\alpha_{1}}-u_{\alpha_{2}}|) \prod_{\rho=1}^{h} I_{\epsilon}(|u_{\beta_{\rho}}-u_{\beta_{\rho+1}}|) -\left(\omega_{1}^0\right)\left(\omega_{h}^0\right)\right]\\
				&=O(T^{h+1}).
			\end{aligned}
		\end{equation}
		With routine calculation, we get
		\begin{equation}\label{Eomega1h3}
			\begin{aligned}
				&\sum_{
					^{\#}{\left(A_{\alpha}^{2}\cap{B_{\beta}^{h+1}}\right)}=1} \left[\mathbf{E}I_{\epsilon}(|u_{\alpha_{1}}-u_{\alpha_{2}}|) \prod_{\rho=1}^{h} I_{\epsilon}(|u_{\beta_{\rho}}-u_{\beta_{\rho+1}}|) -\left(\omega_{1}^0\right)\left(\omega_{h}^0\right)\right]\\
				&=\frac{4}{C_{T,h+1}}\left[\left(\omega_{h}^1\right) -\left(\omega_{1}^0\right)\left(\omega_{h}^0\right)\right] +\frac{2}{C_{T,h+1}} \sum_{\kappa=1}^{h-1}\left[\left(\xi_{h}^{\kappa}\right) -\left(\omega_{1}^0\right)\left(\omega_{h}^0\right)\right]\\
				&=\frac{2}{C_{T,h+1}}\left[2\left(\omega_h^1\right)+\sum\limits_{\kappa=1}^{h-1}\left(\xi_h^{\kappa}\right)-(h+1)\left(\omega_1^0\right)\left(\omega_h^0\right)\right].
			\end{aligned}
		\end{equation}
		Based on \eqref{Eomega1h0}, \eqref{Eomega1h1}, \eqref{Eomega1h2} and \eqref{Eomega1h3}, we get
		\begin{equation}\label{V9}
			\begin{aligned}
				V_{9}=\sigma_{1,h}^2+O\left(T^{-3}\right),
			\end{aligned}
		\end{equation}
		where
		\begin{equation}\label{sigma1h}
			\begin{aligned}
				\sigma_{1,h}^2&=\frac{4N_{0}C_{T,1}}{N^2}W_m^{(1)}(m,0) \sum_{k=1}^{m-1}(T_m-k)W_m^{(h)}(k,0)C_{T,h}/C_{T,h+1}\\
				&\times{} \left[2\left(\omega_h^1\right)+\sum\limits_{\kappa=1}^{h-1}\left(\xi_h^{\kappa}\right)-(h+1)\left(\omega_1^0\right)\left(\omega_h^0\right)\right].
			\end{aligned}
		\end{equation}
		By the same argument of $V_{9}$, we have
		\begin{equation*}
			\begin{aligned}
				V_{10}&=\sigma_{1,h+1}^2+O\left(T^{-3}\right),
			\end{aligned}
		\end{equation*}
		where
		\begin{equation}\label{sigma1h1}
			\begin{aligned}
				\sigma_{1,h+1}^2&=\frac{4N_{0}C_{T,1}}{N^2}W_m^{(1)}(m,0) \sum_{k=1}^{m-1}(T_{m}-k)W_m^{(h+1)}(k,0)C_{T,h+1}/C_{T,h+2}\\
				&  \times{} \left[2\left(\omega_{h+1}^1\right)  +\sum_{\kappa=1}^{h}\left(\xi_{h+1}^{\kappa}\right) -(h+2)\left(\omega_{1}^0\right)\left(\omega_{h+1}^0\right)\right] .
			\end{aligned}
		\end{equation}
		So far, we obtain the asymptotic variance of $\mathcal{K}_{m,T}(\epsilon)$
		\begin{equation}
			\begin{aligned}
				\mathbf{Var}\left(\mathcal{K}_{m,T}(\epsilon)\right) &=\mathbf{Var}\left(\widetilde{\mathcal{K}}_{m,T}(\epsilon)\right)\\
				&=\sigma_{m,m}^2+\sigma_{1,1}^2-\sigma_{m,h}^2-\sigma_{m,h+1}^2 -\sigma_{m,1}^2+\sigma_{1,h}^2+\sigma_{1,h+1}^2.
			\end{aligned}
		\end{equation}
		Theorem \ref{ThCLTKCI} is proved.
		
	\end{proof}

\vspace{1cm} \noindent {\bf {References}} \small

%%%%%%%%%%%%参考文献请按作者姓氏排列, 请注意首字母的大小写, 论文名请用斜体表示 %%%%%%%%%%%%%%%%%%%%%%
\begin{itemize}
%\item [{[1]}] E\,L\,Green. \emph{title}, Sci China Ser A, year, volume(issue): 141-162.
\item[{[1]}] M Akintunde, J Oyekunle, G Olalude.  \emph{Detection of non-linearity in the time series using bds test}, Science Journal of Applied Mathematics and Statistics, 2015, 3(4): 184-187.
\item[{[2]}] E G Baek, W A Brock \emph{A nonparametric test for independence of a multivariate time
	series}, Statistica Sinica, 1992, 2(1): 137-156.
\item[{[3]}] Belaire-Franch J, Contreras D. \emph{How to compute the bds test: a software comparison}, Journal of Applied Econometrics, 2002, 17: 691-699.
\item [{[4]}] Brock W A, Dechert W, Scheinkman J A. \emph{A test  for independence based on the correlation dimension}, University of Wisconsin at Madison, University of Houston and University of Chicago 1987.
\item [{[5]}] Brock W A, Durlauf S N. \emph{Indentification of binary choice models with social interactions}, Journal of Econometrics, 2007, 140(1): 52-75.
\item[{[6]}] Broock W A, Scheinkman J A, Dechert W D, LeBaron B. \emph{A test for independence based on the correlation dimension}, Econometric Reviews, 1996, 15(3): 197-235.
\item[{[7]}] Brooks C, Heravi S M. \emph{The effect of (mis-specified) garch filters on the finite sample distribution of the bds test}, Computational Economics, 1999, 13(2): 147-162.
\item[{[8]}] C{\'a}novas J, Guillam{\'{o}}n A, Vera S. \emph{Testing for independence: Permutation based tests v.s. bds test.} The European Physical Journal Special Topics, 2013, 222(2): 275-284.
\item[{[9]}] Caporale G M, Ntantamis C, Pantelidis T, Pittis N. \emph{The bds test as a test for the adequacy of garch(1,1) specification: a monte carlo study.} Journal of Financial Econometrics, 2005, 3(2): 282-309.
\item[{[10]}] Chen Y T, Kuan C M. \emph{Time irreversibility and egarch effects in us stock index returns.} Journal of Applied Econometrics, 2002, 17(5): 565-578.
\item[{[11]}] Cryer J D, Chan K S. \emph{Time series analysis: with applications in R.} Springer, 2008.
\item[{[12]}] De Lima P. J. F. \emph{Nuisance parameter free properties of correlation integral based statistics. } Econometric Reviews, 1996, 15(3): 237-259.
\item[{[13]}] De Lima P. J. F. \emph{On the robustness of nonlinearity tests to moment condition failure.} Journal of Econometrics, 1997, 76(1-2): 251-280.
\item[{[14]}] Delgado M A. \emph{Testing serial independence using the sample distribution function.} Journal of Time Series Analysis, 1996, 17(3): 271-285.
\item[{[15]}] Disks C, Panchenko V. \emph{Nonparametric tests for serial independence based on quadratic forms}, Statistica Sinica, 2007, 17(1): 81-89, S1-S4.
\item[{[16]}] Fernandes M, Preumont P Y. \emph{The finite-sample size of the bds test for garch standardized residuals}, Brazilian Review of Econometrics, 2012, 32(2): 241-260.
\item[{[17]}] Genest C, Ghoudi K, R\'{e}millard B. \emph{Rank-based extensions of the brock, dechert and scheinkman test}, Journal of the American Statistical Association, 2007, 102(480): 1363-1376.
\item[{[18]}] Granger C W, Maasoumi E, Racine J. \emph{A dependence metric for possibly nonlinear processes}, Journal of Time Series Analysis, 2004, 25(5): 649-669.
\item[{[19]}] Hjellvik V, Tj{\o}stheim D. \emph{Nonparametric statistics for testing of linearity and serial independence}, Journal of Nonparametric Statistics, 1996, (2-3): 223-251.
\item[{[20]}] Hong Y. \emph{Testing for pairwise serial independence via the empirical distribution function}, Journal of the Royal Statistical Society: Series B (Statistical Methodology), 1998, 60(2): 429-453.
\item[{[21]}] Hsieh D A. \emph{Chaos and nonlinear dynamics: application to financial markets.} The Journal of Finance, 1991, 46(5): 1839-1877.
\item[{[22]}] Hui Y, Wong W K, Bai Z, Zhu Z. \emph{A new nonlinearity test to circumvent the limitation of volterra expansion with application}, Journal of the Korean Statistical Society, 2017, 46(3): 365-374.
\item[{[23]}] Jansen D W, De Vries C G. \emph{On the frequency of large stock returns: Putting booms and busts into perspective}, The Review of Economics and Statistics, 1991, 73(1): 18-24.
\item[{[24]}] Kanzler L. \emph{Very fast and correctly sized estimation of the bds statistic}, Available at SSRN: https://ssrn.com/abstract=151669 or http://dx.doi.org/10.2139/ssrn.151669.
\item[{[25]}] Ko{\v{c}}enda E. \emph{An alternative  to the bds test: integration across the correlation integral}, Econometric Reviews, 2001, 20(3)
: 337-351.
\item[{[26]}] Ko{\v c}enda E, Briatka L. \emph{Optimal range for the iid test based on integration across the correlation integral}, Econometric Reviews, 2005, 24(3): 265-296.
\item[{[27]}] Lai D. \emph{Asymptotic distribution of the estimated bds statistic from the residuals of location-scale type processes}, Statistics: A Journal of Theoretical and Applied Statistics, 2000, 34(2): 117-135.
\item[{[28]}] LeBaron B. \emph{A fast algorithm for the bds statistic}, Studies in Nonlinear Dynamics \& Econometrics 1997;2(2).
\item[{[29]}] Lee T H, White H, Granger C W. \emph{Testing for neglected nonlinearity in time series models: A comparison of neural network methods and alternative tests}, Journal of Econometrics, 1993, 56(3): 269-290.
\item[{[30]}] Loretan M, Phillips P C. \emph{Testing the covariance stationarity of heavy-tailed time series: An overview of the theory with applications to several financial datasets}, Journal of Empirical Finance, 1994, 1(2): 211-248.
\item[{[31]}] Luo W, Bai Z, Zheng S, Hui Y. \emph{A modified bds test},  Statistics \& Probability Letters, 2020; 164: 108794.
\item[{[32]}] Madhavan V. \emph{Nonlinearity in investment grade credit default swap (cds) indices of us and europe: evidence from bds and close-returns tests}, Global Finance Journal, 2013; 24(3): 266-279.
\item[{[33]}] Matilla-Garc{\'\i}a M, Mar{\'\i}n M R. \emph{A non-parametric independence test using permutation entropy}, Journal of Econometrics, 2008; 144(1): 139-155.
\item[{[34]}] Matilla-Garc{\'\i}a M, Mar{\'\i}n M R. \emph{A new test for chaos and determinism based on symbolic dynamics}, Journal of Economic Behavior \& Organization, 2010; 76(3): 600?614.
\item[{[35]}] Matilla-Garc{\'\i}a M, Queralt R, Sanz P, V{\'a}zquez F. \emph{A generalized bds statistic}, Computational Economics, 2004a, 24(3): 277-300.
\item[{[36]}] Matilla-Garc{\'\i}a M, Sanz P, V{\'a}zquez F. \emph{Dimension estimation with the bds-g statistic}, Applied Economics, 2004b, 36(11): 1219-1223.
\item[{[37]}] Matilla-Garc{\'\i}a M, Sanz P, V{\'a}zquez F. \emph{The bds test and delay time}, Applied Economics Letters, 2005, 12(2): 109-113.
\item[{[38]}] Pinkse J. \emph{A consistent nonparametric test for serial independence}, Journal of Econometrics, 1998, 84(2): 205-231.
\item[{[39]}] Grassberger P, Procacia I. \emph{Measuring the strangeness of strange attractors}, Physica D: Nonlinear Phenomena, 1983, 9(1-2): 189-208.
\item[{[40]}] Racine J S, Maasoumi E. \emph{ A versatile and robust metric entropy test of time-reversibility and other hypotheses}, Journal of Econometrics, 2007, 138(2): 547-567.
\item[{[41]}] Skaug H J, Tj{\o}stheim D. \emph{A nonparametric test of serial independence based on the empirical distribution function}, Biometrika, 1993, 80(3): 591-602. 

\item[{[42]}]Tsay R S. \emph{Analysis of financial time series[M].} John wiley \& sons, 2005.
\end{itemize}\vskip 10mm

\noindent School of Data Science, Zhejiang University of Finance and Economics, Hangzhou 310018,
China.\\
\indent Email: luowy042@zufe.edu.cn

\end{document}